\documentclass[11pt]{amsart}

\usepackage{amsmath}
\usepackage{amssymb}
\usepackage{mathrsfs}
\usepackage{comment}
\usepackage{color}
\usepackage[colorlinks,citecolor=blue,urlcolor=black,linkcolor=black]{hyperref}

\newtheorem{theorem}{Theorem}[section]
\newtheorem*{theorem*}{Theorem}
\newtheorem{claim}[theorem]{Claim}

\newtheorem{sclaim}[theorem]{Subclaim}
\newtheorem{lemma}[theorem]{Lemma}

\newtheorem{proposition}[theorem]{Proposition}
\newtheorem{corollary}[theorem]{Corollary}

\theoremstyle{definition}
\newtheorem{definition}[theorem]{Definition}

\newtheorem{question}[theorem]{Question}
\newtheorem*{question*}{Question}
\theoremstyle{remark}
\newtheorem{remark}[theorem]{Remark}

\newtheorem{notation}[theorem]{Notation}

\newcommand{\lusim}[1]{\smash{\underset{\raisebox{1.2pt}[0cm][0cm]{$\sim$}}
{{#1}}}}
\def\l{{\langle}}
\def\r{{\rangle}}

\def\br{\blacktriangleright}
\def\forces{\Vdash}

\newcommand{\s}{\subseteq}
\newcommand\diagonal{\bigtriangleup}
\newcommand{\Gal}[2]{\mathrm{Gal}(\mathscr{D}_{#1},{#1},{#2})}
\newcommand\z[1]{\mathring{#1}}
\newcommand{\dom}{\mathrm{dom}}
\newcommand{\len}{\mathrm{len}}
\newcommand{\cf}{{\rm cf}}

\newcommand{\crit}{\mathrm{crit}}
\newcommand\cat[1]{{}^\curvearrowright\langle #1\rangle}
 \newcommand{\one}{\mathop{1\hskip-3pt {\rm l}}}

\newcount\skewfactor
\def\mathunderaccent#1#2 {\let\theaccent#1\skewfactor#2
\mathpalette\putaccentunder}
\def\putaccentunder#1#2{\oalign{$#1#2$\crcr\hidewidth
\vbox to.2ex{\hbox{$#1\skew\skewfactor\theaccent{}$}\vss}\hidewidth}}

\def\smallbox#1{\leavevmode\thinspace\hbox{\vrule\vtop{\vbox
   {\hrule\kern1pt\hbox{\vphantom{\tt/}\thinspace{\tt#1}\thinspace}}
   \kern1pt\hrule}\vrule}\thinspace}

\newcommand{\Add}{\mathrm{Add}}
\newcommand{\Col}{\mathrm{Col}}

\DeclareMathOperator{\Lim}{acc}
\def\qedref#1{$\qed_{\reforiginal{#1}}$}

\newcommand\ale[1]{\marginpar{Alejandro: #1}}
\newcommand\tom[1]{\marginpar{Tom: #1}}

\title[Negating the Galvin property]{Negating the Galvin property}

\author{Tom Benhamou}
\address{School of Mathematical Sciences, Raymond and Beverly Sackler Faculty of Exact Science, Tel-Aviv University, Ramat Aviv 69978, Israel}
\email{tombenhamou@tauex.tau.ac.il}

\author{Shimon Garti}
\address{Einstein Institute of Mathematics,
 The Hebrew University of Jerusalem,
 Jerusalem 91904, Israel}
\email{shimon.garty@mail.huji.ac.il}

\author{Alejandro Poveda}
\address{Einstein Institute of Mathematics,
 The Hebrew University of Jerusalem,
 Jerusalem 91904, Israel}
\email{alejandro.poveda@mail.huji.ac.il}


\begin{document}
\let\labeloriginal\label
\let\reforiginal\ref
\def\ref#1{\reforiginal{#1}}
\def\label#1{\labeloriginal{#1}}

\begin{abstract}
We prove that  Galvin's property consistently fails at successors of strong limit singular cardinals. We also prove the consistency of this property failing at every successor of a singular cardinal. In \linebreak addition, the paper 
analyzes the effect of Prikry-type forcings on the  strong failure of the Galvin property, and explores stronger forms of this property in the context of large cardinals.
\end{abstract}

\maketitle

\tableofcontents

\newpage

\section{Introduction}
A classical theorem by Galvin establishes that if $\kappa^{<\kappa}=\kappa$ then \emph{Galvin's property holds at $\kappa$}; to wit, every family $\mathcal{C}$ 
consisting of clubs at $\kappa$
 and of size $\kappa^+$ contains a subfamily $\mathcal{D}\s \mathcal{C}$ of size $\kappa$ 
whose intersection yields a club at $\kappa$. This property originally appeared  
in a paper by Baumgartner, Hajnal and Maté \cite{MR0369081} where the authors attempt to answer 
the following problem attributed to Fodor: Given a family $\langle S_\alpha\mid \alpha<\kappa\rangle $ of stationary sets on $\kappa$, are there pairwise disjoint stationary sets $\langle T_\alpha\mid \alpha<\kappa\rangle$ such that $T_\alpha\s S_\alpha$ for all $\alpha<\kappa$? The former is quite a natural question in light of Solovay's decomposition theorem of stationary sets \cite{Sol}.
In \cite[\S3]{MR0369081} it is proved that the $\mathsf{CH}$ together with a version of the failure of Galvin's property for stationary sets yield a positive answer to Fodor's problem when $\kappa=\aleph_1$.  
On what this paper is concerned we are interested on Galvin's property at the level of successors of singular cardinals. In recent times this property has experienced a renewed interest after finding some deep connections with the structure of Prikry-type  generic extensions \cite{GitDensity, TomMotiII}. \linebreak In a different direction,  some interesting combinatorial implications of this principle have been discovered in the area of polarized relations. These ideas were extensively exploited by the first two authors and Shelah in \cite{BenGarShe}.

\smallskip

 Observe that Galvin's property holds at $\kappa$ whenever this latter  is a strongly inaccessible cardinal. Likewise, 
  the property holds at $\kappa^+$ provided $2^\kappa=\kappa^+$. 
  A natural query is how necessary is the arithmetical assumption $\kappa^{<\kappa}=\kappa$ in Galvin's theorem. This restriction turns to be  essential, as evidenced by Abraham and Shelah in \cite{MR830084}. More explicitly: 
\begin{theorem*}[Abraham-Shelah]
\label{thmabsh} Assume the $\mathsf{GCH}$ holds. Suppose that $\kappa<\lambda$ are infinite cardinals with 
$\kappa$  regular. 
Then, there is a forcing extension of the set-theoretic universe containing a family $\mathcal{C}$ of clubs at $\kappa^+$ with  $|\mathcal{C}|=\lambda$, that witnesses the following property:
\begin{center}
    For every subfamily $\mathcal{D}\s \mathcal{C}$ with $|\mathcal{D}|=\kappa^+$, $|\bigcap \mathcal{D}|<\kappa$.
\end{center}
Moreover,  $2^\kappa=2^{\kappa^+}=\lambda$ holds in this model provided $\lambda\geq\cf(\lambda)>\kappa^+$.
\end{theorem*}
Some comments are in order regarding the previous theorem. First, note that the above generic extensions  exhibit a rather strong failure of Galvin's property. For instance, in case $\kappa=\aleph_0$, the theorem yields a model of \textsf{ZFC} with a family $\mathcal{C}=\{C_\alpha\mid \alpha<\omega_2\}$
of clubs at $\omega_1$ such that every uncountable subfamily $\mathcal{D}\s \mathcal{C}$ has (even) finite intersection. The second noteworthy aspect is that the regularity assumption on $\kappa$ is essential. Indeed,  if $\kappa$ was singular 
then the construction pursued in \cite[\S1]{MR830084} would be void, for the relevant forcing   collapses $\kappa$.
In light of this one may  wonder whether the negation of  Galvin's property is possible at successors of singular cardinals:
\begin{question*}
\label{qstronglimit} Is it consistent with $\mathsf{ZFC}$ that Galvin's property fails at the successor of a strong limit singular cardinal?
\end{question*}

In the models produced by the above theorem one has $2^\kappa=2^{\kappa^+}=\lambda$. 
Actually, the equality $2^\kappa=\lambda$ is essential as proved in \cite{MR3604115}.
In effect, it is shown there that $2^\kappa<2^{\kappa^+}=\lambda$ entails Galvin's property at every collection of $\lambda$-many clubs at $\kappa^+$.
Hence any attempt to prove the consistency of the failure Galvin's property at successors of strong limit singular cardinals must rely on the violation of the \emph{Singular Cardinal Hypothesis} (SCH).
As a result,  a positive answer to the previous problem would require the existence of (very) large cardinals in the set-theoretic universe (see e.g., \cite{Gitikstrength,Mit}).

\smallskip

In Section~\ref{SectionLocalAndGlobal} of this paper we  would like to give a positive answer to the above question. The idea is simple when the relevant cardinal is rather large: Given
a supercompact cardinal $\kappa$ we shall preliminary make it indes\-tructible under $\kappa$-directed-closed forcings (see \cite{MR0472529}) and later force with the main poset of \cite[\S1]{MR830084}. This will produce a model of \textsf{ZFC} where $\kappa$ remains supercompact and Galvin's property fails at $\kappa^+$. Next, over the resulting model, we shall force with some Prikry-type forcing notion (e.g., Prikry forcing) towards singularizing $\kappa$. The point here is that, provided this Prikry-type poset is sufficiently well-behaved --to wit, it exhibits  a nice chain condition-- the original failure of Galvin's property at $\kappa^+$ will be preserved. Altogether, this yields a positive answer to the original question, provided, of course, the set-theoretic universe accommodates supercompact cardinals.

Nonetheless, the situation is less evident when the question is regarded for more \emph{down to earth} cardinals, such as $\aleph_{\omega+1}$ or $\aleph_{\omega_1+1}$. The usual technique to make a large cardinal $\kappa$ become $\aleph_\omega$ --without dramatically disturbing  the structure of the universe-- is to interleave \emph{collapses} in a Prikry-type forcing (see e.g., \cite{MagSing}).  However,  unless this process is done with sufficient care the resulting poset will not enjoy an optimal chain condition. To remedy this one uses the so-called technique of \emph{guiding generics}. These enable us to choose the \emph{collapsing part} in such a way that it does not contribute to the size of  maximal antichains. For instance, if Prikry forcing is endowed with \emph{Lévy collapses} arising from a guiding generic the resulting forcing keeps the chain condition of the original one; namely, $\kappa^+$-cc \cite[Example~8.6]{Cummings-handbook}.

A major caveat with the construction of guiding generics is that it heavily relies upon the arithmetical assumptions on $2^\kappa$. For instance, in \cite[Lemma~8.5]{Cummings-handbook} it is shown that if $\kappa$ is measurable and $2^\kappa=\kappa^{+}$ then one can build a guiding generic. On the contrary, the situation is less evident and the construction quite more challenging whenever $2^\kappa>\kappa^+$. For concrete examples on this vein see \cite{CumGCH,GitSha}. Note that in our case the construction of guiding generics needs to be performed in a context where not only $2^\kappa$ is large, but also Galvin's property fails at $\kappa^+$. This will introduce additional complications that we  overcome by elaborating on ideas of \cite{GitSha}.

As long as guiding generics are available, a mild variation of the above argument will produce a model where Galvin's property fails at $\aleph_{\omega+1}$. Once this is accomplished one can pose the following more ambitious question:
\begin{question*}
\label{qglobal}
Is it consistent with $\mathsf{ZFC}$ that Galvin's property fails at the successor of every singular cardinal, simultaneously?
\end{question*}
Bearing on our construction of guiding generics we  provide a positive answer to this question. 
The model of \textsf{ZFC} that will exemplify this is a generic extension  by certain \emph{Radin forcing with interleaved  collapses}. The proof idea and associated challenges are explained at length at the beginning of Section~\ref{SectionGlobal}. The reader is referred there for further details. Besides, in \S\ref{SectionImpossibility} we argue that a \emph{Ultimate failure of Galvin's property} is impossible.

\smallskip


Consider, yet again, clubs at $\aleph_1$. By  \cite{MR3604115}, the first cardinal $\kappa$ such that every $\lambda$-sized family of clubs at $\aleph_1$ witnesses Galvin's property depends on the size of $2^{\aleph_0}$ (and not of $2^{\aleph_1}$).  This striking fact leads to the definition of  Galvin's number $\mathfrak{gp}$ in the spirit of cardinal characteristics of the continuum. This concept naturally generalizes to $\mathfrak{gp}_\kappa$, for every cardinal $\kappa\geq \aleph_0$. A central problem here concerns the possible cofinalities of $\mathfrak{gp}_\kappa$. In the case of $\mathfrak{gp}$ (i.e., $\mathfrak{gp}_{\aleph_0}$) it is unknown whether $\cf(\mathfrak{gp})=\omega$ is consistent \cite{ghhm}. On the contrary, if $\kappa>\cf(\kappa)$ it turns that one can do better (cf. \S\ref{SectionGalvinsNumber}).


\smallskip

An intriguing feature of the models we produce is that in none of them it is clear that the strong conclusion of Abraham-Shelah theorem hold. Namely, we do not know if for a singular cardinal $\kappa$ there is a family $\mathcal{C}$ of clubs at $\kappa^+$, $|\mathcal{C}|=\kappa^{++}$, such that the intersection of any $\kappa^+$-sized subfamily $\mathcal{D}\s \mathcal{C}$ is of cardinality ${<}\kappa$. In Section~\ref{SectionTheStrongFailure} we coin this the name of \emph{strong failure of Galvin's property at $\kappa^+$}. All in all, this raises the following problem:
\begin{question*}
Is it a theorem of $\mathsf{ZFC}$ that Galvin's property does not strongly fail at the successor of a strong limit singular cardinal $\kappa$?

More verbosely, does $\mathsf{ZFC}$ prove that every family  $\mathcal{C}$ of clubs at $\kappa^{+}$ with cardinality $\kappa^{++}$ admits a subfamily $\mathcal{D}\s \mathcal{C}$ with $|\mathcal{D}|=\kappa^+$ and $|\bigcap\mathcal{D}|\geq \kappa$?
 \end{question*}
 In Section~\ref{SectionTheStrongFailure} we prove
 that the strong failure of  Galvin's property is indeed stronger than the mere failure of  Galvin's property (see Corollary~\ref{CorStrongNegationSuper}). 
 Specifically, in Proposition~\ref{propmagidor} we show that suitable Magidor/Radin-like generic  extensions satisfy  the following stronger assertion: Every family $\mathcal{C}$ of clubs at $\kappa^{+}$ with $|\mathcal{C}|=\kappa^{++}$ admits a subfamily $\mathcal{D}\s \mathcal{C}$ with $|\mathcal{D}|=\kappa^{++}$ such that the Radin club   eventually enters $\bigcap \mathcal{D}$. In particular, the  set $\bigcap \mathcal{D}$ is of cardinality at least $\kappa$.  On the contrary,  Theorem~\ref{thmprikryweak} 
demonstrates that 
 the above described behaviour of the Radin club cannot be expected for a Prikry sequence. Note, however, that this does not yet discard an scenario where the strong failure holds in a generic extension by Prikry forcing.  A natural strategy to show this would be to begin with the model of \cite[\S1]{MR830084} 
 and argue that Prikry forcing preserves the corresponding witnessing family. Unfortunately, Theorem~\ref{thmprikrydestroys} indicates that this will not work in general, for there are  witnessing families that are always destroyed. The section is later concluded with the proof of Theorem~\ref{thmudiamond}, which characterizes --in terms of a local guessing principle-- what witnessing families for the strong failure 
 are destroyed. The \textsf{ZFC} status of the  strong failure of Galvin's proper\-ty at successors of singular cardinals remains unknown to the authors (cf. \S\ref{SectionOpenProblems}).
\smallskip

In Section~\ref{SectionStronger} we explore some strong variations of Galvin's pro\-perty at large cardinals. Here we introduce the principle  ${\rm Gal}(\mathscr{F},\partial,\lambda)$ that naturally extends Galvin's statement to more general classes of filters. Later, we prove the consistency of the failure of this principle in the context of large cardinals (Theorem~\ref{thmnegmeasurable}). We also prove that $\mathrm{Gal}(\mathscr{F}, \partial,\lambda)$ can be used to characterize the fact that every set $A$ in a Prikry extension with $|A|=\lambda$ contains a ground model set $B$ with $|B|=\partial$ (Theorem~\ref{propdensity}). This extends a previous result by Gitik \cite{GitDensity} on density of old sets in Prikry generic extensions.

\smallskip

The paper concludes with Section~\ref{SectionOpenProblems} by presenting some  open problems.



\smallskip

All the notations and terminologies used are either standard or will be timely introduced. Additionally, the reader may find in Section~\ref{Prelimminaries} all the relevant preliminaries. In particular, he/she will find  a self-complete presentation of the Radin forcing with interleaved collapses that we use in Section~\ref{SectionGlobal}. Further consult of standard textbooks in set theory such as \cite{Kunen, Jech2003} might be required, although the paper has been composed aiming to be as self-contained as possible. Through the text we adopt the \emph{Jerusalem forcing convention}, writing \emph{$p\leq q$} whenever \emph{$q$ is stronger than $p$}.

\section{Preliminaries}\label{Prelimminaries}

\subsection{Notations}
Our notation is pretty standard. For a set $X$, we denote its cardinality by $|X|$. Given a cardinal $\kappa$, $\cf(\kappa)$ stands for its cofinality.
If $\kappa=\cf(\kappa)<\lambda$ then we denote $S^\lambda_\kappa:=\{\delta<\lambda\mid \cf(\delta)=\kappa\}$.
For a set of ordinals $A$, $\Lim(A):=\{\alpha<\sup(A)\mid \sup(A\cap\alpha)=\alpha\}$.
For an ordinal $\alpha$ and a set $X$, an $\alpha$-sequence of elements of $X$ is a function $f:\alpha\rightarrow X$. We denote by $^{\alpha}{}X$ the set of all $\alpha$-sequences of elements of $X$. If $f(i)=x_i$, we will denote $f=\l x_i\mid i<\alpha\r$. The point-wise image of $X$ under $f$ is denoted by $f``X:=\{f(x)\mid x\in X\}$.
If $X$ is a set of ordinals, $[X]^\alpha$ denotes the collection of $\alpha$-sequences of elements of $X$ which are increasing. For a cardinal $\rho$, we identify $[X]^\rho$ with the set $\{Y\in P(X)\mid |Y|=\rho\}$. 

\subsection{Clubs and Filters} 
Let $\kappa$ be an ordinal with $\cf(\kappa)\geq \omega_1$. A set $C\s \kappa$ is  a \textit{club at $\kappa$} if it is closed and unbounded in the order topology of $\kappa$. A set $S\s\kappa$ is  \textit{stationary} if it intersects every club at $\kappa$. The \textit{club filter} on $\kappa$ is,
$$\mathscr{D}_{\kappa}:=\{A\subseteq\kappa\mid \exists C\subseteq A\,\ (C\text{ is a club at }\kappa)\}.$$ 
This paper is mostly concerned with the following property of the club filter:
\begin{definition}[Galvin's property]\label{Galvins}
 Let $\kappa>\omega$ be  regular  and $ \mu\leq \lambda\leq 2^\kappa$. Denote by ${\rm Gal}(\mathscr{D}_\kappa,\mu,\lambda)$ the statement:
$$\text{``For every $\mathcal{C}\s  \mathscr{D}_\kappa$ with $|\mathcal{C}|=\lambda$ there is $\mathcal{D}\in[\mathcal{C}]^\mu$ such that  $\cap\mathcal{D}\in \mathscr{D}_\kappa$.''}$$
We say that \emph{Galvin's property holds at $\kappa$} if $\mathrm{Gal}(\mathscr{D}_\kappa,\kappa,\kappa^+)$ holds.
\end{definition}
The notation ${\rm Gal}(\mathscr{D}_\kappa,\partial,\lambda)$ is borrowed, with some modifications, from 
\cite{matet}.
In \cite{MR0369081}, Baumgartner, Hajnal and Maté present a proof due to Galvin that $\kappa^{<\kappa}=\kappa$ entails $\mathrm{Gal}(\mathscr{D}_\kappa,\kappa,\kappa^+)$. In particular,  $\mathsf{CH}$ yields $\mathrm{Gal}(\mathscr{D}_{\aleph_1},\aleph_1,\aleph_2)$ and \emph{Gödel's Axiom of Constructibility} $V=L$ entails $\mathrm{Gal}(\mathcal{D}_\kappa,\kappa,\kappa^+)$ for all regular cardinals $\kappa>\omega$. It is worth noting that the proof presented in \cite[\S3.2]{MR0369081} applies to all normal filters $\mathscr{F}$ on $\kappa$.

\smallskip

In Section~\ref{SectionStronger}, we will consider a generalization of Galvin's property to more abstract filters. Recall that a \emph{filter} over $X$ is a set $\mathscr{F}\subseteq \mathcal{P}(X)\setminus\{\emptyset\}$ stable under finite intersection, 
upwards-closed with respect to inclusion  and $X\in \mathscr{F}$. It is customary to refer to the elements of $\mathscr{F}$ as \emph{$\mathscr{F}$-large} or \linebreak \emph{$\mathscr{F}$-measure one}.
Interesting filters, however, enjoy of additional 
properties: 
\begin{definition}[Combinatorial properties associated to filters]\label{CombinatorialProperties} \hfill
\begin{enumerate}
\item $\mathscr{F}$ is \textit{uniform} if for every $A\in \mathscr{F}$, $|A|=|X|$.
\item $\mathscr{F}$ is $\lambda$-\textit{complete} if for every $\theta<\lambda$ and $\langle A_{\alpha}\mid \alpha<\theta\rangle\in\mathscr{F}^\lambda$ 
 $$\bigcap_{\alpha<
\beta}A_\alpha\in \mathscr{F}.$$
\item $\mathscr{F}$ is \textit{principal} if there is $x\in X$ such that $\{x\}\in \mathscr{F}$. 
\item $\mathscr{F}$ is an \textit{ultrafilter} if $\forall A\in \mathcal{P}(X) ( A\in \mathscr{F}\vee X- A \in \mathscr{F})$.
\item A filter $\mathscr{F}$ over a cardinal $\kappa$ is \emph{normal} if for every $\langle A_\alpha\mid \alpha<\kappa\rangle \in \mathscr{F}^\kappa$, 
$$\diagonal_{\alpha<\kappa}A_\alpha:=\{\beta<\kappa\mid \text{$\beta\in A_\alpha$, provided $\beta>\alpha$}\}\in \mathscr
{F}.$$
\end{enumerate}
\end{definition}
It is well-known that $\mathscr{D}_\kappa$ is a uniform, normal and $\kappa$-complete filter on $\kappa$.

\smallskip 

   
\subsection{Forcing preliminaries}
Many of the results of this paper use the technique of \emph{Forcing} to establish the consistency of \textsf{ZFC} with certain combinatorial principles. We take advantage of this section to fix notation and provide the reader with a comprehensive exposition of the paper's main poset: the so-called \emph{Radin forcing with interleaved collapses.} Nonetheless, the reader is assumed to be familiar with the theory of set-theoretic forcing.  For general background on this technique see  \cite{Kunen,Jech2003}.

\begin{definition}[Cohen forcing and Lévy Collapse]
Let $\kappa<\lambda$ be cardinals with $\kappa=\cf(\kappa)$.
The \emph{Cohen forcing} for adding $\lambda$-many subsets of $\kappa$ is
$$\Add(\kappa,\lambda):=\{f\colon \kappa\times\lambda\rightarrow\kappa\mid |f|<\kappa\}.$$
The \emph{Lévy collapse} forcings to 
collapse $\lambda$ to $\kappa$ is 
\begin{eqnarray*}
\Col(\kappa,\lambda)&:=&\{f\colon \kappa\rightarrow\lambda\mid |f|<\kappa\}.
\end{eqnarray*}
The order for both forcing notions is inclusion; namely,  $f\leq g$ iff $f\subseteq g$.
\end{definition}

\begin{definition}[Basic properties]
\hfill

$\br$ $\mathrm{Add}(\kappa,\lambda)$ is a $\kappa$-directed-closed and $(2^{<\kappa})^+$-cc forcing;

$\br$ $\Col(\kappa,\lambda)$ is  $\kappa$-directed-closed and $\lambda^+$-cc, provided $\lambda^{<\kappa}=\lambda$. 
\end{definition}

A  poset that will deserve a special attention in this text is the following:
\begin{definition}[Abraham-Shelah {\cite{MR830084}}]\label{AbSh forcing}
Let $\kappa<\lambda$ be cardinals with $\kappa$ regular. 
The following clauses yield the definition of the poset $\mathbb{S}(\kappa,\lambda)$:
\begin{itemize}
    
    \item $\lusim{\mathbb{R}}$ is a $\Add(\kappa,1)$-name for the forcing which consists of closed bounded sets $C\subseteq\kappa^+$ which do not contain a subset of cardinality $\kappa$ from $V$.
    The order of $\lusim{\mathbb{R}}$ is forced to be end-extension.
    \item Denote $\mathbb{S}=\Add(\kappa,1)*\lusim{\mathbb{R}}$.
    \item $\mathbb{S}(\kappa,\lambda)$ is a product of $\lambda$-many copies of $\mathbb{S}$ with mixed support, ${<}\kappa$-support on the $\Add(\kappa,1)$-side and $\kappa$-support on the $\lusim{\mathbb{R}}$-side.
    \end{itemize}
    Whenever $\lambda=\kappa^{++}$ we shall  write $\mathbb{S}(\kappa)$ instead of $\mathbb{S}(\kappa,\kappa^{++})$.
\end{definition}
The poset $\mathbb{S}(\kappa,\lambda)$ is $\kappa$-directed closed. Also, assuming that $2^\kappa=\kappa^+$, $\mathbb{S}(\kappa,\lambda)$ is a $\kappa^{++}$-cc forcing that moreover preserves $\kappa^+$ 
(see \cite[\S1]{MR830084}).

\smallskip

The main feature of $\mathbb{S}(\kappa,\lambda)$  is that it yields a generic extension where $\mathrm{Gal}(\mathscr{D}_{\kappa^+},\kappa^+,\lambda)$ fails in a quite strong sense. For the reader's convenience we shall sketch the proof of this: 
For each $\alpha<\lambda$, denote $\mathbb{S}(\kappa,\lambda)\restriction\{\alpha\}:=\mathbb{S}_\alpha=\mathrm{Add}(\kappa,1)*\lusim{\mathbb{R}}_{\alpha}$ and let \label{ExplanationonGalvin}
 $C_\alpha$ be the club introduced by  $\mathbb{R}_\alpha$. The claim is that $\l C_\alpha\mid \alpha<\lambda\r$ witnesses the failure of $\mathrm{Gal}(\mathscr{D}_{\kappa^+},\kappa^+,\lambda)$. In effect, let us argue that every $C\in[\kappa^+]^\kappa$ is included in at most $\kappa$ many $C_\alpha$'s. First, as a consequence of \cite[Lemma~1.7]{MR830084},  $C$ is introduced by $\mathbb{Q}(\kappa,\lambda)$, where $\mathbb{Q}(\kappa,\lambda)$ is the product of $\lambda$-many copies of  $\mathrm{Add}(\kappa,1)$ with ${<}\kappa$-support. Note that this is a $\kappa^+$-cc subforcing  of $\mathbb{S}(\kappa,\lambda)$ raised from the natural projection. Since $|C|=\kappa$ and every element of $C$ is decided by an antichain of size at most $\kappa$ it follows that $C\in V[G_{\mathbb{Q}(\kappa,\lambda)}\upharpoonright A]$, for some $A\in[\lambda]^\kappa$. Let $\alpha\in \lambda- A$. Since $\mathbb{S}(\kappa,\lambda)\cong \mathbb{S}(\kappa,\lambda)\upharpoonright(\lambda- \{\alpha\})\times\mathbb{S}_\alpha$ one can apply \cite[Lemma 1.13]{MR830084} and infer that $\neg(C\subseteq C_\alpha)$.
Hence $C$ is contained in at most $\kappa$-many $C_\alpha$'s.

\smallskip

The main results of this paper concern singular cardinals. The most basic forcing notion to singularize a regular cardinal (while not collapsing it) is the so-called \emph{Prikry forcing}: 
\begin{definition}[Prikry forcing]\label{Prikry forcing}
Let $\mathscr{U}$ be a normal ultrafilter over $\kappa$. The \textit{Prikry forcing} with $\mathscr{U}$ is  denoted by $\mathbb{P}(\mathscr{U})$ and
consist of all conditions of the form $\l \alpha_1,\dots,\alpha_n,A\r$ where $\alpha_1<\dots<\alpha_n<\min(A)$ and $A\in \mathscr{U}$. 

The order is defined by $\l \alpha_1,\dots,\alpha_n,A\r\leq \l\beta_1,\dots,\beta_m,B\r$ iff:
\begin{enumerate}
    \item $n\leq m$ and for every $i\leq n$, $\alpha_i=\beta_i$;
    \item $\langle \beta_{n+1},\dots,\beta_m\rangle \in [A]^{<\omega}$;
    \item $B\subseteq A$.
\end{enumerate}
For $p=\langle \alpha_1,\dots, \alpha_n,A\rangle\in\mathbb{P}(\mathscr{U})$ it is customary to refer to  $\langle \alpha_1,\dots,\alpha_n\rangle$ as \emph{the stem of $p$.}
\end{definition}
For more information about Prikry forcing and its properties we refer the reader to \cite{MR2768695, TomTreePrikry}. Other relevant forcings for  Section~\ref{SectionTheStrongFailure} are Magidor and Radin forcing, respectively introduced in  \cite{MR0465868} and \cite{Rad}. In the forthcoming section we present a more sophisticated version of these forcings. So, instead of providing more details, we refer the reader to  \cite{MR2768695}.

\subsection{Radin forcing with interleaved collapses}\label{SubsectionRadinwithcollapses}
In this section we provide the reader with  a self-contained exposition of the main forcing notion of Section~\ref{SectionGlobal}: namely, \emph{the Radin forcing with interleaved collapses}.   This construction has already appeared in previous works \cite{CumGCH, YairEskew, Golshani2016WeakDC} under slightly different formulations. Our presentation is mostly inspired by Cummings' work, and we shall particularly follow \cite[\S3]{CumGCH}.  The reader familiar with this work will notice that our construction is based on a weakening of Cumming's notion of \emph{constructing pair} (see Definition \ref{WeakConstruct}). As we will demonstrate, this milder notion turns to be  enough for our purposes. 

\subsubsection{Measure sequences and weak constructing pairs}
We commence by introducing  the notion of a \emph{weak constructing pair}.  Later we will use it to derive a suitable generalization of the notion of a measure sequence \cite{Rad}.

\begin{definition}[Weak constructing pair]\label{WeakConstruct}
A pair $(j,F)$ is called a \emph{weak constructing pair} if it satisfies the following properties:
\smallskip
\begin{enumerate}
    \item $j\colon V\rightarrow{M}$ is an elementary embedding into a transitive inner model with $\crit(j)=\kappa$ and ${}^\kappa M\s M$;
    \item $F$ is an $N$-generic filter for  $\Col(\kappa^{+4},i(\kappa))^N$, where  the map\linebreak  $i\colon V\rightarrow{N}\simeq \mathrm{Ult}(V,\mathscr{U})$ is the ultrapower embedding derived from $j$;
    \item $F\in M$. 
\end{enumerate}
\end{definition}
\begin{remark}
Note that $\mathscr{U}$ is definable from $F$ and $^{\kappa}{}M\s M$,  hence $$N_{i(\kappa)+1}^M=N_{i(\kappa)+1}\text{ and }\Col(\kappa^{+4},i(\kappa))^N\in M.$$ 
It is customary to refer to  $F$ as the \emph{guiding generic}. This terminology is motivated by the fact that  $F$ will be the responsible of \emph{guiding} the forcing  in the process of collapsing the intervals between points in the generic club.
\end{remark}


\begin{definition}
If $(j,F)$ is a weak constructing pair then put
$$
C^*:=\{f\colon\kappa\rightarrow V_\kappa\mid \dom(f)\in \mathscr{U}\,\wedge\, \forall\alpha\in\dom(f)\,(f(\alpha)\in \Col(\alpha^{+4},\kappa))\},
$$
$$
F^*:=\{f\in C^*\mid i(f)(\kappa)\in F\}.
$$
\end{definition}

\begin{definition}\label{MeasureSequenceDerived} 
Let $(j,F)$ be a weak constructing pair. 
A sequence $u$ is said to be \emph{inferred from $(j,F)$} iff
\begin{enumerate}
    \item $u\in M$;
    \item $u(0):=\crit(j)$;
    \item $u(1):=F^*$;
    \item $u(\alpha):=\{X\s V_{u(0)}\mid u\restriction\alpha\in j(X)\}$, for $\alpha\in [2,\mathrm{len}(u))$;
    \item $M\models |\mathrm{len}(u)|\leq u(0)^{++}$.
\end{enumerate}
We say that \emph{$(j,F)$ constructs $u$} if $u=w\restriction \alpha$, where  $w$ is the sequence inferred from $(j,F)$ and $\alpha\in [1,\len(w))$.\footnote{Note that if $w$ is inferred from $(j,F)$ then $\len(w)\geq 2$.}
\end{definition}

\begin{definition}
If $u$ is constructed by some pair $(j,F)$ then $\kappa_u:=u(0)$. Also, if $\mathrm{len}(u)\geq 2$ define 
\begin{enumerate}
    \item $F^*_u:=u(1)$;
    \item $\mu_u:=\{X\s \kappa_u\mid \exists f\in F^*_u\;\dom(f)=X\}$;
    \item $\bar{\mu}_u:=\{X\s V_{\kappa_u}\mid \{\alpha\mid \langle \alpha\rangle \in X\}\in \mu_u\},$
    \item $F_u:=\{i(f)(\kappa_u)\mid f\in F^*_u\}$,
    \item $\mathscr{F}_u:=\bar{\mu}_u\cap \bigcap\{u(\alpha)\mid \alpha\in [2,\mathrm{len}(u))\}$.\footnote{If $\len(u)=2$ we shall agree that  $\mathscr{F}_u:=\bar{\mu}_u$.}
\end{enumerate}
Otherwise, if  $\len(u)=1$,  we put $F^*_u=F_u:=\{\emptyset\}$ and $\mathscr{F}_u:=\{\emptyset\}$.
\end{definition}
At this point some explanations are in order. The set $F^*_u$ is the \emph{pull-back} of the generic filter $F$ and  $\bar{\mu}_u$ is, in essence,  the normal measure on $\kappa_u$ derived from $j$. Similarly, for $\alpha\geq 2$ the set $u(\alpha)$ yields a non-principal $\kappa_u$-complete ultrafilter on $V_{\kappa_u}$ that concentrates on sets  \emph{resembling}  $u\restriction\alpha$.  For instance, a typical element of $u(2)$  is  of the form 
$$\{\langle \alpha, G^*\rangle\mid \exists \mathscr{V} (\text{$\mathscr{V}$ normal measure on $\alpha$}\,\wedge\, \text{$G^*$ is a *-$\mathscr{V}$-generic})\},$$
where $G^*$ is  $*$-$\mathscr{V}$-generic  if there is a $\mathrm{Ult}(V,\mathscr{V})$-generic filter $G$ for the poset $\Col(\alpha^{+4},i_\mathscr{V}(\kappa))$ such that $G^*:=\{f\mid \dom(f)\in \mathscr{V}\,\wedge\,  i_\mathscr{V}(f)(\alpha)\in G\}$.
\begin{definition}
The family of \emph{measure sequences} $\mathcal{U}_{\infty}$ is defined as follows:
\begin{itemize}
    \item $\mathcal{U}_0:=\{u\mid \text{$\exists (j,F)$($(j,F)$   constructs $u$)}\}$;
    \item $\mathcal{U}_{n+1}:=\{u\in \mathcal{U}_n\mid \mathcal{U}_n\cap V_{\kappa_u}\in \mathscr{F}_u\};$
    \item $\mathcal{U}_\infty:=\bigcap_{n<\omega}\mathcal{U}_n.$
\end{itemize}
\end{definition}
The advantage of the set $\mathcal{U}_\infty$ is that every measure $u(\alpha)$ in a  sequence $u\in \mathcal{U}_{\infty}$ concentrates on measure sequences. Clearly, if $u\in \mathcal{U}_\infty$ then also $u\restriction\alpha\in\mathcal{U}_\infty$, for all $\alpha<\len(u)$.
\begin{definition}
Let $u\in \mathcal{U}_\infty$. For each $A\in\mathscr{F}_u$ denote $\z{A}:=\{\kappa_w\mid w\in A\}$.
\end{definition}

\begin{definition}\label{DiagonalIntersection}
Given $u\in \mathcal{U}_\infty$,  $A\in \mathscr{F}_u$ and $\{A_v\mid v\in A\}$, define 
$$\diagonal_{v} A_v:=\{w\in \mathcal{U}_\infty\mid v\in V_{\kappa_w}\,\Longrightarrow w\in A_v\}.$$
\end{definition}

It is routine to check that under the above conditions $\z{A}$   (resp. $\diagonal_{v\in A} A_v$) is $\mu_u$-large (resp. $\mathscr{F}_u$-large).

\subsubsection{The main poset}
In this section we present our version of the Radin forcing with interleaved collapses and prove some of its main properties. The forcing is devised to introduce  a club on a large cardinal $\kappa$ and simultaneously collapse the intervals between the points of the said club.

\smallskip
Let us begin defining the basic modules of the main poset:
\begin{definition}\label{BasicModules}
Let $u\in\mathcal{U}_\infty$. Denote by $\mathbb{R}^*_u$ the poset whose conditions are $5$-tuples $(u, \lambda,h, A, H)$ such that:
\begin{enumerate}
    \item $\lambda<\kappa_u$;
    \item $h\in \Col(\lambda^{+4},\kappa_u);$
    \item $A\in \mathscr{F}_u$. Also, in case $\len(u)\geq 2$,  $\min\z{A}>\max(\sup\mathrm{ran}(h),\lambda)$;
    \item $H$ is a function with
    $\dom(H)=\z{A}$ and $H\in F^*_u$.
\end{enumerate}
For two condition $p=(u,\lambda^p, h^p,A^p,H^p)$ and $q=(u, \lambda^q, h^q, A^q, H^q)$ in $\mathbb{R}^*_u$ we write $q\leq p$ (\emph{$p$ is stronger than $q$}) if the following hold:
\begin{itemize}
    \item $\lambda^p=\lambda^q$;
    \item $h^q\s h^p$;
    \item $A^p\s A^q$;
    \item $H^q(\kappa_w)\s H^p(\kappa_w)$, for all $w\in A^p$.
\end{itemize}
Whenever the relevant condition is clear from the context  we shall tend to suppress the superscript  and simply write $(u,\lambda, h, A,H)$. 
\end{definition}
\begin{remark}
If $\len(u)= 1$ then $A=\emptyset$ and so $H=\emptyset\in F^*_u$.  Additionally, for each $p\in \mathbb{R}^*_u$, the  poset  $\mathbb{R}^*_u/p$ is isomorphic to $\Col({\lambda}^{+4},\kappa_u)/h$.
\end{remark}

\begin{definition}\label{RadinWithCol} 
Let $u\in \mathcal{U}_\infty$. Denote by $\mathbb{R}_u$ the poset whose conditions are finite sequences $p=\langle p_n\mid n\leq \ell(p)\rangle$ such that:
\begin{enumerate}
    \item $p_n\in \mathbb{R}^*_{w_n}$ for some $w_n\in \mathcal{U}_\infty;$
    \item  
    $\lambda_{n+1}=\kappa_{w_n}$, for $n<\ell(p)$;    \item $w_{\ell(p)}=u$. 
\end{enumerate}
Given $p,q\in \mathbb{R}_u$ we write $q\leq^* p$ (\emph{$p$ is a pure extension of $q$}) iff $$\text{$\ell(p)=\ell(q)$ and $q_n\leq p_n$ for all $n\leq \ell(p)$}.$$
If no confusion arises   we shall tend to write $p_n=(w_n,\lambda_n,h_n,A_n,H_n)$. 
\end{definition}
\begin{notation}
For a condition $p=\langle p_n\mid n\leq \ell(p)\rangle\in \mathbb{R}_u$ and $n\leq \ell(p)$ we  denote $p^{\leq n}:=\langle p_i\mid i\leq n\rangle$. Similarly, $p^{>n}:=\langle p_i\mid n+1\leq  i\leq \ell(p)\rangle$.
\end{notation}
\begin{remark}
Note that Definition~\ref{BasicModules}(3) together with Definition~\ref{RadinWithCol}(2) imply that $A^{p_n}\cap A^{p_{n+1}}=\emptyset$, for all $n<\ell(p)$.
\end{remark}

We now define what will be the prototypical extensions of a condition:
\begin{definition}\label{Onepoint}
Let $u\in\mathcal{U}_\infty$ and $p=(u, \lambda, h, A, H)\in \mathbb{R}^*_u$. 

For each $v\in A$, put
$$p\cat v:=(v, \lambda,h, A_{v,-}, H_{v,-})^\smallfrown(u, \kappa_v, H(\kappa_v), A_{v,+}, H_{v,+}),$$
where $\langle A_{v,-}, H_{v,-}\rangle$ and $\langle A_{v,+}, H_{v,+}\rangle$ are defined as follows:

$\br$ $A_{v,-}:=A\cap V_\nu$ and $H_{v,-}:=H\restriction \z{A}_{v,-};$

$\br$ $A_{v,+}:= \{w\in A\mid \kappa_w> \sup(\mathrm{ran}(H(\kappa_v))\cup\{\kappa_v\})\}$ and $H_{v,+}:=H\restriction \z{A}_{v,+}.$ 

\smallskip

Similarly, for each $p=\langle p_n\mid n\leq \ell(p)\rangle \in\mathbb{R}_u$ and $v\in A_n$ define $$p\cat v:=\langle q_k\mid k\leq \ell(p)+1\rangle,$$
where $q_n{}^\smallfrown q_{n+1}=p_n\cat v$, and $q_k=p_k$ when $k\notin\{n,n+1\}$.

For $\vec{v}=\langle v_0,\dots,v_k\rangle\in \prod_{i=0}^k A_{n_i}$ define $p{}^\curvearrowright{\vec{v}}$ recursively as $(p{}^\curvearrowright{\vec{v}\restriction k})\cat {v_k}.$\footnote{By convention, $p{}^\curvearrowright\emptyset:=p.$}
\end{definition}
Note that the operation ${}^\curvearrowright{\vec{v}}$ is commutative with respect to the order of the measure sequences appearing in $\vec{v}$. The previous notion enables us to define the main ordering on $\mathbb{R}_u$:
\begin{definition}[The main poset]\label{MainPoset}
Let $u\in\mathcal{U}_\infty$. Given two conditions $p,q\in \mathbb{R}_u$ we write $q\leq p$ (\emph{$p$ is stronger than $q$}) iff there is a sequence $\vec{v}\in \prod_{i=0}^{|\vec{v}|} A_i$ such that $ q^\curvearrowright\vec{v}\leq^* p$.
\end{definition}
\begin{remark}\label{RemarkOnlength}
If $\len(u)=1$ and $p\in\mathbb{R}^*_u$ then $\mathbb{R}_u/p\simeq\Col(\lambda^{+4},\kappa_u)/h$.
\end{remark}
One may be wondering whether the operation ${}^\curvearrowright{\vec{v}}$ of Definition~\ref{Onepoint} is always well-defined. While this is not the case in general there is yet a $\leq^*$-dense subposet of $\mathbb{R}_u$ that satisfies this property. 
\begin{definition}
Let $u\in\mathcal{U}_\infty$. Given $p\in\mathbb{R}_u$ and $\vec{v}\in \mathcal{U}_\infty^{<\omega}$ we say that \emph{$\vec{v}$ is addable to $p$} iff $p{}^\curvearrowright \vec{v}\in \mathbb{R}_u$. A condition $p\in \mathbb{R}_u$ is called \emph{pruned} iff every $\vec{v}\in \prod_{i=0}^k A_{n_i}$ is addable to $p$.
\end{definition} 
\begin{remark}
Note that $v\in A_n$ is addable to $p$ iff:
\begin{enumerate}
    \item $A_n\cap V_{\kappa_v}\in \mathscr{F}_{v}$.
    \item $H_n\restriction (\z{A}_n)_{v,-}\in F^*_{v}$.
\end{enumerate}
Also, $p$ is pruned iff for every $i\leq m^p$ and $v\in A^{p_i}$, $v$ is addable to $p$.
\end{remark}
\begin{lemma}
Let $u\in \mathcal{U}_{\infty}$. For every $p\in \mathbb{R}_u$ there is $p^*\in \mathbb{R}_u$ with $p\leq^* p^*$ such that $p^*$ is pruned (i.e., for $n\leq \ell(p)$ and $v\in A^{p^*_n}$, $v$ is addable to $p^*$.)
\end{lemma}
\begin{proof}
By the previous remark, it suffices to show that $(1)$ and $(2)$ hold for every $v\in A_n$.  Since $w_n\in \mathcal{U}_{\infty}$, there is $(j,F^*_{w_n})$ a weak constructing pair for $w_n$. By Definition~\ref{BasicModules}, $H_n\in F^*_{w_n}$, hence
$j(H_n)\restriction\kappa_{w_n}=H_n\in F^*_{w_n}=F^*_{w_n\restriction \alpha}$, for all $\alpha\in[2,\len(w_n))$. Also, $j(A_n)\cap V_{\kappa_{w_n}}=A_n\in \mathscr{F}_{w_n}\subseteq \mathscr{F}_{w_n\restriction \alpha}$, for every $\alpha\in[2, \len(w_n))$. Therefore,  setting $$A^*_n:=\{v\in V_{\kappa_{w_n}}\mid H_n\restriction \kappa_{v}\in F^*_{v}\wedge A_n\cap V_{\kappa_{v}}\in \mathscr{F}_{v}\},$$
we have
$w_n\restriction\alpha\in j(A^*_n),$
for all $\alpha\in[2, \len(w_n))$. Similarly, $A^*_n\in \bar{\mu}_{w_n}$.  Altogether, $A^*_n\in\mathscr{F}_{w_n}$.
Now, let $p^*$ be the $\leq^*$-extension of $p$ where  $A_n$ is shrunken to $A^*_n$, for every $n\leq m^p$. Clearly, $p^*$ is as wanted.
\end{proof}

By virtue of the previous lemma we are entitled to assume that  all the conditions we work with are pruned.

\smallskip

The next factoring lemma will be useful in the cardinal-structure analysis of the forthcoming section:
\begin{lemma}\label{Factoring}
Let $u\in\mathcal{U}_\infty$ with $\len(u)\geq 2$. For each $p\in\mathbb{R}_u$ and  $n< \ell(p)$,
\begin{enumerate}
    \item $\mathbb{R}_u/p \simeq \mathbb{R}_{w_n}/p^{\leq n}\times \mathbb{R}_{u}/p^{>n}$.
    \end{enumerate}
    Also, if $n+1<\ell(p)$ and $\len(w_{n+1})=1$ then
    \begin{enumerate}
    \setcounter{enumi}{1}
    \item $\mathbb{R}_u/p \simeq \mathbb{R}_{w_n}/p^{\leq n}\times \Col(\kappa_{w_n}^{+4},\kappa_{w_{n+1}})/h\times  \mathbb{R}_{u}/p^{>n+1}$.
\end{enumerate}
\end{lemma}
\begin{proof}
Clause~(1) is trivial. For (2), using Remark~\ref{RemarkOnlength} we have
$$\mathbb{R}_u/p^{>n}\simeq \mathbb{R}_{w_{n+1}}/p_n\times \mathbb{R}_u/p^{>n+1}\simeq \Col(\kappa_{w_n}^{+4},\kappa_{w_{n+1}})/h\times \mathbb{R}_u/p^{>n+1}.$$
Now, by (1) the result follows. 
\end{proof}

\subsubsection{Properties of the main forcing}

In this section we prove that the modified Radin forcing  of  Definition~\ref{MainPoset} behaves regularly and obtain the corresponding cardinal structure in the generic extension. For the rest of section $u$ will denote a fixed measure sequence with $\len(u)\geq 2$.

\smallskip

The next lemma implies that $\mathbb{R}_u$ does not collapse cardinals $\geq \kappa_u^+$:

\begin{lemma}\label{ccness}
For each $u\in\mathcal{U}_\infty$, there is a map $c_u\colon \mathbb{R}_u\rightarrow V_{\kappa_u}$ such that 
$$c_u(p)=c_u(q)\Longrightarrow\;\text{$p$ and $q$ are $\leq^*$-compatible}.$$

In particular, the poset $\mathbb{R}_u$ is $\kappa^+_u$-linked.
\end{lemma}
\begin{proof}
Fix $u\in\mathcal{U}_\infty$ and define $c_u(p):=\langle p_n\mid n<\ell(p)\rangle^\smallfrown \langle \lambda_{\ell(p)}, h_{\ell(p)}\rangle.$ 

Assume $c_u(p)=c_u(q)$ and write $\ell$, $\lambda$ and $h$ for the corresponding common values.\footnote{Note that from $c_u(p)=c_u(q)$  it follows that $\ell(p)=\ell(q)$.} Since $H^{p_\ell}, H^{q_\ell}\in F^*_u$ then
$$A:=\{\alpha<\kappa_u\mid \text{$H^{p_\ell}(\alpha)\cup H^{q_\ell}(\alpha)$ is a function}\}\in \mu_u.$$
Now, let $B^*:=A\cap \z{A}^{p_\ell}\cap \z{A}^{q_\ell}$ and 
$B:=\{w\in A^{p_\ell}\cap A^{q_\ell}\mid \kappa_w\in B\}.$ 

Clearly, $B\in\mathscr{F}_u$ and $\z{B}\s B^*$. Define $r:=\langle r_m\mid m\leq \ell\rangle$ where
$$r_m:=\begin{cases}
p_n, & \text{if $m<\ell$;}\\
(u, \lambda, h,B, (H^{p_\ell}\cup H^{q_\ell})\restriction \z{B}), & \text{otherwise.}
\end{cases}
$$
Obviously, $p,q\leq^* r$, as desired.
\end{proof}
Hereafter $G\s \mathbb{R}_u$ will stand for a $V$-generic filter. Moreover, for technical reasons we shall assume  that if $p\in G$ then $\lambda_0:=\omega$.\label{ChoiceofG}  

\begin{definition}\label{RadinObject}
The \textit{Radin sequence induced by $G$} is $\vec{u}$, the increasing enumeration of the set $R:=\{v\mid \exists p\in G\,\exists n<\ell(p)\,(v=w^{p_n})\}$.

Denote $C:=\{\kappa_{v}\mid v\in R\}.$\footnote{Although $R$ and $C$ formally depend on $G$ we have suppressed the dependence on $G$ to enlighten the notation.} 
\end{definition}

\begin{lemma}\label{Prelimminaryforclub}
Let $p\in G$ and some $v$ with $\len(v)\geq 2$ appearing in $p$. Then, 
\begin{enumerate}
\item $A\in \mathscr{F}_{v}$ iff there is $\delta<\kappa_v$ such that  $\{w\in R\mid \kappa_w\in (\delta,\kappa_v)\}\s {A}$.
    \item $\sup(\z{A}\cap C)=\kappa_{v}$ iff there is some $\alpha\in[2,\len(v))$ such that $A\in v(\alpha)$.
\end{enumerate}
\end{lemma}
\begin{proof}
(1): Suppose that $A\in\mathscr{F}_v$ and put $$D:=\{p\leq q\mid \exists i\leq \ell(q)\; (w_i=v\,\wedge\, A_i\s A)\}.$$
Clearly $D$ is dense below $p$, hence may pick $q\in G\cap D$. Put $\delta:=\kappa_{w_{i-1}}$ and note that $q\forces_{\mathbb{R}_u}\{w\in \lusim{R}\mid \kappa_w\in (\delta,\kappa_v)\}\s \check{A}_i$. Since $q\in G$ then $$\{w\in R\mid \kappa_w\in (\delta,\kappa_v)\}\s A_i\s A.$$

For the other direction, if $A\notin \mathscr{F}_{v}$, then either $V_{\kappa_v}\setminus A\in \bar{\mu}_v$ or  there is some $\alpha\geq 2$ such that $V_{\kappa_v}\setminus A\in v(\alpha)$. Without loss of generality assume that we fall in the second case. Then  $(V_{\kappa_v}\setminus A)\cap B\in v(\alpha)$, for every 
$B\in \mathscr{F}_{v}$. It thus follows that for each $\delta<\kappa_v$ the set
$$E_{\delta}=\{p\leq q\mid \exists i<\ell(q)\, (w_{i+1}=v\,\wedge\, w_i\notin A\,\wedge\, \kappa_{w_i}>\delta)\}$$
is dense. This yields the desired result.

(2): Note that $\sup(\z{A}\cap C)=\kappa_{v}$ iff  $\{w\in R\mid \kappa_w\in (\delta,\kappa_v)\}\cap A\neq \emptyset,$ for all $\delta<\kappa_v$. By (1) this latter is equivalent to $V_{\kappa_v}\setminus A\notin \mathscr{F}_v$, which it turns to be equivalent to the existence of some $\alpha<\len(v)$ such that $A\in v(\alpha).$
\end{proof}

\begin{corollary}\label{ClubIndeed}
$C$ is a club in $\kappa_u$.
\end{corollary}
\begin{proof}
To show that $C$ is unbounded, apply Lemma~\ref{Prelimminaryforclub}(2) to $p\in G$ and $A:=V_{\kappa_u}$, hence getting $\sup(C)=\kappa_u$.
To show that $C$ is closed, let $\delta\notin C$ and let us prove that $C\cap\delta$ is bounded. 
Indeed, let $p\in G$ with $p\Vdash_{\mathbb{R}_u} \delta\notin \lusim{C}$. Consider the set $$D=\{p\leq q\mid q\Vdash_{\mathbb{R}_u} \lusim{C}\cap\delta\text{ is bounded}\}$$ 
Let $p\leq p'$ be any condition. Since $p'\Vdash \delta\notin \lusim{C}$ it follows that $\delta\neq \kappa_v$, for all $v$ mentioned in $p'$. 
Let $n$ be the least index such that $\delta<\kappa_{n}$, and shrink $A_n$ to $A^{*}_n$ so that $\min(\z{A}^{*}_n)>\delta$. Then $p'\leq p^*$ and $p^*\Vdash_{\mathbb{R}_u} \sup(\lusim{C}\cap\delta)<\delta$, hence $p^*\in D$. Altogether, $D$ is dense above $p'$. By density $C\cap\delta$ is bounded.
\end{proof}

The next is the key lemma. We follow the proof from \cite[Lemma~5.8]{MR2768695} and more closely the argument from \cite{YairEskew}:

\begin{lemma}[Prikry lemma]\label{Prikrylemma} For each $v\in\mathcal{U}_\infty$ suppose that  $$\text{$2^{\kappa_v}\leq\kappa_{v}^{++}$ and  \,$2^{\kappa_v^{++}}=\kappa_v^{+3}$}.$$ 

Then, the triple $\langle \mathbb{R}_u, \leq, \leq^*\rangle$ satisfies the Prikry property: namely, for every sentence $\sigma$ in the language of forcing and for every $p\in\mathbb{R}_u$ there is $p\leq^*p^*$ such that $p^*\parallel\sigma$.
\end{lemma}
\begin{proof}
Let $p\in\mathbb{R}_u$  and $\sigma$ be a statement in the language of forcing. We shall argue that there is $p\leq^* p^*$ such that $p^*$ decides $\sigma$. To ease the notations, we shall assume that $p=(u, \omega, h_0, A_0,H_0)$ and will put $\kappa:=\kappa_u$.  
For expository purposes --in a slight abuse of notation--  we shall write $v(1)$ instead of $\bar{\mu}_v$, whenever $v\in \mathcal{U}_\infty$ (cf. Definition~\ref{MeasureSequenceDerived}).

\smallskip

For every $v\in A_0$, and every pair
$s^+=(s,h)\in [V_{\kappa_v}]^{<\omega}$ such that there are $\langle B,H\rangle, \langle A,F\rangle, g$ with $s\,^{\smallfrown}(v,\kappa_{w_{\ell(s)}},h, B,H)^{\smallfrown}(u, \kappa_{v}, g, A, F)\in \mathbb{R}_{u}$, define:
$$D^{0}_{s^+,v}:=\{g\geq H_0(\kappa_v)\mid \exists \langle B,H\rangle, \langle A,F\rangle \;(s^{\smallfrown}( v,\kappa_{w_{\ell(s)}},h,B,H)^{\smallfrown}( u,\kappa_v,g,A,F) \parallel\sigma)\},$$
$$D^{1}_{s^+,v}:=\{g\geq H_0(\kappa_{v})\mid \forall g'\geq g\,\forall \langle B,H\rangle, \langle A,F\rangle\;(s^{\smallfrown}( v,\kappa_{w_{\ell(s)}},h,B,H)^{\smallfrown}( u,\kappa_v,g',A,F\r \nparallel\sigma)\}.$$
Clearly, $D_{s^+,v}:=(D^{0}_{s^+,v}\cup D^{1}_{s^+,v})$ is dense open in $\Col(\kappa_v^{+4},\kappa)$. Since every measure sequence with critical point $\kappa_v$ is of length at most $\kappa_v^{++}$ and there are at most $2^{(2^{\kappa_v})}\leq \kappa_v^{+3}$ measures on $V_{\kappa_v}$, we conclude that there are at most $\kappa_v^{+3}$ measure sequences with critical point $\kappa_v$.\footnote{Here we use the extra hypothesis of the lemma.} Hence, $$D_{\kappa_v}:=\bigcap \{D_{s^+,w}\mid s^+\in [V_{\kappa_v}]^{<\omega}\,\wedge\,w\in\mathcal{U}_\infty\,\wedge\, \kappa_w=\kappa_v\}$$
is still dense open in $\Col(\kappa_v^{+4},\kappa)$. 

\smallskip

Consider $D^*=[\alpha\mapsto D_\alpha]_{u(1)}$. By elementarity, $D^*$ defines a dense open subset of $\Col(\kappa^{+4},i_{u(1)}(\kappa))^{N_{u(1)}}$. Since $[H_0]_{u(1)}\in F_u$ and $F_u$ is $N_{u(1)}$-generic for $\Col(\kappa^{+4},i_{u(1)}(\kappa))^{N_{u(1)}}$, there is $[H_0]_{u(1)}\leq[H^*]_{u(1)}\in F_u\cap D^*$. Shrink $A_0$ to $A^*:=\{v\in A_0\mid H_0(\kappa_v)\leq H^*(\kappa_v)\in D_{\kappa_v}\}$. Clearly, $A^*\in\mathscr{F}_u$. 

\smallskip

The condition $(u,\omega, h_0, A^*,H^*)$ enjoys the following \emph{diagonalization} pro\-perty: For every $v\in A^*$,  $s^+=(s,h)\in [V_{\kappa_v}]^{<\omega}$ and 
$g\geq H^*(\kappa_v)$, $$\exists \langle B,H\rangle, \langle A,F\rangle\;s^{\smallfrown}( v,\kappa_{w_{\ell({s})}},h,B,H)^{\smallfrown}( u,\kappa_v,g,A,F) \parallel\sigma$$ if and only if  $$(*)_{s^+,v} \ \ \exists\langle B,H\rangle, \langle A,F\rangle\;s^{\smallfrown}( v,\kappa_{w_{\ell(s)}},h,B,H)^{\smallfrown}( u,\kappa_v,H^*(\kappa_v),A,F) \parallel\sigma.$$
Indeed, the down-to-up implication is obvious. As for the opposite, since $$H^*(\kappa_v)\in D_{\kappa_v}\s D^0_{s^+, v}\cup D^1_{s^+,v},$$  $g\in D^0_{s^+,v}$ and $H^*(\kappa_v)\leq g$ it follows that $H^*(\kappa_v)\in D^0_{s^+,v}$.

Moreover, if  $\sigma$ is decided, this decision must be the same, for the above displayed conditions are compatible. 
Denote by $(*)^0_{s^+,v}$ (resp. by $(*)^1_{s^+,v}$) the fact that $\sigma$ (resp. $\neg\sigma$) is forced. 
For each $v\in A^*$ and $s^+\in [V_{\kappa_v}]^{<\omega}$ such that $(*)_{s^+,v}$ holds let us pick  witnessing pairs $\langle B_{s^+,v},H_{s^+,v}\rangle ,\langle A_{s^+,v},F_{s^+,v}\rangle$. We can take a lower bound for the functions $H_{s^+,v}$'s in $F_v$: Note that $$D:=\{[H_{s^+,v}]_{v(1)}\mid s^+\in [V_{\kappa_v}]^{<\omega}\}\s N_{v(1)}$$ is a directed subset of $\Col(\kappa^{+4},j_{v(1)}(\kappa))^{N_{v(1)}}$  of cardinality $\kappa_v$. Hence, by closure of $N_{v(1)}$ under $\kappa_v$-sequences, $D\in N_{v(1)}$. Thus, there is an upper bound for the set $D$ that belongs to $F_v$.  Denote this latter by $H_v$. 

\smallskip

By normality, we have that $B_v:=\diagonal_{s^+} B_{s^+,v}\in \mathscr{F}_v$, where 
 $$\diagonal_{s^+}B_{s^+,v}:=\{w\in \mathcal{U}_\infty\mid s^+\in [V_{\kappa_w}]^{<\omega}\,\Longrightarrow\, w\in B_{s^+,v}\}.$$
For each $v\in A_0$, put  $A_v:=\bigcap_{s^+} A_{s^+,v}$. By completeness, $A_v\in\mathscr{F}_u$. Next, set
$B:=A^*\cap(\diagonal_{v} A_{v})\in \mathscr{F}_u$ (cf. Definition~\ref{DiagonalIntersection}). Finally, arguing as in the previous paragraph we can let $H\in F^*_u$ be an upper bound for the $F_{s^+,v}$'s.

\smallskip


Let us now move to \emph{the Röwbottom part} of the proof. For each lower part $s^+=(s,h)\in [V_{\kappa_u}]^{<\omega}$ consider the sets
$$B_0(s^+):=\{v\in B\mid (*)^0_{s^+,v}\text{ holds}\},\  B_1(s^+):=\{v\in B\mid (*)^1_{s^+,v}\text{ holds}\}\footnote{By convention if $(*)_{s^+,v}$ holds then $s^+\in [V_{\kappa_v}]^{<\omega}$.}$$ and $B_2(s^+):=B\setminus (B_0(s^+)\cup B_1(s^+))$. For each $s^+$ the above partitions the set $B\in\mathscr{F}_u$ into three components. In particular, for each $s^+$ and every $\alpha\in [1,\len(u)$) there is an index $i_\alpha\in 3$  
such that $B_{i_\alpha}(s^+)\in u(\alpha)$.  Let $A(\alpha):=\diagonal_{s^+} B_{i_\alpha}(s^+)\in u(\alpha)$ and $B^*=B\cap (\cup_{\alpha<\len(u)}A(\alpha))\in\mathscr{F}_u$. 

\smallskip

Let $r$ be an extension of $p^*=(u, \omega, h_0, B^*,H^*)$ deciding $\sigma$ with minimal length. We will be done as long as we argue that $\ell(p^*)=\ell(r)$. Without loss of generality assume that $r\Vdash_{\mathbb{R}_u}\sigma$. Suppose towards a contradiction that the length of $r$ is $n+2$: namely, $r$ takes the form
$$r= s^{\smallfrown}( v,\kappa_{w_{\ell(s)}},h,B,H)^{\smallfrown}( u,\kappa_v,g, A,F), \ \text{with $|s|=n$.}$$

Denote  $s^+:=(s,h)$. Clearly, there exists $\alpha$ such that $v\in A(\alpha)$, and since $v\in B_{i_\alpha}(s^+)$, it follows that $i_\alpha=0$. Hence,  for every $v\in A(\alpha)$ with $s^+\in [V_{\kappa_v}]^{<\omega}$, $(*)^0_{s^+,v}$ holds: to wit, there are $B_{v}$ and $H_{v}$  such that
$$(\star) \ \ s^{\smallfrown}( v,\kappa_{w_{\ell(s)}},h,B_v,H_v)^{\smallfrown}( u,\kappa_v,H^*(\kappa_v),B^*,H^*)\Vdash_{\mathbb{R}_u}\sigma.$$ 
Since $A(\alpha)\in u(\alpha)$ it follows that $u\restriction\alpha\in j(A(\alpha))$. Therefore we can define $\l B^{<\alpha},H^{<\alpha}\r:=j(v\mapsto \l B_{v},H_{v}\r)(u\restriction\alpha)$.\footnote{Here $(j,F)$ is a weak constructing pair from which $u$ is inferred.} By elementarity, $B^{<\alpha}\in \mathscr{F}_{u\restriction\alpha}$ and $[H^{<\alpha}]_{u(1)}\in F_u$. Also define
$$B^{\alpha}:=\{v\in A(\alpha)\mid B^{<\alpha}\cap V_{\kappa_v}=B_{v}\wedge H^{<\alpha}\restriction \kappa_v=H_{v}\}\in u(\alpha).$$
Pick $H^{**}\in F^*_u$ such that $[F^{<\alpha}]_{u(1)}, [H^*]_{u(1)}\leq [H^{**}]_{u(1)}.$ Finally, let
$$B^{>\alpha}:=\{v\in B^*\mid \exists \xi<\len(v)\, (B^\alpha\cap V_{\kappa_v}\in v(\xi))\}.$$
For all $\beta\in (\alpha,\len(u))$, $u\restriction\beta \in j(B^{>\alpha})$ because $j(B^\alpha)\cap V_\kappa=B^\alpha\in u(\alpha)$. This implies that    $B^{>\alpha}\in u(\beta)$, for all $\beta\in (\alpha,\len(u))$.  

Put $B^{**}:=B^*\cap(B^{<\alpha}\cup B^{\alpha}\cup B^{>\alpha})$ and $q:=s^{\smallfrown}( u,\kappa_{w_{\ell(s)}},h,B^{**},H^{**})$. 

We claim that $q\Vdash_{\mathbb{R}_u}\sigma$, which will produce a contradiction with the minimality choice upon $n$.  It thus suffices to prove that for any $q\leq q'$ there is $v\in B_0(s^+)$ such that $q'$ is compatible with the condition
$$q_v:=s^{\smallfrown}( v,\kappa_{w_{\ell(s)}},h,B_v,H_v)^{\smallfrown}( u,\kappa_v,H^*(\kappa_v),B^*,H^*).$$ In effect, if this is the case $(\star)$ will in particular imply that $q'$ is compatible with a condition forcing $\sigma$. So, suppose that
$$q'=t^{\smallfrown}t'^{\smallfrown}( u,\lambda',h',A',H'), \ \text{with ${w_{\ell(t)}}={w_{\ell(s)}}$.}$$
Let us split the discussion into three cases:

\smallskip

\textbf{Case 1:} Suppose $|t'|=0$. Find some $v\in B^\alpha\cap A'$ with $\kappa_v>\kappa_{w_{\ell(s)}}$. Note that such $v$ exists, for 
  $B^\alpha\cap A'\in u(\alpha)$.
 Also, it follows that $v\in B_0(s^+)$. 
 
 Consider the condition $q^*:=q'{}\cat v$. Writing $q^*$ explicitly, we have that $$q^*:=t^{\smallfrown}( v,\kappa_{w_{\ell(t)}},h',A'_{v,-}, H'\restriction \z{A}'_{v,-})^{\smallfrown}( u,\kappa_v,H'(\kappa_v),A'_{v,+},H'\restriction \z{A}'_{v,+}).$$
 \begin{claim}
 $q_v\leq q^*$.
 \end{claim}
 \begin{proof}[Proof of claim]
 We split the verification in three parts.
 
 $\br$ Since $q\leq_{\mathbb{R}_u} q^*$ then  $s\leq_{\mathbb{R}_{w_{\ell(t)}}}t$.
 
 $\br$ First, $h\leq h'$ follows from $q\leq_{\mathbb{R}_u}q'$. Second, $A'\cap V_{\kappa_v}\s B^{<\alpha}\cap V_{\kappa_v}=B_v$, as $v\in B^\alpha$. By similar reasons, $H_v=H^{<\alpha}\restriction \kappa_v$. Using again that $q\leq_{\mathbb{R}_u} q^*$ we have that $H_v(\kappa_w)\leq H'(\kappa_w)$, for all $w\in A'_{v,-}$. 
 
 $\br$  For the upper part, $q\leq_{\mathbb{R}_u} q^*$ again implies, $H'(\kappa_v)\geq H^{**}(\kappa_v)\geq H^*(\kappa_v)$, $A'_{v,+}\subseteq B^*$ and $H'(\kappa_w)\geq H^{**}(\kappa_w)\geq H^*(\kappa_w)$, for all $w\in A'_{v,+}.$
 \end{proof}
 

\textbf{Case 2:} Suppose $|t'|>0$ and that  every $w_m$ in $t'$ is from $B^{<\alpha}$. Then, pick $v\in B^{\alpha}\cap A'$ such that $\kappa_{v}>\kappa_{w_{\ell(t')}}$. 
By the same argument as in Case~1, the condition
$q^*=q'{}\cat v$
will be above both $q'$ and $q_v$.

\smallskip

\textbf{Case 3:} Suppose we do not fall in none of the above cases. Let $i<|t'|$ be the minimal index such that $w_i\in B^{\alpha}\cup B^{>\alpha}$. If $w_i\in B^{\alpha}$ then $q'\geq q_{w_i}$. 
Otherwise, $w_i\in B^{>\alpha}$ and by definition there is $\xi<\len(w_i)$ such that $B^\alpha\cap V_{\kappa_{w_i}}\in w_i(\xi)$. Therefore we can find $v\in B^\alpha\cap V_{\kappa_{w_i}}$ such that $\kappa_v>\kappa_{w_{i-1}}$ 
and argue as before that  $q^*:=q'{}\cat v$ is stronger than  $q'$ and $q_v$.

\smallskip

The above completes the proof of the Prikry lemma.
\end{proof}

\begin{remark}\label{TailVeryclosed}
An outright consequence of the above lemma is that  for each $p\in \mathbb{R}_u$,   forcing with $\mathbb{R}_u/p$ does not introduce new subsets to $\lambda_0^{+4}$. Indeed, this is becuase $\mathbb{R}_u/p$ is $\lambda_0^{+4}$-closed with respect to the order $\leq^*$ and the triple $\langle \mathbb{R}_u,\leq, \leq^*\rangle$ has the Prikry property. For details, see  \cite[Lemma~1.9]{MR2768695}.
\end{remark}

Let $\vec{u}=\langle u_\xi\mid \xi<\kappa\rangle$ be the Radin sequence. For $\xi<\kappa$, put $\kappa_\xi:=\kappa_{u_\xi}$. 

\begin{proposition}[Cardinal structure]\label{CardinalStructure}
 The following holds in $V[G]$:
\begin{enumerate}
    \item Every $V$-cardinal $\geq \kappa_u^{+}$ is a cardinal;
    \item The only cardinals $\leq \kappa_u$  are
    $$\{\aleph_0,\aleph_1, \aleph_2, \aleph_3, \aleph_4\}\cup \{(\kappa_{\xi}^{+k})^V\mid 1\leq k\leq 4,\, \xi<\kappa\}\cup\Lim(C)\cup\{\kappa_u\};$$
\end{enumerate}
Also, if $u$ has a repeat point (i.e., there is $\gamma<\len(u)$ such that $\mathscr{F}_u=\mathscr{F}_{u\restriction\gamma}$) then  $\kappa_u$ remains measurable in $V[G]$.
\end{proposition}
\begin{proof}
(1) follows from Lemma~\ref{ccness}. For (2) we  argue as follows:

$\br$ Let us first show that the first $4$ infinite cardinals of $V$ are preserved. Let $p\in G$ mentioning $u_0$, the first member of $\vec{u}$. By Lemma~\ref{Factoring}(1) we have that
$\mathbb{R}_u/p\simeq \mathbb{R}_{u_0}/p^{\leq 0}\times \mathbb{R}_u/p^{>0}.$ On one hand,  note that the second of this forcings does not affect the combinatorics of the first $4$ alephs (see Remark~\ref{TailVeryclosed}).\footnote{Also, recall that we pick $G$ in such a way that if $p\in G$ then $\lambda_0:=\omega$ (see page~\pageref{ChoiceofG}).} On the other hand, $\len(u_0)=1$ and $p^{\leq 0}\in \mathbb{R}_{u_0}^*$ hence Remark~\ref{RemarkOnlength} yields $\mathbb{R}_{u_0}/p^{\leq 0}\simeq \Col(\aleph_4,\kappa_{0})$. This yields the desired result. 

$\br$ The argument is similar to the above. Let $\xi<\kappa$ and $p\in G$ mentioning both $u_{\xi}$ and $u_{\xi+1}$. Invoking Lemma~\ref{Factoring} we have $$(\star)\;\mathbb{R}_u/p\simeq \mathbb{R}_{u_\xi}/p^{\leq n}\times \Col(\kappa_{\xi}^{+4},\kappa_{\xi+1})\times\mathbb{R}_u/p^{>n+1}.$$
Yet again, the top-most forcing does not affect the combinatorics up to $\kappa^{+4}_\xi$. Regarding the bottom-most one, this is $\kappa_\xi^{+}$-cc and so it preserves all cardinals $\geq \kappa_\xi^{+}$. Finally,  the Lévy collapse does not disturb this pattern up to (and including) $\kappa_{\xi}^{+4}$. Since this holds for arbitrary $\xi<\kappa$, it follows that every member of $\Lim(C)$ is also a cardinal in $V[G]$. 

\smallskip

Altogether, the above yields one inclusion. As for the converse, if $\theta$ is a $V$-cardinal not in the above described set then there would be some $\xi<\kappa$ such that $\kappa_{\xi}^{+4}<\theta\leq \kappa_{\xi+1}$. Note that  $(\star)$ would imply
that $\theta$ is collapsed.

\smallskip

$\br$ The preservation of $\kappa_u$ follows from Lemma~\ref{ClubIndeed} and the above. Also, a similar argument combining $(\star)$ with Remark~\ref{TailVeryclosed} yield that $\kappa_u$ is strong limit in $V[G]$. For the moreover part see \cite[Theorem~5.15]{MR2768695}.
\end{proof}

\section{The consistency of the local and global failure}\label{SectionLocalAndGlobal}
In \cite{MR830084},  Abraham and Shelah proved the consistency of \textsf{ZFC} with the negation of Galvin's property at the successor of a regular cardinal $\kappa$. 
For this they started with a model of \textsf{GCH} and used the  poset $\mathbb{S}(\kappa,\lambda)$ of Definition~\ref{AbSh forcing} to obtain the desired configuration. 
Unfortunately, as in many other similar forcing constructions, the  arguments in \cite{MR830084} cannot be adapted to handle the case where $\kappa$ is a singular cardinal. In effect, if $\kappa$ is singular then  forcing with $\mathbb{S}(\kappa)$ will collapse $\kappa$, hence voiding the whole argument. This raises the natural question on whether the failure of Galvin's property can be obtained at successors of  singular cardinals. Our mission in this section is to answer this in the affirmative, provided the set-theoretic universe contains a supercompact cardinal. More specifically, we will begin showing that  Galvin's property can consistently fail at the successor of a strong limit singular cardinal of any prescribed cofinality. Later,  we will improve this result by showing that Galvin's property can consistently fail at the first successor of every singular cardinal, simultaneously.   

\subsection{Local failure of Galvin's property}

\smallskip

The following observation will be our  \emph{master lemma} for the entire Section~\ref{SectionLocalAndGlobal}:

\begin{lemma}[Preserving the failure of Galvin's property] 
\label{lemchain} Let $\kappa,\partial, \lambda$ be  infinite cardinals such that 
 $\kappa^+\leq\partial\leq\lambda$.
Assume further that Galvin's property ${\rm Gal}(\mathscr{D}_{\kappa^+},\partial,\lambda)$ fails in $V$.
\begin{enumerate}
\item[$(\aleph)$] If $\mathbb{P}$ is a $\kappa^+$-cc forcing notion and  $G\subseteq\mathbb{P}$ is $V$-generic, then $$V[G]\models \neg {\rm Gal}(\mathscr{D}_{\kappa^+},\partial,\lambda).$$ 
\item[$(\beth)$] If $\mathbb{P}$ is a $\lambda$-closed forcing notion and $G\subseteq\mathbb{P}$ is $V$-generic, then  $$V[G]\models\neg {\rm Gal}(\mathscr{D}_{\kappa^+},\partial,\lambda).$$ 
\end{enumerate}
\end{lemma}

\begin{proof}

$(\aleph)$: Let $\{C_\alpha\mid \alpha\in\lambda\}\subseteq\mathscr{D}_{\kappa^+}$  witness  the failure of ${\rm Gal}(\mathscr{D}_{\kappa^+},\partial,\lambda)$ \linebreak in the ground model, $V$.
Let $C$ be a club subset of $\kappa^+$ in $V[G]$.
By the $\kappa^{+}$-cc of the poset $\mathbb{P}$ we can find a club $C'\s C$ of $\kappa^{+}$ lying in $V$.
Note that $$|\{\alpha<\lambda\mid C\subseteq C_\alpha\}|\leq|\{\alpha<\lambda\mid C'\subseteq C_\alpha\}|<\partial,$$ where the last inequality is justified by the failure of ${\rm Gal}(\mathscr{D}_{\kappa^+},\partial,\lambda)$ in $V$.
\smallskip

$(\beth)$: Since $\mathbb{P}$ adds neither new clubs of $\kappa^+$ nor $\partial$-sequences 
in $\lambda$, the sequence
$\{C_\alpha\mid \alpha\in\lambda\}$ exemplifies the failure of ${\rm Gal}(\mathscr{D}_{\kappa^+},\partial,\lambda)$ in $V[G]$.
\end{proof}

The following result provides a positive answer for the consistency of the failure of Galvin's property at the level of successors of singulars:

\begin{theorem}
\label{thmmt} Let $\kappa$ be a supercompact cardinal and $\lambda>\kappa^+$ be any cardinal. 
Then the following two statements  are consistent with \textsf{ZFC}:
\begin{enumerate}
\item[$(\aleph)$] $\kappa$ is supercompact and $\mathrm{Gal}(\mathscr{D}_{\kappa^+},\kappa^+,\lambda)$ fails;
\item [$(\beth)$] $\kappa$ is a strong limit singular cardinal and $\mathrm{Gal}(\mathscr{D}_{\kappa^+},\kappa^+,\lambda)$ fails.
\end{enumerate}
\end{theorem}

\begin{proof}
$(\aleph)$: By \cite{MR0472529} we may assume that $\kappa$ is indestructible under $\kappa$-directed-closed forcing notions.
We shall also assume that \textsf{GCH} holds above $\kappa$ and let $\mathbb{S}({\kappa, \lambda})$  be the corresponding Abraham-Shelah forcing (see Definition~\ref{AbSh forcing}).
The forcing $\mathbb{S}({\kappa},\lambda)$ is $\kappa$-directed-closed, hence it preserves supercompactness of $\kappa$ and yields a model of $\neg \mathrm{Gal}(\mathscr{D}_{\kappa^+},\kappa^+,\lambda)$.

$(\beth)$: For the second part, force over a model witnessing $(\aleph)$ with either Prikry or Radin forcing, depending on the desired cofinality on $\kappa$.
\end{proof}
In the light of the previous theorem it is natural to ask whether one can get the failure of Galvin's property at the successor of a more \emph{down to earth} cardinal, such as $\aleph_{\omega}$. While it is true that the standard \emph{Prikry forcing with interleaved collapses} is $\kappa^{+}$-cc it is nevertheless not clear how to define suita\-ble guiding generics in the model of Clause~$(\aleph)$. In effect, note that in this model $2^\kappa=\kappa^{++}$ and hence the usual arguments to build guiding generics over the normal ultrapower break down (see e.g. \cite[Example~8.6]{Cummings-handbook}). In the next subsection we shall address this issue,  proving that $ \Gal{\kappa^+}{\kappa^{++}}$ can actually fail for all singular cardinals $\kappa$.
\subsection{Global failure of Galvin's property}\label{SectionGlobal}
The present section is devoted to the proof of the following result:
\begin{theorem}
\label{thmglobal} 
Assume that \textsf{ZFC} is consistent with the existence of a supercompact cardinal.
Then   \textsf{ZFC} is also  consistent with  
$$\text{$``{\rm Gal}(\mathscr{D}_{\kappa^+},\kappa^+,\kappa^{++})$ fails at every limit cardinal $\kappa$''.}$$
Moreover, \textsf{ZFC} is consistent with 
$$\text{$``\Gal{\aleph_{4\cdot\xi+1}}{\aleph_{4\cdot\xi+2}}$ fails for every $\xi\in \mathrm{Ord}$''.}$$
\end{theorem}

The strategy is to start forcing with 
$\mathbb{S}(\kappa)$ and obtain the negation of Galvin's property at the first successor of a supercompact cardinal $\kappa$ (see Theorem~\ref{thmmt}). Afterwards, we force with Radin forcing aiming to introduce a generic club $C\s \kappa$ where Galvin's property fails at the successor of every member of $C$. Since we wish to get a model where Galvin's property fails at every successor of a singular cardinal, we shall additionally
collapse the intervals between the successive Radin points. In particular this will guarantee that the limit cardinals of the resulting universe are the accumulation points of the Radin club 
$C$.   The main issue is how to perform the collapses and yet preserve the failure of  Galvin's property. As Lemma~\ref{lemchain} advanced, a good chain condition is sufficient for this. A na\"{i}ve attempt would be to force with the usual Radin forcing and later force with an Easton-support iteration of suitable Lévy collapses. However, this approach is problematic. In effect,  after forcing with $\mathbb{S}(\kappa)$  the power set of  $\kappa$ 
becomes large (i.e., $\kappa^{++}$), and this property is reflected down to all points in the eventual Radin club. In particular, when reaching $\kappa_\omega$ --the first limit point in $C$--  our  iteration of Lévy collapses  will have the $2^{\kappa_{\omega}}$-cc and $2^{\kappa_{\omega}}>\kappa_{\omega}^+$. To work around this  
there is a need for constructing the club and perform the collapses \emph{simultaneously}; to wit, we shall interleave the Lévy collapses within the Radin forcing using  the mechanism of \emph{guiding generics}. Although the above-described forcing construction has already appeared in previous works \cite{CumGCH, YairEskew, Golshani2016WeakDC} 
the fact that 
we are working in the absence of \textsf{GCH} makes the arguments more complicated. This is particularly\- evident in  the construction of the guiding generics (Lemma~\ref{LemmaConstructingPair}).  In \cite{CumGCH}, Cummings constructs such guiding generics by performing a two-step preparation that introduces $\alpha^{++}$-many Cohen subsets at every inaccessible cardinal $\alpha\leq\kappa$, where $\kappa$ is the relevant cardinal. In particular, Cummings arguments yield a model where \textsf{GCH} fails at $\kappa$ and there is a guiding generic for a suitable Lévy collapse.  If we would like to mimic this argument we should carry out an analogous forcing iteration  replacing $\mathrm{Add}(\alpha,\alpha^{++})$ by the poset $\mathbb{S}(\alpha)$.\label{discussion} 
Unfortunately, Cummings' approach heavily relies on some nice features of Cohen forcing that are hardly extrapolated to our  context; instead, we shall use Woodin's \emph{forcing-theoretic surgery}. Specifically, our arguments are inspired by Gitik-Sharon proof of the consistency of $\neg \textsf{SCH}_{\aleph_{\omega^2}}+\neg \textsf{AP}_{\aleph_{\omega^2}}$ \cite{GitSha}.

\smallskip

The structure of the section is as follows. We shall begin proving --under the assumption of a supercompact cardinal-- the existence of a weak constructing pair $(j,F)$ (a.k.a. \emph{guiding generics}). Later, we shall use the pair $(j,F)$ to define the corresponding Radin forcing with collapses (see \S\ref{SubsectionRadinwithcollapses}). By forcing with this poset  we will produce a model exemplifying the configuration displayed in  Theorem~\ref{thmglobal}.
\begin{lemma}[Constructing pairs]\label{LemmaConstructingPair}\label{Revise it with the change of collapses}
Let $\kappa$ be a supercompact cardinal. Then there is a model of \textsf{ZFC} where the following hold:
\begin{enumerate}
    \item $\Gal{\kappa^+}{\kappa^{++}}$ fails and \textsf{GCH} holds for all cardinals $\geq \kappa^{+}$;
    \item for every regular cardinal $\lambda\geq\kappa^{++}$, there is  a weak constructing pair $(j_\lambda,F_\lambda)$ such that $j_\lambda$ witnesses that $\kappa$ is $\lambda$-supercompact.
\end{enumerate}
\end{lemma}

\begin{proof}
Let $\kappa$ be a supercompact cardinal. Without loss of generality assume that there is no inaccessible cardinal ${>}\kappa$.\footnote{If $\lambda$ is the first inaccessible above $\kappa$ then $V_\lambda$ is a model of \textsf{ZFC} where $\kappa$ is supercompact and there is no inaccessible above $\kappa$. In that case $W:=V_\lambda$ will be our initial ground model.} By forcing with the standard Jensen's iteration we may assume that the \textsf{GCH} holds (see \cite[\S8]{BP}). 
Appealing to Theorem~\ref{thmmt} we get a model of \textsf{ZFC} where $\kappa$ is supercompact, $2^\kappa=\kappa^{++}$,  the \textsf{GCH} holds for cardinals $\geq \kappa^{+}$ and $\Gal{\kappa^+}{\kappa^{++}}$ fails.

Denote the above resulting model by  $V$. Working in $V$, let $\mathbb{P}$ be the Easton-support iteration $\langle \mathbb{P}_\alpha; \lusim{\mathbb{Q}}_\beta\mid \beta\leq \alpha\leq \kappa\rangle$ defined as follows: if $\alpha\leq \kappa$ is inaccessible then $\one_{\mathbb{P}_\alpha}\forces_{\mathbb{P}_\alpha}\lusim{\mathbb{Q}}_\alpha=\lusim{\mathrm{Add}}(\alpha,\alpha^{++})$ while $\one_{\mathbb{P}_\alpha}\forces_{\mathbb{P}_\alpha}\text{$``\lusim{\mathbb{Q}}_\alpha$ is trivial''}$, otherwise. Let $G$ a $V$-generic filter for $\mathbb{P}$. By factoring $\mathbb{P}$ as $\mathbb{P}_\kappa\ast \lusim{\mathbb{Q}}_\kappa$  we have that $G=G_\kappa\ast g$ for suitable generic filters for $\mathbb{P}_\kappa$ and $(\lusim{\mathbb{Q}}_\kappa)_{G_\kappa}$, respectively.

Since $\mathbb{P}$ is $\kappa^{+}$-cc then Lemma~\ref{lemchain}($\aleph$) ensures that $\Gal{\kappa^+}{\kappa^{++}}$ fails in $V[G]$.  Additionally, a counting-nice-names argument shows that the \textsf{GCH} pattern for cardinals $\geq \kappa^{+}$ is preserved after forcing with $\mathbb{P}$. This yields (1).

\begin{claim}
Clause~(2) of the lemma holds in $V[G]$.
\end{claim}
\begin{proof}[Proof of claim]
Let $\lambda\geq \kappa^{++}$ be a regular cardinal in $V[G]$. Working in the model $V$, let $j\colon V\rightarrow{M}$ be a 
$\lambda$-supercompact embedding derived from a 
$\lambda$-supercompact 
measure on $\mathcal{P}_\kappa(\lambda)$. 
Let  $\ell \colon V\rightarrow{\bar{M}}$ be the $\kappa^+$-supercompact embedding induced by the projection of this measure onto $\mathcal{P}_\kappa(\kappa^+)$. As usual, there is a canonical map that factors $j$ through $\ell$: namely, the map 
$$k(\ell(f)(\ell``\kappa^+)):=j(f)(j``\kappa^+).$$
Note that  $\kappa^{++}_{\bar{M}}=\kappa^{++}=\kappa^{++}_M$ and $\kappa^{+3}_{\bar{M}}<\ell(\kappa)<\kappa^{+3}$, 
hence $\crit(k)=\kappa^{+3}_{\bar M}$.

\smallskip

We now show how to lift the embeddings $j, \ell$ and $k$. Let us commence by lifting these embeddings after forcing with $\mathbb{P}_\kappa$. Note that
\begin{eqnarray*}
\ell(\mathbb{P_\kappa})=\mathbb{P}_\kappa\ast \lusim{\mathrm{Add}}(\kappa,\kappa^{++})\ast \ell(\mathbb{P})_{(\kappa,\ell(\kappa))},\\
j(\mathbb{P}_\kappa)=\mathbb{P}_\kappa\ast \lusim{\mathrm{Add}}(\kappa,\kappa^{++})\ast j(\mathbb{P})_{(\kappa,j(\kappa))}.
\end{eqnarray*}
It is clear that $\ell$ can be lifted to an embedding $\ell\colon V[G_\kappa]\rightarrow{\bar{M}[G\ast H]}$, where $H$ is $\bar{M}[G]$-generic for $\ell(\mathbb{P})_{(\kappa,\ell(\kappa))}$.
\begin{sclaim}
There is some of such $H$ in $V[G]$. 

In particular, $\ell$ lifts to a $V[G]$-definable  embedding $\ell\colon V[G_\kappa]\rightarrow{\bar{M}[\ell(G_\kappa)]}$ whose target model is closed under $\kappa^+$-sequences in $V[G]$.
\end{sclaim}
\begin{proof}[Proof of subclaim]
This is a quite standard argument so we just sketch it. First, note that $\bar{M}[G]$ is closed under $\kappa^+$-sequences in $V[G]$, hence $\ell(\mathbb{P})_{(\kappa,\ell(\kappa))}$ is $\kappa^{++}$-directed-closed in $V[G]$.\footnote{The closure under $\kappa^+$-sequences follows from the $\kappa^+$-ccness of $\mathbb{P}$.} Also every dense set in $M[G]$ for this forcing poset takes the form $\ell(f)(\ell``\kappa^+)$, where $f\in V[G]$ and $f\colon \mathcal{P}_\kappa(\kappa^+)\rightarrow{V_{\kappa+1}}$. Thus,  there are at most $2^{\kappa^+}=\kappa^{++}$-many such sets. Combining the closure of the model with the closure  of $\ell(\mathbb{P})_{(\kappa,\ell(\kappa))}$ in  $V[G]$ one can constructs the desired generic filter $H$. Finally, the second claim follows by noticing that $\ell(\mathbb{P})_{(\kappa,\ell(\kappa))}$ does not introduce new $\kappa^+$-sequences to $\bar{M}[G]$.
\end{proof}

Let us now lift the embedding $k$. 

\begin{sclaim}\label{liftingk}
The map $k$ lifts to a $V[G]$-definable elementary embedding $k\colon \bar{M}[\ell(G_\kappa)]\rightarrow M[j(G_\kappa)]$ whose target is closed under $\lambda$-sequences in $V[G]$.
\end{sclaim}
\begin{proof}[Proof of subclaim]
Since $\crit(k)>\kappa^{++}$ then $k``\mathbb{P}=\mathbb{P}$ and so $k$ lifts to $k\colon \bar{M}[G]\rightarrow M[G]$. Certainly, $M[G]$ is closed under $\lambda$-sequences in $V[G]$. 

Let $H$ be the generic filter constructed in the previous subclaim. Clearly, $k``H$ is a directed subset of $j(\mathbb{P})_{(\kappa,j(\kappa))}$, $k``H\in V[G]$ and $|k``H|=\kappa^{++}$. Thus,  the closure of the model $M[G]$ in $V[G]$ implies  that $k``H\in M[G]$. So, we can pick $p\in j(\mathbb{P})_{(\kappa,j(\kappa))}$ a upper bound for $k``H$.  Arguing as in the previous subclaim  construct $H'\in V[G]$ a $M[G]$-generic filter for $j(\mathbb{P})_{(\kappa,j(\kappa))}$ such that $p\in H'$. This is indeed feasible because the following hold:
\begin{enumerate}
    \item[$(\alpha)$] $M[G]$ is closed under $\lambda$-sequences in $V[G]$;
    \item[$(\beta)$] $j(\mathbb{P})_{(\kappa,j(\kappa))}$ is $\lambda^+$-directed-closed in $V[G]$;\footnote{Here we use our original assumption that there were not inaccessible cardinals ${>}\kappa$. Clearly, $\mathbb{P}$ does not create new such inaccessibles and so there are no $M[G]$-inaccessible cardinals in $(\kappa,\lambda]$.}
    \item[$(\gamma)$] There are at most $\lambda^+$-many dense sets $D\in M[G]$ for $j(\mathbb{P})_{(\kappa,j(\kappa))}$. Since now $\crit(k')>i(\kappa)$ this yields an $N$-generic filter for 
\end{enumerate}
Thereby, $k$ lifts to  $k\colon \bar{M}[G\ast H]\rightarrow{M}[G\ast H']$.  Clearly, $k$ is $V[G]$-definable  and  its target  is closed under $\lambda$-sequences in $V[G]$.
\end{proof}

Incidentally the above also shows that $j$ lifts to $j\colon V[G_\kappa]\rightarrow{M}[j(G_\kappa)].$

\smallskip

Let us now address the lifting under $\mathrm{Add}(\kappa,\kappa^{++})_{V[G_\kappa]}$ beginning with $\ell$:

\begin{sclaim}\label{SubclaimGitikSharon}
 $\ell$ lifts to a $V[G]$-definable  embedding $\ell\colon V[G]\rightarrow{\bar{M}[\ell(G)]}$ whose target model is closed under $\kappa^+$-sequences in $V[G]$. Moreover, $\ell(\kappa)$ can be enumerated as $\langle \ell(c_\alpha)(\kappa)\mid \alpha<\kappa^{++}\rangle$, for some functions $c_\alpha\colon \kappa\rightarrow \kappa$.
\end{sclaim}
\begin{proof}[Proof of sublcaim]
This can be proved as in \cite[Lemma 2.26]{GitSha}.
\end{proof}

The next subclaim shows that $k$ also lifts:

\begin{sclaim}
 The map $k$ lifts to a $V[G]$-definable elementary embedding $k\colon \bar{M}[\ell(G)]\rightarrow M[j(G)]$ whose target is closed under $\lambda$-sequences in $V[G]$.
\end{sclaim}
\begin{proof}[Proof of subclaim]
Let $h$ be the $\ell(\mathbb{Q}_\kappa)$-generic filter mentioned in $\ell(G)$. Note that $k``h$ is a subset  of $j(\mathbb{Q}_\kappa)$ 
of size $|\ell(\kappa^{++})|=\kappa^{++}$, hence $k``h\in M[j(G_\kappa)]$. Since $k``h$ is directed we can pick $p\in j(\mathbb{Q}_\kappa)$ an upper bound for $k``h$. Again, Clauses~$(\alpha)$--$(\gamma)$ of Subclaim~\ref{liftingk} hold when replacing $M[G]$ by $M[j(G_\kappa)]$ and $j(\mathbb{P})_{(\kappa,j(\kappa))}$ by $j(\mathbb{Q}_\kappa)$, and so we can construct $h'\in V[G]$ with $p\in h'$.
Altogether, the map $k$ lifts to $k\colon \bar{M}[\ell(G)]\rightarrow{ M[j(G)]}.$
\end{proof}
Finally, $j$ lifts to $j\colon V[G]\rightarrow{M[j(G)]}$, because $j``g\s h'.$

\smallskip

We have thus formed elementary embeddings $j\colon V[G]\rightarrow{M[j(G)]}$ and $\ell\colon V[G]\rightarrow{\bar{M}[\ell(G)]}$ that witness  $\lambda$ and $\kappa^+$-supercom\-pactness of $\kappa$ in $V[G]$. Also, we have obtained $k\colon \bar{M}[\ell(G)]\rightarrow{M[j(G)]}$ with $\crit(k)>\kappa$. 

\begin{sclaim}
 There is a weak constructing pair $(j_\lambda, F_\lambda)$ such that the elementary embedding  $j_\lambda$ witnesses that $\kappa$ is $\lambda$-supercompact.
\end{sclaim}
\begin{proof}[Proof of subclaim]
Let $\mathscr{U}$ be the normal measure on $\kappa$ derived from $\ell$ and $i\colon V[G]\rightarrow{N}$ be the corresponding ultrapower embedding. Let $k'\colon N\rightarrow \bar{M}[\ell(G)]$ be factor map between $\ell$ and $i$: namely, $k'(i(f)(\kappa))=\ell(f)(\kappa)$. By Subclaim~\ref{SubclaimGitikSharon}, every ordinal  $\alpha\leq \ell(\kappa)$ can be represented using a function $c\colon \kappa\rightarrow\kappa$, hence $\ell(\kappa)+1\s \mathrm{ran}(k')$, and thus so $\crit(k')>i(\kappa)$. A standard counting argument allows us to form a $\bar{M}[\ell(G)]$-generic filter $F$ for $\Col(\kappa^{+4},\ell(\kappa))^{\bar{M}[\ell(G)]}$ in $V[G]$. Since $\crit(k')>i(\kappa)$ then it is clear that $F$ is also $N$-generic for $\Col(\kappa^{+4},i(\kappa))^N$. 

On another front $\mathscr{U}$ is the normal measure generated by $j$, for  $\crit(k)>\kappa$. Also, since $i(\kappa)<\kappa^{+3}$ then $F\in H({\kappa^{+3}})^{V[G]}\s M[j(G)]$. Therefore $(j,F)$ is the sought weak constructing pair.
\end{proof}
This completes the proof of the claim.
\end{proof}
This is the end of the lemma.
\end{proof}
\begin{remark}
The above lemma can be slightly modified to get a constructing pair in the sense of \cite[\S3]{CumGCH}. For this,  derive a large enough extender from the final $\lambda$-supercompact embedding and form the corresponding  ultrapower. Once again, the normal measure derived from both the supercompact and extender  embedding are the same. The difference now is that the factor map between the normal ultrapower and the extender embedding has a rather small width, hence the generic $F$ is transferable. 
\end{remark}
It is plausible that Lemma~\ref{LemmaConstructingPair} can be obtained from more modest large-cardinal assumptions. For instance, Cummings obtained a similar result from the weaker assumption of $\kappa$ being $\mathcal{P}_3\kappa$-hypermeasurable \cite[\S3]{CumGCH}. Regrettably, it is not clear to the authors how Cummings' construction can be adapted to get $\neg\Gal{\kappa^+}{\kappa^{++}}$. See the discussion at page~\pageref{discussion}.

\smallskip

We are now in conditions to prove the promised theorem:

\begin{proof}[Proof of Theorem~\ref{thmglobal}]
Let $\kappa$ be a supercompact cardinal. Appealing to Lemma~\ref{LemmaConstructingPair} we get a model of \textsf{ZFC} where $\Gal{\kappa^+}{\kappa^{++}}$ fails, $\kappa$ is supercompact and there is a weak constructing pair $(j,F)$ witnessing that $\kappa$ is $\kappa^{++}$-supercompact. Denote this model by $V$ and let $u_*$ be the sequence inferred from $(j,F)$. Arguing  as in \cite[Lemma~1]{CumGCH}, for each 
$\alpha<\kappa^{+++}$ the sequence $u_*\restriction\alpha$ exists and belongs to $\mathcal{U}_\infty$.  
In particular, there is $\alpha<\kappa^{+3}$ such that the sequence $u:=u_*\restriction\alpha\in \mathcal{U}_\infty$ has a repeat point. 

\smallskip

Since $j\colon V\rightarrow M$ witnesses   $\kappa^{++}$-supercompactness,  $\Gal{\kappa^+}{\kappa^{++}}$ fails in the model $M$. Therefore, 
$$A:=\{w\in \mathcal{U}_\infty\mid \text{$``\Gal{\kappa_w^+}{\kappa_{w}^{++}}$ fails''}\}\in\mathscr{F}_{u}.$$

 Now, force with  $\mathbb{R}:=\mathbb{R}_u/p$, where $p:=\langle (u, \omega, \emptyset, A, H)\rangle$ and $H$ is some member of $F^*_u$. 
Let $G$ be a $\mathbb{R}$-generic filter over $V$ and  $C$ be the generic club induced by $G$.  By Proposition~\ref{CardinalStructure}, $\kappa$ remains inaccessible in $V[G]$. Let $\langle \kappa_\xi\mid \xi<\kappa\rangle$ be the increasing enumeration of $C$ and write, for each  $\xi<\kappa$, 
$$\Phi(\xi)\equiv \text{``$\Gal{\kappa_{\xi}^+}{\kappa^{++}_{\xi}}$ fails''}$$

\begin{claim}\label{preservingPhi}
For all $\xi<\kappa$, $\Phi(\xi)$ holds in $V[G]$. 
\end{claim}
\begin{proof}
Fix $\xi<\kappa$ and pick $p\in G$ mentioning both $u_\xi$ and $u_{\xi+1}$, say at coordinate $m$ and $m+1$, respectively. 
Note that $\Phi(\xi)$ holds in $V$ by virtue of our choice of $A$. Invoking Lemma~\ref{Factoring}(2) we have
$$\mathbb{R}_u/q\simeq \mathbb{R}_{w_{\xi}}/q^{\leq m}\times \Col(\kappa_{\xi}^{+4},\kappa_{{\xi+1}})\times\mathbb{R}_u/q^{>m+1}.$$
The top-most of these forcings does not introduce new subsets to $\kappa_\xi^{++}$, hence Lemma~\ref{lemchain} implies that $\Phi(\xi)$ is preserved in the corresponding generic extension, $V_1$. Also, in $V_1$, the remaining is a product of a $\kappa_{\xi}^+$-cc times  $\kappa^{+4}_{\xi}$-directed-closed forcings. Yet again, by Lemma~\ref{lemchain},  forcing with the first factor keeps $\Phi(\xi)$ true in the second generic extension, $V_2$. Also, Easton's Lemma implies that the remaining poset is $\kappa^{+4}_\xi$-directed-closed in $V_2$, hence $\Phi(\xi)$ holds in the final generic extension. 
Altogether, $\Phi(\xi)$ holds in $V[G]$. 
\end{proof}
Looking at Proposition~\ref{CardinalStructure}(2) it is easy to check that 
$$\kappa^+_\xi:=\begin{cases}
\aleph_{4\cdot(\xi+1)+1}, & \text{if $\xi<\omega$;}\\
\aleph_{4\cdot \xi+1}, & \text{if $\xi\geq \omega$.}
\end{cases}$$
By further forcing with $\Col(\aleph_0,\aleph_4)$ we get that, for every ordinal $\xi<\kappa$,  $\kappa_\xi^+=\aleph_{4\cdot\xi+1}$ holds in the resulting generic extension. Moreover, this forcing is $\aleph_5$-cc and, as a result, preserves $\Phi(\xi)$ for every $\xi<\kappa$ (see Lemma~\ref{lemchain}$(\aleph)$).\label{calculation}

\smallskip

Combining this with Claim~\ref{preservingPhi} we have that 
$$\text{$\Gal{\aleph_{4\cdot\xi+1}}{\aleph_{4\cdot\xi+2}}$ fails for all $\xi<\kappa$.}$$
Finally,  $W:=V[G]_\kappa$ yields the desired model of \textsf{ZFC}.
\end{proof}

\subsection{The impossibility of the ultimate global failure}\label{SectionImpossibility}
In this section we would like to demonstrate that Theorem~\ref{thmglobal} cannot be improved to encompass all regular cardinals. More precisely, we would like to argue that this configu\-ration is not achievable if one bears on Abraham-Shelah methods to produce the failure of Galvin's property.

\smallskip

The first observation that we make is that   $\mathrm{Gal}(\mathscr{D}_\kappa,\partial,\lambda)$  always holds, provided $\partial<\kappa$ or $\lambda>(2^\kappa)$. Thus, the only interesting cases are those where $\kappa\leq \partial\leq\lambda\leq 2^\kappa$. The second observation is that $\mathrm{Gal}(\mathscr{D}_\kappa,\partial,\lambda)$ yields $\mathrm{Gal}(\mathscr{D}_\kappa,\kappa,2^\kappa)$, and hence by negating the latter one obtains the failure of the former. Thereby, $\neg \mathrm{Gal}(\mathscr{D}_\kappa,\kappa,2^\kappa)$ yields a sort of \emph{ultimate failure} for Galvin's property at $\kappa$. 
The final observation in this vein that we would like to make is that forcing with $\mathbb{S}(\kappa)$  produces such a ultimate failure at $\kappa^+$.\footnote{See our explanations in page~\pageref{ExplanationonGalvin}.} In particular, in the model of Theorem~\ref{thmglobal} this ultimate failure holds for all successors of limit cardinals. All of these have inspired the next concept:

\begin{definition}[The Ultimate failure]
The \emph{Ultimate failure of Galvin's property} is the assertion that $\mathrm{Gal}(\mathscr{D}_\kappa,\kappa,2^\kappa)$ fails for all regular cardinal $\kappa$.
\end{definition}

\begin{proposition}
For every cardinal $\kappa$ there is a regular cardinal $\lambda>\kappa$ such that  $\mathrm{Gal}(\mathscr{D}_\lambda,\lambda,2^\lambda)$ holds. In particular, the Ultimate failure of Galvin's property is impossible.
\end{proposition}
\begin{proof}
Assume towards a contradiction that there is some cardinal $\kappa$ such that $\mathrm{Gal}(\mathscr{D}_\lambda,\lambda,2^\lambda)$ fails for all $\lambda\geq \kappa$. The proof strategy is to show that for some cardinal $\lambda>\kappa$ the weak diamond principle $\Phi_\lambda$ holds.\footnote{ The principle $\Phi$ (i.e. $\Phi_{\aleph_1}$) was introduced in \cite{MR0469756}.
For the definition of the more general principle $\Phi_\lambda$  see, e.g., \cite[Definition~1.1]{MR3604115}} Indeed, if this is the case then Theorems~2.1 and 2.9 of \cite{MR3604115} would yield $\mathrm{Gal}(\mathscr{D}_\lambda,\lambda,2^\lambda)$, hence forming  the desired contradiction. Before proving the next claim observe that $2^\kappa>\kappa^+$, for otherwise $\mathrm{Gal}(\mathscr{D}_{\kappa^+},\kappa^+,2^{\kappa^+})$ would hold. 
\begin{claim}
$2^\theta=2^\kappa$, for all $\theta\in [\kappa,2^\kappa)$.
\end{claim}
\begin{proof}[Proof of claim]
It suffices to check that $2^\theta\leq 2^\kappa$. Suppose otherwise, and let $\theta\in (\kappa,2^\kappa)$ be the least regular cardinal such that $2^\theta>2^\kappa$. Note that such $\theta$ exists as $2^\kappa>\kappa^+$. If $\theta=\delta^+$ then of minimality $\theta$ yields $2^\delta<2^\theta$ and so $\Phi_\theta$ holds \cite[\S1]{MR3604115}. As explained in the previous paragraph this is impossible. 
For the other case, minimality yields $2^{<\theta}=2^\kappa<2^\theta$ and so \cite[Theorem~2.9]{MR3604115} yields yet again 
o
$\mathrm{Gal}(\mathscr{D}_\theta,\theta,2^\theta)$.
\end{proof}
We are now in conditions to prove the proposition. Put $\lambda:=2^\kappa$. If $\lambda$ takes the form $\lambda=\theta^+$ then the above claim yields $2^\theta=\lambda=\theta^+<2^\lambda$. Hence, $\Phi_{\lambda}$ holds. Otherwise, $\lambda$ is a limit  cardinal such that $\langle 2^\theta\mid \theta<\lambda\rangle$ is eventually constant and thus results from \cite[\S7.1]{MR0469756} entail $\Phi_\lambda$.\footnote{For a explicit proof see \cite[\S2]{BenGarHay}.}  In any of the previous cases we achieve the desired contradiction.
\end{proof}
The previous proposition shows that class many gaps are inevitable if one wishes to force the ultimate failure of Galvin's property. In particular, any global theorem involving the method of Abraham and Shelah is doomed to miss class many instances of $\neg\mathrm{Gal}(\mathscr{D}_\kappa,\kappa,2^\kappa)$. Nevertheless, the property $\neg \mathrm{Gal}(\mathscr{D}_\kappa,\kappa,2^\kappa)$ is forceable at finitely-many consecutive cardinals and, as a result, the size of the gaps in Theorem~\ref{thmglobal} seem to be improvable. We intend to develop this issue in a subsequent work.

\subsection{Galvin's number at successor of singular cardinals}\label{SectionGalvinsNumber}
After esta\-blishing the consistency of many instances for the failure of Galvin's proper\-ty, the time is ripe to discuss 
the cardinal 
characteristic called the \emph{Galvin number}. 
This notion goes back to \cite{MR3787522} and reads as follows: 

\begin{definition}
\label{defgalnumber} For an infinite cardinal $\kappa$ the \emph{Galvin number at $\kappa$}, $\mathfrak{gp}_\kappa$, is 
$$\mathfrak{gp}_\kappa:=\min\{\lambda\mid \text{$\mathrm{Gal}(\mathscr{D}_{\kappa^+}, \kappa^+, \lambda^+)$ holds}\}.$$
\end{definition}

An outright consequence of Theorem \ref{thmmt} 
is the consistency (with \textsf{ZFC}) of
$$\text{$``\kappa$ is supercompact and $\mathfrak{gp}_\kappa=\lambda$'',}\footnote{Actually, in the model of Theorem~\ref{thmmt}$(\aleph)$, $\mathfrak{gp}_\kappa=2^{\kappa^+}=2^\kappa$.}$$
for
$\lambda\geq\cf(\lambda)>\kappa^+$. In particular, $\mathfrak{gp}_\kappa$ might be singular of cofinality greater than $\kappa^+$. Likewise, this situation can be arranged for a strong limit singular cardinal $\kappa$ of any prescribed cofinality (Theorem~\ref{thmmt}$(\beth)$). 

\smallskip

It has been shown in \cite{ghhm} that if $\kappa$ is  regular then  the cofinality of $\mathfrak{gp}_\kappa$ cannot be \emph{too small}.\footnote{For instance, if $\kappa=\aleph_3$ then \cite[Theorem~2.1]{ghhm} shows that $\cf(\mathfrak{gp}_\kappa)\geq \omega_1$.}
However, we do not know, e.g., if $\cf(\mathfrak{gp})=\omega$ is consistent. In case $\kappa>\aleph_0$  it is also open whether $\cf(\mathfrak{gp}_\kappa)<\kappa$ is consistent. Despite of this, one can still prove that the difference between $\cf(\mathfrak{gp}_\kappa)$ and $\kappa$ is relatively small  provided, of course, that $\cf(\mathfrak{gp}_\kappa)<\kappa$ is possible at all. 

\smallskip

In the strong limit singular scenario we can do better:

\begin{theorem}
\label{propcofgp} Let $\kappa$ be a strong limit singular cardinal. Then,
\begin{enumerate}
\item [$(\aleph)$] $\cf(\mathfrak{gp}_\kappa)>\kappa$.
\item [$(\beth)$] It is consistent, with a supercompact cardinal, that $\cf(\mathfrak{gp}_\kappa)=\kappa^+$. 
\end{enumerate}
\end{theorem}
\begin{proof}
Let us commence with $(\aleph)$.
Put $\lambda:=\mathfrak{gp}_\kappa$ and  $\sigma:=\cf(\lambda)$, and
assume that $\sigma<\kappa$. Let $\mathfrak{a}=\langle \lambda_i\mid i<\sigma\rangle$ be a sequence of regular cardinals with:
\begin{itemize}
\item $\sup \mathfrak{a}=\lambda$;
\item ${\rm tcf}(\prod\mathfrak{a},J)=\lambda^+$, where $J:=J^{\rm bd}_\sigma$.
\end{itemize}
From \cite{MR1318912}  one can choose a scale $\bar{f}=\langle f_\gamma\mid \gamma<\lambda^+\rangle$ in $(\prod\mathfrak{a},J)$
such that every $\beta\in S^{\lambda^+}_{\kappa^+}$ is a good point of $\bar{f}$.\footnote{For a concrete reference, see \cite[Lemma~2.12, Theorem~2.13]{AbrMag}.} This can be arranged since $2^\sigma<\kappa$.
Consider an arbitrary set $A\subseteq\lambda^+$ so that $|A|=\kappa^+$ and put $$T_A:=\{f_\gamma(i)\mid \gamma\in A,\,i<\sigma\}.$$
\begin{claim}
$|T_A|=\kappa^+$.
\end{claim}
\begin{proof}[Proof of claim]
By thinning-out $A$ to some $A'\in [A]^{\kappa^+}$  we may assume that for all sufficiently large $i<\sigma$, if $\gamma,\delta\in A'$ and $\gamma\neq\delta$ then $f_\gamma(i)\neq f_\delta(i)$. In effect, if $\beta:=\sup(A)$ then, without loss of generality, $\beta\in S^{\lambda^+}_{\kappa^+}$. If $B\s \beta$ witnesses the fact that $\beta$ is a good point for $\bar{f}$ then one can thin-out $A$ in such a way that the elements of $A$ and $B$ are interleaved. Now, if $i<\sigma$ is sufficiently large then for distinct $\gamma,\delta\in A$ we will have $f_\gamma(i)\neq f_\delta(i)$. For more details on this vein see \cite[Theorem~1.4]{ghhm}.
\end{proof}

Since  $\mathfrak{gp}_\kappa=\lambda$ we can fix a family $\mathcal{D}=\langle D_\alpha\mid \alpha<\lambda\rangle$ of clubs at $\kappa^+$ with
the property that  every $\mathcal{E}\s \mathcal{D}$ with $|\mathcal{E}|=\kappa^+$ has intersection of size $\leq \kappa$.
For each $\gamma<\lambda^+$, let $C_\gamma:=\bigcap\{D_{f_\gamma(i)}\mid i<\sigma\}$.
Note that $C_\gamma$ is a club of $\kappa^+$.

\smallskip

Let $\mathcal{C}:=\langle C_\gamma\mid \gamma\in\lambda^+\rangle$.
Choose $A\in [\lambda^+]^{\kappa^+}$ with $C:=\bigcap_{\gamma\in A}C_\gamma$ being a club at $\kappa^+$.
This choice is possible since $\mathfrak{gp}_\kappa=\lambda$.
Since $|T_A|=\kappa^+$ we see that $D=\bigcap_{\beta\in T_A}D_{\beta}$ is bounded in $\kappa^+$.
However, $D=C$, which is an absurd.
One concludes, therefore, that $\cf(\mathfrak{gp}_\kappa)\neq\sigma$ for $\sigma<\kappa$.
But then $\cf(\mathfrak{gp}_\kappa)>\kappa$, because $\kappa$ is singular. 
This completes the verification of $(\aleph)$. 

\smallskip

We are thus left with the verification of Clause~$(\beth)$. Let $\kappa$ be a cardinal and assume the \textsf{GCH} holds above $\kappa$ (see~\cite[\S8]{BP}). By preparing our cardinal $\kappa$ \emph{á-la-Laver} we can assume it is indestructible under $\kappa$-directed-closed forcing.  Fix $\lambda>\cf(\lambda)=\kappa^+$. Let $G\subseteq\mathbb{S}(\kappa,\lambda)$ be a $V$-generic filter. We note that $2^{\kappa^+}=\lambda^+$ holds in $V[G]$:  In effect, on one hand, by counting nice names we have 
$2^{\kappa^+}\leq\lambda^+$. On the other hand, by K\"{o}nig's theorem  $2^{\kappa^+}>\lambda$, as  $\cf(\lambda)=\kappa^+$. Similarly, it is clear that $2^{\kappa}\geq\lambda$. Finally, invoking \cite[Lemma 1.7]{MR830084}, we know that every subset of $\kappa$ is introduced by a $\kappa^+$-cc subforcing of $\mathbb{S}(\kappa,\lambda),$ hence counting-nice-names arguments yield $2^\kappa=\lambda$.

\smallskip

Let $\mathcal{C}=\langle C_\alpha\mid \alpha<\lambda\rangle$ be the clubs at $\kappa^+$ added by $\mathbb{S}(\kappa,\lambda)$.
It follows from \cite{MR830084} that $\neg_{\mathrm{st}}\mathrm{Gal}(\mathscr{D}_{\kappa^+},\kappa^+,\lambda)$ holds 
and hence $\mathfrak{gp}_\kappa\geq\lambda$.
On the other hand, $2^\kappa=\lambda<\lambda^+=2^{\kappa^+}$, so $\mathfrak{gp}_\kappa\leq\lambda$ as proved in \cite{MR3604115}.
Putting these two facts together we see that $\mathfrak{gp}_\kappa=\lambda$, which yields $\cf(\mathfrak{gp}_\kappa)=\kappa^+$.
To accomplish the construction force with either Prikry or Magidor/Radin forcing towards singularizing $\kappa$. Note that to ensure $\cf(\mathfrak{gp}_\kappa)=\kappa^+$ prevails in the resulting generic extension one needs to appeal to Lemma~\ref{lemchain}. 
\end{proof}

The last part of the theorem shows that $\cf(\mathfrak{gp}_\kappa)$ can be fully characterized at strong limit singular cardinals $\kappa$:

\begin{corollary}
\label{corcofgp} 
Let $\kappa<\lambda$ be cardinals, being $\kappa$ a supercompact. 

Then, $\cf(\lambda)>\kappa$ if and only if  in  some cardinal preserving and $\kappa^{++}$-cc forcing extension $W$ the following holds:
$$W\models \text{``$\mathfrak{gp}_\kappa=\lambda$ $\wedge$ $\kappa$ is strong limit singular''}.$$
\end{corollary}
\section{The strong failure and Prikry-type generic extensions}\label{SectionTheStrongFailure}
The main result of \cite{MR830084} says that forcing with $\mathbb{S}(\kappa,\lambda)$ produces a model of \textsf{ZFC} exhibiting a quite strong form of  failure of Galvin's property. Namely, the said model contains a family $\mathcal{C}$ of clubs at $\kappa^+$, $|\mathcal{C}|=\lambda$, that witnesses the following: For every  subfamily $\mathcal{D}\s \mathcal{C}$ with $|\mathcal{D}|=\kappa^+$ then $|\bigcap \mathcal{D}|<\kappa$. Note that when this is relativized, e.g., to $\kappa=\aleph_0$ one obtains a $\aleph_2$-sized family $\mathcal{C}$ of clubs at $\aleph_1$ such that every uncountable subfamily of it has finite intersection. Instead, if $\mathcal{C}$ was just a mere witness for $\neg \mathrm{Gal}(\mathscr{D}_{\aleph_1},\aleph_1,\aleph_2)$ this latter might contain uncountable subfamilies whose intersection is infinite.

\begin{definition}[Strong failure of Galvin's property]
Let $\kappa\geq \omega$ be a cardinal  and $ \mu\leq \lambda\leq 2^\kappa$.  Denote by $\neg_{\mathrm{st}}{\rm Gal}(\mathscr{D}_{\kappa^+},\mu,\lambda)$ the statement:
$$\text{``There is $\mathcal{C}\s  \mathscr{D}_{\kappa^+}$ with $|\mathcal{C}|=\lambda$ such that for all $\mathcal{D}\in[\mathcal{C}]^\mu$ then  $|\cap\mathcal{D}|<\kappa$.''}$$
 \emph{Galvin's property strongly fails at $\kappa^+$} if $\neg_{\mathrm{st}}\mathrm{Gal}(\mathscr{D}_{\kappa^+},\kappa^{+},\kappa^{++})$ holds.
\end{definition}
Using the above terminology Abraham-Shelah theorem can be rephrased in a more compact way as follows: For every regular cardinal $\kappa$ and $\lambda\geq \cf(\lambda)>\kappa^+$,  $\neg_\mathrm{st}\mathrm{Gal}(\mathscr{D}_{\kappa^+},\kappa^+,\lambda)$ is consistent with $\textsf{ZFC}$. In this respect, the forthcoming proposition shows that the model of \cite[\S1]{MR830084} provides an optimal upper bound for the size of the intersections of the  family of clubs:
\begin{proposition}
\label{clmverystrong} Let $\kappa$ be regular and  $\mu$ be singular.
\begin{enumerate}
\item [$(\aleph)$] Let $\lambda\geq\cf(\lambda)>\kappa^{++}$. For every family $\mathcal{C}$ of clubs at $\kappa^{++}$, $|\mathcal{C}|=\lambda$, there is $\mathcal{D}\s \mathcal{C}$ with $|\mathcal{D}|=\lambda$ such that $|\bigcap \mathcal{D}|\geq \kappa$;

\item [$(\beth)$] Assume that $\lambda\geq\cf(\lambda)>\mu^+$ and that $\mathcal{C}$ witnesses  $\neg \mathrm{Gal}(\mathscr{D}_{\mu^+},\mu^+, \lambda)$. For all $\theta<\mu$ there is $\mathcal{D}\s \mathcal{C}$, $|\mathcal{D}|=\lambda$ such that $|\bigcap\mathcal{D}|\geq \theta;$
\item [$(\gimel)$] Assume that $\kappa$ is a weakly inaccessible cardinal. Then, the conclusion of Clause~$(\beth)$ holds when replacing $\mu$ by $\kappa$.
\end{enumerate}
\end{proposition}

\begin{proof} 
$(\aleph)$: Let $\mathcal{C}=\langle C_\alpha\mid \alpha<\lambda\rangle$. Put $S:=S^{\kappa^{++}}_\kappa$ and let $\langle t_\gamma\mid \gamma\in S\rangle$ be a club guessing sequence (see \cite[Theorem~2.17]{AbrMag}).
For every $\alpha<\lambda$ choose $\gamma(\alpha)\in S$ for which $t_{\gamma(\alpha)}\subseteq C_\alpha$.
By the pigeon-hole principle there are $B\in[\lambda]^\lambda$ and $\gamma\in S$ such that if $\beta\in B$ then $\gamma(\beta)=\gamma$.
Finally, setting $\mathcal{D}:=\langle C_\alpha\mid \alpha\in B\rangle$ we have that $t_\gamma\s \bigcap \mathcal{D}$, and we are done.

$(\beth)$: Fix $\theta<\mu$ and let $\langle t_\gamma\mid \gamma\in S^{\mu^+}_\theta\rangle$ be a club-guessing sequence.
For each $\alpha<\lambda$ let $\gamma(\alpha)\in S^{\mu^+}_\theta$ be such that $t_{\gamma(\alpha)}\subseteq C_\alpha$.
Consider the mapping $\alpha\mapsto\gamma(\alpha)$.
Since $\cf(\lambda)>\mu^+$ there must be a set $B\in[\lambda]^\lambda$ and a fixed ordinal $\gamma\in\mu^+$ such that $\beta\in B$ implies $ \gamma(\beta)=\gamma$. Setting  $\mathcal{D}:=\langle C_\alpha\mid \alpha\in B\rangle$ we have that $t_\gamma\s \bigcap\mathcal{D}$, and we are done.

$(\gimel)$: Use a club-guessing sequence sitting on $S^\kappa_\theta$, for  $\theta=\cf(\theta)<\kappa$. 
\end{proof}

An outright consequence of Clause~$(\aleph)$ above is that the failure of Galvin's property at $\aleph_2$ cannot yield finite intersections of the subfamilies $\mathcal{D}$.

\smallskip

Throughout this section we discuss the consistency (with $\textsc{ZFC}$) of the principle $\neg_{\mathrm{st}}\mathrm{Gal}(\mathscr{D}_{\kappa^+},\kappa^{+},\kappa^{++})$ when $\kappa$ is a singular cardinal. As we demons\-trate, the situation  is quite much more challenging --and intriguing-- than it used to be in the regular context. 
For instance, while our first result shows that $\neg_{\mathrm{st}}\mathrm{Gal}(\mathscr{D}_{\kappa^+},\kappa^{+},\kappa^{++})$ is strictly stronger than $\neg\mathrm{Gal}(\mathscr{D}_{\kappa^+},\kappa^{+},\kappa^{++})$, it is not known to the authors whether $\neg_{\mathrm{st}}\mathrm{Gal}(\mathscr{D}_{\kappa^+},\kappa^{+},\kappa^{++})$ fails for all strong limit singular cardinals $\kappa$. This sort of expectations are --in principle-- plausible, in that singular cardinals are known to behave quite differently than its regular relatives.  In what follows, $\kappa$ will denote a measurable cardinal with sufficiently high Mitchell order (see e.g., \cite[Definition~2.4]{MitchHandbook}).

\smallskip

We commence showing that after forcing with Magidor/Radin poset with respect to a long enough coherent/measure sequence one gets a model where $\neg_{\mathrm{st}}\mathrm{Gal}(\mathscr{D}_\kappa,\kappa^{++},\kappa^{++})$ cannot hold. Although Magidor/Radin forcings have not been presented in detail in this paper both  can be seen as a  \emph{toy-version} of the main forcing of Section~\ref{SubsectionRadinwithcollapses} (for details, see \cite[\S5]{Gitik-handbook}). In particu\-lar, when referring to these forcings we shall stick to the notations and terminologies settled in Subsection~\ref{SubsectionRadinwithcollapses}.

\begin{proposition}
\label{propmagidor} Let $\mathbb{R}_u$ be  Magidor/Radin forcing with respect to a \linebreak coherent/measure sequence $u$. Let $G\s \mathbb{R}_u$ a generic filter over $V$.

Suppose that $\mathcal{C}\in V[G]$  is a collection of clubs at $\kappa^+$ with $|\mathcal{C}|=\kappa^{++}$.
Then there is a $1$-$1$ function $\varphi\colon \kappa\rightarrow\kappa^+$, $\varphi\in{V}$, $\mathcal{D}\in [\mathcal{C}]^{\kappa^{++}}$ and $\xi<\kappa$ such that  $$\varphi``(C-\xi)\subseteq\bigcap \mathcal{D}.\footnote{Here $C$ denotes the Magidor/Radin club induced by $G$ (cf. Definition~\ref{RadinObject}).}$$
In particular, $\neg_{\mathrm{st}}\mathrm{Gal}(\mathscr{D}_{\kappa^+},\kappa^{++},\kappa^{++})$ does not hold, provided $\len(u)\geq \kappa$.
\end{proposition}
\begin{proof}
Suppose $\mathcal{C}=\langle C_\alpha\mid \alpha<\kappa^{++}\rangle.$
For each $\alpha<\kappa^{++}$ let $D_\alpha\in V$ be a club at $\kappa^+$ such that $D_\alpha\s C_\alpha$.\footnote{Note that this choice is possible because $\mathbb{R}_u$ is $\kappa^+$-cc.} Notice that $\langle D_\alpha\mid \alpha<\kappa^{++}\rangle$ may not belong to the ground model $V$, even though each $D_\alpha$ do. Put $S:=(S^{\kappa^+}_\kappa)^V.$ For each $\alpha<\kappa^{++}$, let $\beta_\alpha\in \mathrm{acc}(C_\alpha)\cap S$. By the pigeon-hole principle there are $I\in [\kappa^{++}]^{\kappa^{++}}$ and $\beta\in S$ such that $\beta_\alpha=\beta$, for all $\alpha\in I$. Choose an increasing continuous function $\varphi\colon \kappa\rightarrow\beta$ in $V$.
We claim that by further shrinking $I$ the map $\varphi$ witnesses the statement of the proposition. In effect, for each $\alpha\in I$ the set $D_\alpha\cap \beta$ defines a club in $\beta$ that belongs to $V$. In particular, $E_\alpha:=\varphi^{-1}(D_\alpha\cap \beta)\in V$ is a club in $\kappa$, for each $\alpha\in I$. Hence, $\langle E_\alpha\mid \alpha\in I\rangle\s \mathscr{F}_u$. Combining Proposition~\ref{Prelimminaryforclub}(1) with the pigeon-hole principle we can find $J\in[I]^{\kappa^{++}}$ and $\xi<\kappa$ such that $C-\xi\s E_\alpha$, for all $\alpha\in I$. Letting $\mathcal{D}:=\langle C_\alpha\mid \alpha\in J\rangle$ it follows that $\varphi``(C-\xi)\s \bigcap\mathcal{D}$.

For the in particular clause:  If $\len(u)\geq \kappa$ then $|C|=\kappa$ (see \cite[\S5]{Gitik-handbook}). Hence,  $\varphi$ being $1$-$1$ implies that $|\varphi``(C-\xi)|=\kappa$, as desired.
\end{proof}

The following corollary is immediate at the light of the above proposition:

\begin{corollary}\label{CorStrongNegationSuper}
Assume that $\kappa$ is supercompact. Then the following hold:
\begin{enumerate}
\item [$(\aleph)$] 
For every $\mu=\cf(\mu)$ there is a generic extension where $\kappa>\cf(\kappa)=\mu$ is strong limit cardinal and both
$$\text{${\rm Gal}(\mathscr{D}_{\kappa^+},\kappa^{+},\kappa^{++})$ and $\neg_{\mathrm{st}}{\rm Gal}(\mathscr{D}_{\kappa^+},\kappa^{+},\kappa^{++})$ fail.}$$
\item [$(\beth)$] There is a generic extension where $\kappa$ is measurable  and both
$$\text{${\rm Gal}(\mathscr{D}_{\kappa^+},\kappa^{+},\kappa^{++})$ and $\neg_{\mathrm{st}}{\rm Gal}(\mathscr{D}_{\kappa^+},\kappa^{+},\kappa^{++})$ fail.}$$
\end{enumerate}
Hence, the failure of Galvin's property does not entail the strong failure.
\end{corollary}
\begin{proof}
By Theorem~\ref{thmmt} one can force $\neg \mathrm{Gal}(\mathscr{D}_{\kappa^+},\kappa^+,\kappa^{++})$ while preserving the supercompactness of $\kappa$. Denote the resulting model by $V$. Let $G\s \mathbb{R}_u$ a generic filter over $V$ and force with the Magidor/Radin forcing $\mathbb{R}_u$. Since $\mathbb{R}_u$ is   $\kappa^+$-cc the principle $\mathrm{Gal}(\mathscr{D}_{\kappa^+},\kappa^+,\kappa^{++})$ fails in $V[G]$ (see Lemma~\ref{lemchain}).

To obtain a model for $(\aleph)$, one uses a sequence $u$ such that $\len(u)=\kappa+\mu$. Likewise, to obtain  a model for $(\beth)$, one picks $u$ having a repeat point.\footnote{For an extensive discussion on the cofinality and large-cardinal properties of $\kappa$ in Magidor/Radin-like generic extensions, cf. \cite[\S5]{MR2768695})}
\end{proof}
\begin{corollary}
There are
$\kappa^+$-cc forcings which destroy  $\neg_{\mathrm{st}}\mathrm{Gal}(\mathscr{D}_{\kappa^+},\kappa^+,\lambda)$.
\end{corollary}
\begin{remark}
Recall that the model of Theorem~\ref{thmglobal} is a generic extension by a Radin-like forcing with respect to a measure sequence $u$ with $\len(u)\geq \kappa^+$. Every singular cardinal in this model is a limit point $\kappa_v$ of the Radin club $C$ and  many of them correspond to measure sequences $v$ with $\len(v)\geq \kappa_v$. In particular, for all these $v\in \mathcal{U}_\infty$ the principle $\neg_{\mathrm{st}}\Gal{\kappa_v^+}{\kappa^{++}_v}$ fails. 
\end{remark}

At this point it would be interesting to understand whether a similar property holds in Prikry extensions.
It turns out that the answer is negative.
In order to understand the difference between Prikry and Magidor/Radin generic extensions, note that Proposition~\ref{propmagidor}  produces a set of size $\kappa$ (i.e., $\varphi``(C-\xi)$) contained in $\kappa^{++}$-many clubs.
Moreover, these latter come from an arbitrary family of clubs at $\kappa^+$.
The following theorem shows that a parallel phenomenon is impossible in the context of Prikry forcing:
\begin{theorem}
\label{thmprikryweak} Let $\mathscr{U}$ be a normal ultrafilter over $\kappa$ and let $\mathbb{P}(\mathscr{U})$ be the associated Prikry forcing.
Let $G\subseteq\mathbb{P}(\mathscr{U})$ be $V$-generic and suppose that $A\in{V[G]}$ is almost contained in every ground model club of $\kappa$.\footnote{I.e, for every $C\s \kappa$ club at $\kappa$ and $C\in V$  there is some $\xi<\kappa$ such that $A-\xi\s C$. }  

Then, there exists $\xi<\kappa$ such that $A-\xi\subseteq\langle \rho_n\mid n<\omega\rangle$, where $\langle \rho_n\mid n<\omega\rangle$ denotes the Prikry sequence derived from $G$.
In particular, $|A-\xi|\leq\aleph_0$.
\end{theorem}
\begin{proof}
Assume towards contradiction that $A$ is almost contained in every $V$-club at $\kappa$ but $A\setminus \langle \rho_n\mid n<\omega\rangle$ is unbounded in $\kappa$.
Let $p\in{G}$ and a collection of $\mathbb{P}(\mathscr{U})$-names $\{\lusim{a}_n\mid n<\omega\}$ such that, for each $n<\omega$, $$p\Vdash_{\mathbb{P}(\mathscr{U})}\lusim{a}_n\in \lusim{A}-\langle \lusim{\rho}_m\mid m<\omega\rangle .$$
For each $n<\omega$, the set $D_n:=\{q\in\mathbb{P}(\mathscr{U})\mid q\parallel\lusim{a}_n\}$ is dense open. 
We shall apply the so-called \emph{Strong Prikry property} to the collection of sets $\{D_n\mid n<\omega\}$ (see \cite[Lemma~1.13]{Gitik-handbook}). 
Let $p\leq q$ be any condition, say, with stem $t$.
For each $n<\omega$, the Strong Prikry property yields a further condition $q\leq^*p_n=(t,A_n)$ and $m_n<\omega$ such that $p_n^
\frown\bar{\alpha}\in{D_n}$,  for every $\bar{\alpha}\in[A_n]^{m_n}$.
By the very definition of $D_n$, the condition $p_n^\frown\bar{\alpha}$ decides the value of $\lusim{a}_n$.
Therefore, one can define a function $f_n:[A_n]^{m_n}\rightarrow\kappa$ such that $$p_n^\frown\bar{\alpha}\Vdash_{\mathbb{P}(\mathscr{U})} f_n(\bar{\alpha})=\lusim{a}_n.$$

For every  $\bar{\alpha}=(\alpha_1,\ldots,\alpha_{m_n})\in[A_n]^{m_n}$ and every $1\leq{j}\leq{m_n}$ we know that $p_n^\frown \bar{\alpha}\Vdash_{\mathbb{P}(\mathscr{U})}\alpha_j\neq\lusim{a}_n$, since $p\leq p_n$ and $p\Vdash_{\mathbb{P}(\mathscr{U})}\lusim{a}_n\notin\langle\rho_m\mid m<\omega\rangle$.
Hence $f_n(\bar{\alpha})$ falls in one of the open intervals $I_\ell$, where
$$
I_{\ell}:=\begin{cases}
[0,\alpha_1), & \text{if $\ell=0$;}\\
(\alpha_\ell,\alpha_{\ell+1}), & \text{if $\ell<m_n$;}\\
(\alpha_{m_n},\kappa), & \text{otherwise.}
\end{cases}
$$
Let us write $\ell_{\bar{\alpha}}:=\ell$ whenever $f_n(\bar{\alpha})\in{I_\ell}$.
Apply Rowbottom's theorem (\cite[Theorem~7.17]{Kan}) to obtain, for each $n<\omega$, a fixed natural number $\ell_n$ and a set $A'_n\subseteq{A_n}, A'_n\in\mathscr{U}$ such that $\ell_n=\ell_{\bar{\alpha}}$ for every $\bar{\alpha}\in[A_n]^{m_n}$.

\smallskip

Now we shrink $A'_n$ further, as follows.
Suppose that $\ell_n<m_n$.
For every $\ell_n$-tuple $(\alpha_1,\ldots,\alpha_{\ell_n})$ from $A'_n$ we already know that any extension of $(\alpha_1,\ldots,\alpha_{\ell_n})$ by another sequence of ordinals $(\alpha_{\ell_{n+1}},\ldots,\alpha_{m_n})$ from $A'_n$ would give $\alpha_{\ell_n}<f_n((\alpha_1,\ldots,\alpha_{\ell_n})^\frown (\alpha_{\ell_{n+1}},\ldots,\alpha_{m_n}))<\alpha_{\ell_{n+1}}$  since this ordinal falls in the interval $I_{\ell_n}$. Fix $(\alpha_1,\dots, \alpha_{\ell_n})\in [A'_n]^{\ell_n}$.  Now note that $f_n((\alpha_1,\ldots,\alpha_{\ell_n})^\frown (\alpha_{\ell_n+1},\ldots,\alpha_{m_n}))$ is a regressive function and $A'_n\in\mathscr{U}$, hence there is  $A_{(\alpha_1,\ldots,\alpha_{\ell_n})}\subseteq A'_n$,  $A_{(\alpha_1,\ldots,\alpha_{\ell_n})}\in\mathscr{U}$ such that $f_n(\alpha_1,\ldots,\alpha_{m_n})$  is constant on $A_{(\alpha_1,\dots,\alpha_{\ell_n})}$. 
Now, let us \emph{diagonalize} all these sets by putting $$A''_n:=\diagonal\{A_{(\alpha_1,\ldots,\alpha_\ell)}\mid (\alpha_1,\ldots,\alpha_{\ell_n})\in[A'_n]^{\ell_n}\}.$$ The main feature of $A''_n$ is the  following: for every $(\alpha_1,\dots,\alpha_{m_n})\in [A''_n]^{m_n}$ the value of $f_n((\alpha_1,\ldots,\alpha_{m_n}))$ just depends on $(\alpha_1,\dots,\alpha_{\ell_n})$. More verbosely,  $f_n((\alpha_1,\dots, \alpha_{m_n}))=g_n(\alpha_1,\dots,\alpha_{\ell_n})$, for some $g_n\colon [A''_n]^{\ell_n}\rightarrow \kappa.$\footnote{Note that, for each $\bar{\alpha}\in [A''_n]^{\ell_n}$,  $g_n(\bar{\alpha})>\max(\bar{\alpha}).$}

\smallskip

The above process is rendered for every $n<\omega$ and, as an outcome, we  define $B:=\bigcap\{A''_n\mid n<\omega\}$. Clearly, $B\in\mathscr{U}$.
Recall that $t$ was the common stem of each $p_n$, and let $q:=(t,B)$.
Notice that $p_n\leq^*q$ for every $n<\omega$.
The crucial property of $q$ is that if $\bar{\alpha}\in[B]^{\ell_n}$ then $q^\frown\bar{\alpha}\Vdash_{\mathbb{P}(\mathscr
{U})}\lusim{a}_n>\max(\bar{\alpha})$; this being true for every $n<\omega$.
Hence for every $n<\omega$ we have defined functions $g_n\colon[B]^{\ell_n}\rightarrow\kappa$ such that $q^\frown\bar{\alpha}\Vdash_{\mathbb{P}(\mathscr{U})}\lusim{a}_n=g_n(\bar{\alpha})>\max(\bar{\alpha})$.
Note that we may certainly assume that $q\in{G}$, and therefore the shape of $q$ is $(\rho_0,\ldots,\rho_N,B)$ for some $N<\omega$.
Working now in $V$, define $g\colon \kappa\rightarrow\kappa$ by:
\[g(\alpha):=\sup_{n<\omega}\{g_n(\alpha_1,\ldots,\alpha_{\ell_{n-1}},\alpha)\mid \alpha_1,\ldots,\alpha_{\ell_{n-1}}\in B\cap\alpha\}+1.\]
Notice that $g(\alpha)<\kappa$ for every $\alpha<\kappa$, for $\kappa$ is a  regular cardinal  in $V$.
Next, let $C_g$ be the closure points of $g$, that is $C_g:=\{\delta<\kappa\mid g``\delta\s \delta\}$.
Clearly, $C_g\in{V}$ and it is a club at $\kappa$.
By the assumptions of the theorem there exists an integer $m\geq{N}$ such that if $n\geq{m}$ then $a_n\in{C_g}$.

Fix $n\geq{m}$ and consider the condition $q^\frown(\rho_{N+1},\ldots,\rho_{N+\ell_n})\in{G}$.
On the one hand, $q^\frown(\rho_{N+1},\ldots,\rho_{N+\ell_n})\Vdash_{\mathbb{P}(\mathscr{U})} \lusim{a}_n=g_n(\rho_{N+1},\ldots,\rho_{N+\ell_n})>\rho_{N+\ell_n}$, hence $a_n=g_n(\rho_{N+1},\ldots,\rho_{N+\ell_n})>\rho_{N+\ell_n}$. On the other hand, $a_n\in{C_g}$ and since $\rho_{N+\ell_n}<a_n$ it follows that $g(\rho_{N+\ell_n})<a_n$. So by the definition of $g$ we see that $g_n(\rho_{N+1},\ldots,\rho_{N+\ell_n})\leq g(\rho_{N+\ell_n})<a_n$. This produces the desired contradiction.
\end{proof}

An interesting upshot of the above theorem is the next strengthening of the Mathias criterion proved in \cite{Mathias}:
Suppose that $\mathscr{U}$ is a normal ultrafilter over $\kappa$, $\mathbb{P}(\mathscr{U})$ is Prikry forcing with $\mathscr{U}$ and $G\subseteq\mathbb{P}(\mathscr{U})$ is $V$-generic.
Mathias proved that 
 $\langle \rho_n\mid n\in\omega\rangle$ is a Prikry sequence for $\mathbb{P}(\mathscr{U})$ over $V$ if and only if it is almost contained in every element of $\mathscr{U}$.
\begin{corollary}
\label{cormathiascriterion} Let $\mathscr{U},\mathbb{P},G$ and $\langle \rho_n\mid n\in\omega\rangle\in V[G]$ be as above. Then, $\langle \rho_n\mid n\in\omega\rangle$ is a Prikry sequence iff for every club $C\subseteq\kappa$, $C\in{V}$ there exists $m<\omega$ such that $\rho_n\in{C}$ whenever $n\geq{m}$.
\end{corollary}
\begin{proof}
For the forward direction use the fact that if $C$ is a club at $\kappa$ then $C\in\mathscr{U}$, as $\mathscr{U}$ is normal.
Hence if $\langle \rho_n\mid n\in\omega\rangle$ is a Prikry sequence then it is almost contained in $C$.
For the opposite direction notice that if $\langle \rho_n\mid n\in\omega\rangle$ is almost contained in every ground model club of $\kappa$ then for some $m<\omega$ one has that $\langle \rho_n\mid n\geq m\rangle$ is an infinite subsequence of the Prikry sequence associated with $G$ (Theorem~\ref{thmprikryweak}).
From the Mathias criterion \cite{Mathias} it follows now that $\langle \rho_n\mid n\in\omega\rangle$ is almost contained in every element of $\mathscr{U}$, hence it is a Prikry sequence over $V$ and we are done.
\end{proof}

Although Theorem~\ref{thmprikryweak} unveils notable differences between Prikry and Magidor/Radin extensions it does not rule out the possibility that Prikry forcing  yields the consistency of $\neg_{\mathrm{st}}\mathrm{Gal}(\mathscr{D}_{\kappa^+},\kappa^+,\kappa^{++}).$ A natural strategy to produce such a situation will be to begin with  $\neg_{\mathrm{st}}\mathrm{Gal}(\mathscr{D}_{\kappa^+},\kappa^+,\kappa^{++})$ at a measurable cardinal $\kappa$ and afterwards argue that Prikry forcing does preserve the corresponding witnessing family $\mathcal{C}$. Unfortunately, as the next theorem reveals, this strategy is not going to work in general and a specific construction of a witness is needed, for there are such witnessing families that are always destroyed.

\begin{theorem}
\label{thmprikrydestroys} Let $\mathcal{C}$ be a witness for $\neg_{\mathrm{st}}\mathrm{Gal}(\mathscr{D}_{\kappa^+},\kappa^+,\kappa^{++})$.
Then there exists another family of clubs $\mathcal{D}$ at $\kappa^+$, such that:
\begin{enumerate}
\item [$(\aleph)$] $\mathcal{D}$ is also a witness for $\neg_{\mathrm{st}}\mathrm{Gal}(\mathscr{D}_{\kappa^+},\kappa^+,\kappa^{++})$;
\item [$(\beth)$] For every  normal ultrafilter $\mathscr{U}$ over $\kappa$, forcing with $\mathbb{P}(\mathscr{U})$ yields a generic extension where $\mathcal{D}$ is not a witness for $\neg_{\mathrm{st}}\mathrm{Gal}(\mathscr{D}_{\kappa^+},\kappa^+,\kappa^{++})$.
\end{enumerate}
Moreover, in Clause~$(\beth)$ one can get the stronger statement saying that $\mathcal{D}$ ceases to be a witness for $\neg_{\mathrm{st}}\mathrm{Gal}(\mathscr{D}_{\kappa^+},\kappa^{++},\kappa^{++})$.
\end{theorem}
\begin{proof}
Put $\mathcal{C}=\langle C_\alpha\mid \alpha<\kappa^{++}\rangle$ and $S:=S^{\kappa^{++}}_\kappa$. 
Without loss of generality assume that there is some ordinal $\gamma$ such that $\gamma\in \mathrm{acc}(C_\alpha)\cap S$, for all $\alpha<\kappa^{++}.$  
Let $B:=\{\beta_\eta\mid \eta<\kappa\}$ be a club at $\gamma$ of order-type $\kappa$.
We translate each $C_\alpha\cap{B}$ to a club $E_\alpha$ at $\kappa$ by letting $E_\alpha:=\{\eta\in\kappa\mid \beta_\eta\in{C_\alpha}\}$. 

Let $C'$ be the club of limit cardinals ${<}\kappa$.
It will be useful to think of $C'$ as the set of closure points of the function $f\colon \kappa\rightarrow\kappa$ defined by $f(\alpha):=|\alpha|^+$.
Now we shrink each $E_\alpha$ by setting $E'_\alpha:=E_\alpha\cap{C'}$.
Notice that the sequence $\langle E'_\alpha\mid \alpha<\kappa^{++}\rangle$ witnesses $\neg_{\mathrm{st}}\mathrm{Gal}(\mathscr{D}_{\kappa^+},\kappa^+,\kappa^{++})$, as 
$$|\bigcap_{\alpha\in{I}}E'_\alpha|\leq|\bigcap_{\alpha\in{I}}E_\alpha|\leq |\bigcap_{\alpha\in{I}}(C_\alpha\cap{B})| \leq|\bigcap_{\alpha\in{I}}C_\alpha|<\kappa,$$
for all $I\in [\kappa^{++}]^{\kappa^+}$. For every $\alpha<\kappa^{++}$ let $F_\alpha:=\bigcup\{[\alpha,\alpha^+]\mid \alpha\in{E'_\alpha}\}$.
It is convenient to think of $F_\alpha$ as a thickening of $E'_\alpha$; that is, each point in $E'_\alpha$ gives rise to an interval in $F_\alpha$.
Notice that each $F_\alpha$ is a club at $\kappa$ and, as a result, $\{\beta_\eta\mid \eta\in{F_\alpha}\}$ is a club at $\gamma$.
We shall add the part of $C_\alpha$ above $\gamma$ to this set. Namely, for every $\alpha<\kappa^{++}$ we define:
\[D_\alpha:=\{\beta_\eta\mid \eta\in{F_\alpha}\}\cup(C_\alpha-\gamma).\]
Note that $D_\alpha$ is a club of $\kappa^+$. The next two claims will complete the proof:
\begin{claim}
$\langle D_\alpha\mid \alpha<\kappa^{++}\rangle$ witnesses Clause~$(\aleph)$.
\end{claim}
\begin{proof}[Proof of claim]
If $\alpha_0,\alpha_1\in{C'}$ and $\alpha_0<\alpha_1$ then $\alpha_0^+<\alpha_1$, for $\alpha_1$ is a limit cardinal. In particular,  $[\alpha_0,\alpha_0^+]\cap[\alpha_1,\alpha_1^+]=\varnothing$.
Now, fix $I\in[\kappa^{++}]^{\kappa^+}$.

Recall that $|\bigcap_{\alpha\in I}C_\alpha|<\kappa$ so, in particular, $|\bigcap_{\alpha\in I}(C_\alpha-\gamma)|<\kappa$.
However, $C_\alpha-\gamma=D_\alpha-\gamma$  and hence $|\bigcap_{\alpha\in I}(D_\alpha-\gamma)|<\kappa$.
Thereby, by proving that $|\bigcap_{\alpha\in I}(D_\alpha\cap\gamma)|<\kappa$ we will be done with the verification of Clause~$(\aleph)$.

\smallskip

By definition, $\bigcap_{\alpha\in I}(D_\alpha\cap\gamma)= \{\beta_\eta\mid \eta\in\bigcap_{\alpha\in{I}}F_\alpha\},$ so it suffices to prove that $|\bigcap_{\alpha\in I}F_\alpha|<\kappa$.
Fix an ordinal $\beta\in\bigcap_{\alpha\in I}F_\alpha$.
For every $\alpha\in{I}$ choose $\gamma_\alpha\in{E'_\alpha}$ such that $\beta\in[\gamma_\alpha,\gamma_\alpha^+]$.
Suppose that $\alpha_0,\alpha_1\in{I}$.
Since $\beta\in[\gamma_{\alpha_0},\gamma_{\alpha_0}^+]\cap[\gamma_{\alpha_1},\gamma_{\alpha_1}^+]$ we conclude 
that $\gamma_{\alpha_0}=\gamma_{\alpha_1}$.
Namely, there is a fixed cardinal $\gamma$ such that $\gamma_\alpha=\gamma$ for every $\alpha\in{I}$. It follows that $$\bigcap_{\alpha\in I}F_\alpha\subseteq \{[\gamma,\gamma^+]\mid \gamma\in\bigcap_{\alpha\in{I}}E'_\alpha\}.$$
Since the sequence $\langle E'_\alpha\mid \alpha\in\kappa^{++}\rangle$ witnesses the strong failure of the Galvin property we have that 
$|\bigcap_{\alpha\in I}F_\alpha|<\kappa$, as wanted.
\end{proof}


\begin{claim}
$\langle D_\alpha\mid \alpha<\kappa^{++}\rangle$ witnesses Clause~$(\beth)$.
\end{claim}
\begin{proof}[Proof of claim]
Let $\mathscr{U}$ be a normal ultrafilter over $\kappa$.
By normality, $E'_\alpha\in\mathscr{U}$ for every $\alpha<\kappa^{++}$.
Invoking the Mathias criterion for genericity, for each $\alpha<\kappa^{++}$ we choose $n_\alpha<\omega$ such that $n\geq{n_\alpha}$ implies $\rho_n\in{E'_\alpha}$, where $\langle \rho_n\mid n\in\omega\rangle$ is the Prikry sequence for $\mathbb{P}(\mathscr{U})$.
Choose $J\in [\kappa^{++}]^{\kappa^{++}}$ and $m<\omega$ so that  $n_\alpha=m$, for all $\alpha\in J$.
We claim that in $V[G]$
$$\{\beta_\eta\mid \eta\in\bigcup\{[\rho_n,\rho_n^+]\mid m\leq n<\omega\}\} \subseteq \bigcap_{\alpha\in J}(D_\alpha\cap\gamma).$$ 
This will show that $\langle D_\alpha\mid \alpha<\kappa^{++}\rangle$ witnesses Clause~$(\beth)$, for the left-hand-side set is of cardinality $\kappa$.

So fix $\alpha\in{J}$ and $n\geq{m}$.
Since $n_\alpha=m$ one has $\rho_n\in{E'_\alpha}$ and then $[\rho_n,\rho_n^+]\subseteq{F_\alpha}$.
Hence $\bigcup_{m\leq{n}<\omega}[\rho_n,\rho_n^+]\subseteq{F_\alpha}$ for every $\alpha\in{J}$. Therefore, $\{\beta_\eta\mid \eta\in\bigcup_{m\leq{n}<\omega}[\rho_n,\rho_n^+]\} \subseteq D_\alpha\cap \gamma$, for every $\alpha\in{J}$.
\end{proof}
Actually, note that in the previous claim we have established the validity of the moreover part of the theorem, as $J$ was of size $\kappa^{++}$. The proof is thus accomplished. 
\end{proof}

In closing this section we would like to introduce a principle which is tantamount to the fact that Prikry forcing $\mathbb{P}(\mathscr{U})$ destroys the strong failure of Galvin's property.
Considering the witness constructed in Theorem \ref{thmprikrydestroys}, we had a function $f(\alpha)=[\alpha,\alpha^+]$ which was a ``guessing" function for subsets of $D_\alpha$. We will prove that in  sense (see Definition \ref{defudiamond} below) these are the only kind of witnesses which are killed by the Prikry forcing.

\smallskip

Recall that if $\mathbb{P}$ is a forcing notion, $G\subseteq\mathbb{P}$ is $V$-generic and $A\in{V[G]}$ then $A$ is called \emph{a fresh set} if $A\notin{V}$ but $A\cap\xi\in{V}$ whenever $\xi<\sup(A)$.

\begin{lemma}
\label{lemfresh} Let $\mathscr{U}$ be a normal ultrafilter over $\kappa$, let $G\subseteq\mathbb{P}(\mathscr{U})$ be $V$-generic and let $C\in{V[G]}$ be a set of ordinals of size at least $\kappa$.
\begin{enumerate}
\item [$(\aleph)$] Either there is a set $B\in{V}$, $ |B|=\kappa$ such that $B\subseteq{C}$, or there is a fresh set $A\in{V[G]}$,  $|A|=\kappa$ such that $A\subseteq{C}$.
\item [$(\beth)$] If $A$ is a fresh set in $V[G]$ and $\gamma_A=\sup(A)$ then $\cf^{V[G]}(\gamma_A)=\omega$.
\end{enumerate}
\end{lemma}
\begin{proof}

We concentrate on $(\aleph)$ and refer the reader to \cite[Theorem 6.1]{TomMotiII} for the proof of Clause~$(\beth)$.
Let $\langle \rho_n\mid n\in\omega\rangle$ be the Prikry sequence derived from $G$.
Working in $V$, we fix a $\mathbb{P}(\mathscr{U})$-name $\lusim{C}$ for the set $C$ and a sequence of $\mathbb{P}(\mathscr{U})$-names $\langle \lusim{c}_i\mid i<\kappa\rangle$ for (some of) the elements of $C$.
Let $p\in{G}$ be such that $$p\Vdash_{\mathbb{P}(\mathscr{U})}\text{$``\langle \lusim{c}_i\mid i<\kappa\rangle\subseteq\lusim{C}$ and $i<j<\kappa\Rightarrow \lusim{c}_i<\lusim{c}_j$.''}$$

For every $\alpha<\kappa$ the set $D_\alpha:=\{p\in\mathbb{P}(\mathscr{U})\mid  p\parallel\lusim{c}_\alpha\}$ is dense and open.

\smallskip

By virtue of the strong Prikry property for $\mathbb{P}(\mathscr{U})$ one can choose $n_\alpha<\omega$ and $p_\alpha$ with $p\leq^*p_\alpha$, such that if $\bar{\beta}$ is of size $n_\alpha$ then $p_\alpha^\frown\bar{\beta}\in{D_\alpha}$.
After making our choices for every $\alpha<\kappa$   we  let $I\in[\kappa]^\kappa$ and $n<\omega$ such that $n_\alpha=n$, for all $\alpha\in I$. 
Finally, let $I:=\{i_\alpha\mid \alpha<\kappa\}$.

By the closure of the pure extension ordering $\leq^*$ we can choose, for each $\delta<\kappa$, some $p^*_\delta$ such that $p_{i_\alpha}\leq^*p^*_\delta$, for all $\alpha<\delta$.
It follows that if $\bar{\beta}$ is of size $n$ then ${p^*_\delta}^\frown\bar{\beta}$ decides $\lusim{c}_{i_\alpha}$ for every $\alpha<\delta$.
We may assume that $p^*_\delta\in{G}$.

Focus on the condition $p^*_{\rho_0}\in{G}$.
By the above considerations, if $|\bar{\beta}|=n$ then ${p^*_{\rho_0}}^\frown\bar{\beta}\Vdash\text{$`` \forall\alpha\in\rho_0\, (\lusim{c}_{i_\alpha}=c_{i_\alpha})$''}$. 
Thus, letting $A_0:=\{c_{i_\alpha}\mid \alpha<\rho_0\}$ we see that $A_0\subseteq{C}, A_0\in{V}$ and $|A_0|=\rho_0$.
Observe that $A_0$ is bounded by $c_{i_{\rho_0}}$  and hence, in particular, $C-A_0$ is of size at least $\kappa$.
Applying the above argument to $C-A_0$ and $p^*_{\rho_1}$ we obtain $A_1\subseteq{C}, A_1\in{V}$ such that $|A_1|=\rho_1$.
Performing this process $\omega$-many times we produce a sequence $\langle A_n\mid n<\omega\rangle$ in $V[G]$ which enjoys the following properties for every $n<\omega$:
\begin{enumerate}
\item [$(\alpha)$] $A_n\in{V}, A_n\subseteq{C}$.
\item [$(\beta)$] $|A_n|=\rho_n$.
\item [$(\gamma)$] $\sup(A_n)<\min(A_{n+1})$.
\end{enumerate}
Notice that $|\bigcup_{n\in\omega}A_n|=\kappa$ by virtue of $(\beta)$.

If $\bigcup_{n<\omega}A_n\in{V}$ then call it $B$ and call it a day.
If not, then $A=\bigcup_{n<\omega}A_n$ must be a fresh set.
Indeed, if $\xi<\sup(A)$ then there exists $n<\omega$ for which $\sup(A_n)<\xi\leq\sup(A_{n+1})$, and thence $A\cap\xi=A_0\cup\cdots\cup{A_n}\cup(A_{n+1}\cap\xi)$.
But the finite sequence $\langle A_0,\ldots,A_{n+1}\rangle$ belongs to $V$, so $A\cap\xi\in{V}$.
\end{proof}

The criterion that we are looking for is based on a sort of a local guessing sequence, as described by the following:

\begin{definition}
\label{defudiamond} Let $\mathscr{U}$ be an ultrafilter over $\kappa$,  $n<\omega$ and $\langle C_\alpha\mid \alpha\in{I}\rangle$ be a sequence of sets. A sequence $\langle S(\bar{\alpha})\mid \bar{\alpha}\in[\kappa]^n\rangle$ is called a \emph{$(\mathscr{U},n)$-diamond sequence for $\langle C_\alpha\mid \alpha\in{I}\rangle$} iff there are sets $\langle A_\alpha\mid \alpha\in I\rangle \s\mathscr{U}$  such that:
\begin{itemize}
    \item $|S(\bar{\beta})|\geq\min(\bar{\beta})$ for every $\bar{\beta}\in[\kappa]^n$;
    \item $S(\bar{\beta})\subseteq{C_\alpha}$ whenever $\bar{\beta}\in[A_\alpha]^n$.
\end{itemize}
\end{definition}

Our criterion for destroying witnesses of $\neg_{\mathrm{st}}\mathrm{Gal}(\mathscr{D}_{\kappa^+},\kappa^+,\kappa^{++})$ in Prikry generic extensions reads as follows:

\begin{theorem}[Destroying witnesses]
\label{thmudiamond} Let $\mathscr{U}$ be a normal ultrafilter over $\kappa$ and let $G\subseteq\mathbb{P}(\mathscr{U})$ be $V$-generic.
Let $\mathcal{C}=\langle C_\alpha\mid \alpha<\kappa^{++}\rangle$ 
be a witness for $\neg_{\mathrm{st}}\mathrm{Gal}(\mathscr{D}_{\kappa^+},\kappa^+,\kappa^{++})$ in the ground model, $V$. 
\smallskip

Then, the following are equivalent:
\begin{enumerate}
\item [$(\aleph)$] $\mathcal{C}$ is no longer a witness for $\neg_{\mathrm{st}}\mathrm{Gal}(\mathscr{D}_{\kappa^+},\kappa^+,\kappa^{++})$ in $V[G]$;
\item [$(\beth)$] In $V$, one can find $n<\omega$, and index set $I\in [\kappa^{++}]^{\kappa^+}$ 
and a $(\mathscr{U},n)$-diamond sequence $\langle S(\bar{\alpha})\mid \bar{\alpha}\in[\kappa]^n\rangle$ for $\langle C_\alpha\mid \alpha\in I\rangle$.
\end{enumerate}
\end{theorem}
\begin{proof}

\underline{$(\beth)\Rightarrow (\aleph)$:}
In $V[G]$, fix any $I\in [\kappa^{++}]^{\kappa^+}$.
Our goal is to show that $|\bigcap_{\alpha\in I}C_\alpha|\geq \kappa.$ 
For each $\alpha\in{I}$ choose $A_\alpha\in\mathscr{U}$ such that $S(\bar{\beta})\subseteq{C_\alpha}$ for every $\bar{\beta}\in[A_\alpha]^n$.
Let $\langle\rho_n\mid n\in\omega\rangle$ be the Prikry sequence derived from $G$.
Since $A_\alpha\in\mathscr{U}$ one can find, for every $\alpha\in{I}$, some $n_\alpha<\omega$ so that $\langle \rho_m\mid m\geq{n_\alpha}\rangle\subseteq{A_\alpha}$.
Shrink $I$ to $J\in[I]^{\kappa^+}$, if needed, so that $n_\alpha=m$, for all $\alpha\in{J}$. 
For every $m\leq\ell<\omega$ let $\bar{\alpha}_\ell:=\langle \rho_\ell,\ldots,\rho_{\ell+n-1}\rangle$.
Notice that $\bar{\alpha}_\ell\in[A_\alpha]^n$ for every $\alpha\in{J}$ and every $m\leq\ell<\omega$.
Hence, $S(\bar{\alpha}_{\ell})\subseteq{C_\alpha}$ for $\alpha\in{J}, m\leq\ell<\omega$.
Thus, letting $B:=\bigcup\{S(\bar{\alpha}_\ell)\mid m\leq\ell\in\omega\}$, we have $B\subseteq\bigcap_{\alpha\in J}C_\alpha$.
Since $|S(\bar{\alpha}_\ell)|\geq\rho_\ell$ and $\sup_{\ell<\omega}\rho_\ell=\kappa$ then $|B|=\kappa$. Altogether, $\mathcal{C}$ ceases to be a witness for $\neg_{\mathrm{st}}\mathrm{Gal}(\mathscr{D}_{\kappa^+},\kappa^+,\kappa^{++})$.

\smallskip

\underline{$(\aleph)\Rightarrow (\beth)$:} Assume that the sequence $\mathcal{C}=\langle C_\alpha\mid \alpha<\kappa^{++}\rangle$ does not witness $\neg_{\mathrm{st}}\mathrm{Gal}(\mathscr{D}_{\kappa^+},\kappa^+,\kappa^{++})$ in $V[G]$.
Fix $I\in [\kappa^{++}]^{\kappa^+}$ such that $|C|\geq\kappa$ where $C:=\bigcap_{\alpha\in I} C_\alpha$.
Apply Lemma~\ref{lemfresh} to $C$.
We may assume, that there exists a fresh set $A\subseteq{C}$, for otherwise $\mathcal{C}$ would not be a witness for $\neg_{\mathrm{st}}\mathrm{Gal}(\mathscr{D}_{\kappa^+},\kappa^+,\kappa^{++})$ in the ground model.
Let $\gamma:=\gamma_A=\sup(A)$, so $\cf^{V[G]}(\gamma)=\omega$ (see Lemma~\ref{lemfresh}$(\beth)$). 
Choose a cofinal sequence in $\gamma$ of the form $\langle \gamma_n\mid n<\omega\rangle$ and a sequence of names $\langle \lusim{\gamma}_n\mid n<\omega\rangle$ for these ordinals.

Observe that (in $V[G]$) $A$ is expressible as $\bigcup_{n<\omega}(A\cap\gamma_n)$. Since $|A|=\kappa$ one can choose, for each $n<\omega$, some $m_n<\omega$ so that $|A\cap\gamma_{m_n}|\geq\rho_n$.
Define $g(n)=m_n$ for every $n<\omega$, and notice that $g\in{V}$.\footnote{This being true as Prirky forcing does not introduce new members of $^{\omega}{}\omega$.} 
Fix a condition $p=(t,A^p)\in{G}$ which forces all the relevant information.
That is, $p$ forces that $\lusim{A}$ is fresh, $\langle \lusim{\gamma}_n\mid n<\omega\rangle$ is cofinal in $\sup(\lusim{A})$,  $\lusim{A}\subseteq{C_\alpha}$ for every $\alpha\in{I}$ and $|\lusim{A}\cap\lusim{\gamma}_{g(n)}|\geq\check{\rho}_n$ for every $n<\omega$.

For every $\ell<\omega$ we invoke the \emph{Strong Prikry Property} of $\mathbb{P}(\mathscr{U})$ aiming to pick $n_\ell<\omega$ and  $p\leq^*p_\ell$ so that, if $p_\ell=(t,A_\ell)$, then for every $\bar{\alpha}\in[A_\ell]^{n_\ell}$ the condition $p_\ell^\frown\bar{\alpha}$ decides the values of $\lusim{\gamma}_\ell$ and $\lusim{A}\cap\lusim{\gamma}_\ell$.
Let $B:=\bigcap_{\ell<\omega}A_\ell$. Clearly, $B\in\mathscr{U}$, so that $q:=(t,B)\in\mathbb{P}(\mathscr{U}).$ Finally, note that $p\leq^*q$.

If $\bar{\alpha}\in[B]^{n_\ell}$ then $q^\frown\bar{\alpha}\Vdash_{\mathbb{P}(\mathscr{U})} \lusim{A}\cap\lusim{\gamma}_\ell=D(\bar{\alpha})$, for some $D(\bar{\alpha})$.
Define, $$J:=\{\alpha<\kappa^{++}\mid \exists{r}\geq{q}\, (r\Vdash_{\mathbb{P}(\mathscr{U})}\lusim{A}\subseteq{C_\alpha})\}.$$ Note that $J\supseteq{I}$ and $J\in{V}$.
If $\alpha\in{J}$ then there is a condition $q_\alpha=(s_\alpha,B_\alpha)$ such that $q\leq{q_\alpha}$ and $q_\alpha\Vdash_{\mathbb{P}(\mathscr{U})}\lusim{A}\subseteq{C_\alpha}$.
Since $|J|\geq\kappa^+$ we may freely assume that $s_\alpha=s$ for every $\alpha\in{J}$. Let $m=|s|$ and define 
$$
E(\bar{\alpha}):=\begin{cases}
D((s-t)^\frown\bar{\alpha}), & \text{if $(s-t)^\frown\bar{\alpha}\in[B]^{n_m+1}$;}\\
\kappa, & \text{otherwise.}
\end{cases}
$$
\begin{claim}
$\langle E(\bar{\alpha})\mid \bar{\alpha}\in[\kappa]^{n_m+1}\rangle $ yields a $(\mathscr{U},n_m+1)$-diamond sequence for $\langle C_\alpha\mid \alpha\in{J}\rangle$, as witnessed by $\langle B_\alpha\mid \alpha\in J\rangle$.
\end{claim}
\begin{proof}[Proof of claim]
Fix $\bar{\alpha}\in[B_\alpha]^{n_m+1}$ and notice that $$q_\alpha^\frown\bar{\alpha}\Vdash_{\mathbb{P}(\mathscr{U})} E(\bar{\alpha})=\lusim{A}\cap\lusim{\gamma}_{g(n_m+1)}\subseteq{C_\alpha}$$ and also $q_\alpha^\frown\bar{\alpha}\Vdash_{\mathbb{P}(\mathscr{U})}|E(\bar{\alpha})|\geq\lusim{\rho}_{n_m+1}$, so we are done.
\end{proof}
This completes the verification of $(\aleph)\Rightarrow (\beth)$.
\end{proof}

The above theorem gives a useful criterion for destroying a ground model witness to $\neg_{\mathrm{st}}\mathrm{Gal}(\mathscr{D}_{\kappa^+},\kappa^+,\kappa^{++})$.
But actually this criterion is sufficient for any witness, new and old alike, by the following:
\begin{corollary}\label{Culmination}
\label{corcriterion} Assume $\neg_{\mathrm{st}}\mathrm{Gal}(\mathscr{D}_{\kappa^+},\kappa^+,\kappa^{++})$ holds in $V$. Let $\mathscr{U}$ be a normal ultrafilter over $\kappa$ and $G\subseteq\mathbb{P}(\mathscr{U})$ be $V$-generic. Then, the  following assertions are equivalent:
\begin{enumerate}
    \item $\neg_{\mathrm{st}}\mathrm{Gal}(\mathscr{D}_{\kappa^+},\kappa^+,\kappa^{++})$ fails in $V[G]$;
    \item Every ground model witness for $\neg_{\mathrm{st}}\mathrm{Gal}(\mathscr{D}_{\kappa^+},\kappa^+,\kappa^{++})$ is destroyed.
\end{enumerate}
\end{corollary} 
\begin{proof}

The forward direction is trivial.
For the backward direction it suffices to establish the following easy claim concerning $\kappa^+$-cc forcing notions:
\begin{claim}\label{groundmodelwitness}
Assume $\mathbb{P}$ is a $\kappa^+$-cc forcing notion, and let $\lambda$ be an ordinal.

For every $\langle C_\alpha\mid \alpha<\lambda\rangle\subseteq\mathscr{D}_{\kappa^+}$ in a generic extension by $\mathbb{P}$ there is a sequence $\langle D_\alpha\mid \alpha<\lambda\rangle\subseteq\mathscr{D}_{\kappa^+}$ in $V$ such that $D_\alpha\subseteq{C_\alpha}$ for every $\alpha<\lambda$.
\end{claim}
\begin{proof}
For each $\alpha<\lambda$, choose $E_\alpha\in\mathscr{D}_{\kappa^+}$ such that $E_\alpha\subseteq{C_\alpha}$ and $E_\alpha\in{V}$.
This is possible since $\mathbb{P}$ is $\kappa^+$-cc.
Note, however, that $\langle E_\alpha\mid \alpha<\lambda\rangle$ need not be member of $V$, although it can be modified to be so by the following procedure.
For each $\alpha<\lambda$, let $\mathcal{A}_\alpha=\{q^\alpha_\eta\mid \eta<\delta_\alpha\}$ be 
a maximal antichain of conditions which decide $E_\alpha$; namely, for each $\alpha<\delta_\alpha$, there is $E^\alpha_\eta\in V$ such that $q^\alpha_\eta\Vdash_\mathbb{P}\lusim{E}_i=E^\alpha_\eta$.
Define a function $f\colon\lambda\rightarrow{V}$, $f\in{V}$ by $f(\alpha):=\langle E^\alpha_\eta\mid \eta<\delta_\alpha\rangle$.
For each $\alpha<\lambda$, let $D_\alpha:=\bigcap_{\eta<\delta_\alpha}E^\alpha_\eta$. By $\kappa^+$-ccness of $\mathbb{P}$,  $D_\alpha\in\mathscr{D}_{\kappa^+}$ and it is forced by $\mathcal{A}_\alpha$ (hence by $\one_\mathbb{P}$) that $``\check{D}_\alpha\subseteq{\check{E}_\alpha}\subseteq{\lusim{C}_\alpha}$''.
Moreover, $\langle D_\alpha\mid \alpha<\lambda\rangle\in{V}$, so we are done. 
\end{proof}
At this stage the proof has been accomplished.
\end{proof}

\section{Stronger forms of Galvin's property on normal filters}\label{SectionStronger} 

Galvin's theorem applies to arbitrary normal filters on $\kappa$ and not just to the club filter $\mathscr{D}_\kappa$ (\cite[\S3.2]{MR0369081}).
Specifically, if $\kappa^{<\kappa}=\kappa$ then Galvin's theorem says that $\mathrm{Gal}(\mathscr{F},\kappa,\kappa^+)$ holds for every normal filter $\mathscr{F}$ over $\kappa$. Here  $\mathrm{Gal}(\mathscr{F},\mu,\lambda)$ denotes the natural extension of the principle $\mathrm{Gal}(\mathscr{D}_\kappa,\mu,\lambda)$.



The central idea in Galvin's proof is to find a subfamily of clubs whose intersection equals its diagonal intersection up to some negligible set. Once this is the accomplished normality yields the desired property.
Unfortunately,  this idea is limited to subfamilies of size $\kappa$.
Therefore one has to work in a different direction to get the consistency $\mathrm{Gal}(\mathscr{F},\mu,\lambda)$, whenever $\mu\geq\kappa^+$. 
The next remark present some easy facts on the principles $\mathrm{Gal}(\mathscr{F},\mu,\lambda)$: 
\begin{remark}
Let $\mathscr{F}$ be a filter over $\kappa$. Then:
\begin{enumerate}
    \item [($\aleph$)] 
    For every $\mu'\leq\mu\leq\lambda\leq\lambda'$, ${\rm Gal}(\mathscr{F},\mu,\lambda)\Rightarrow {\rm Gal}(\mathscr{F},\mu',\lambda')$.
\item [($\beth$)]  $\mathscr{F}$ is $\mu$-complete iff for every $\mu'<\mu$, ${\rm Gal}(\mathscr{F},\mu',\mu')$.
\item [$(\gimel)$]
If $\cf(\mu)=\kappa$ then ${\rm Gal}(\mathscr{F},\mu,\mu)$ fails.
\end{enumerate}
\end{remark}

The following theorem shows that stronger failures of Galvin's property can be force at the level of measurable cardinals.
This will be followed by a proposition showing that the opposite direction is consistent as well.
\begin{theorem}
\label{thmnegmeasurable} It is consistent with \textsf{ZFC} that $\kappa$ is a measurable cardinal and ${\rm Gal}(\mathscr{U},\kappa^+,\lambda)$ fails for every $\lambda>\kappa$ and every normal ultrafilter $\mathscr{U}$ over $\kappa$.
\end{theorem}
\begin{proof}
Fix a normal ultrafilter $\mathscr{U}$ over $\kappa$.
Recall that  $\binom{\lambda}{\kappa}\rightarrow\binom{\kappa^+}{\kappa}$ stands for the polarized relation asserting that for every coloring $c\colon \lambda\times\kappa\rightarrow 2$ there are sets $A\in[\lambda]^{\kappa^+}$ and $B\in[\kappa]^\kappa$ for which $c\upharpoonright(A\times B)$ is constant.

\begin{claim}\label{Claimpolarized}
${\rm Gal}(\mathscr{U},\kappa^+,\lambda)$ implies  $\binom{\lambda}{\kappa}\rightarrow\binom{\kappa^+}{\kappa}$.
\end{claim}
\begin{proof}[Proof of claim]
Assume that ${\rm Gal}(\mathscr{U},\kappa^+,\lambda)$ holds and let $c:\lambda\times\kappa\rightarrow 2$ be a coloring.
For each $\alpha<\lambda$ and $i\in\{0,1\}$, let $S^\alpha_i:=\{\beta<\kappa\mid c(\alpha,\beta)=i\}$. Since $\mathscr{U}$ is a ultrafilter, 
for each $\alpha<\lambda$ there is $i(\alpha)\in \{0,1\}$ with $S^\alpha_{i(\alpha)}\in\mathscr{U}$.
We may assume that $i(\alpha)=i$ for some fixed  $i\in\{0,1\}$ and every $\alpha<\lambda$.

Apply ${\rm Gal}(\mathscr{U},\kappa^+,\lambda)$ to the family $\mathcal{S}:=\langle S^\alpha_i\mid \alpha<\lambda\rangle$ and  obtain in return a subfamily $\mathcal{T}=\langle S^{\alpha_\gamma}_i\mid \gamma<\kappa^+\rangle$ so that  
 $B:=\bigcap\mathcal{T}\in\mathscr{U}$. In particular,  $|T|=\kappa$.\footnote{This follows from the fact that every normal ultrafilter is uniform (cf. Definition~\ref{CombinatorialProperties}).}
Let $A:=\{\alpha_\gamma\mid \gamma\in\kappa^+\}$.
Observe that $c\upharpoonright(A\times B)$ is constantly equal to $i$. With this we conclude that $\binom{\lambda}{\kappa}\rightarrow\binom{\kappa^+}{\kappa}$ holds.\footnote{Actually, the stronger principle  $\binom{\lambda}{\kappa}\rightarrow\binom{\kappa^+}{\mathscr{U}}$ holds.}
\end{proof}
The finall step of the proof will be forcing the negative polarized relation $\binom{\lambda}{\kappa}\nrightarrow\binom{\kappa^+}{\kappa}$, thus proving $\neg{\rm Gal}(\mathscr{U},\kappa^+,\lambda)$ for every normal ultrafilter $\mathscr{U}$ over $\kappa$.
We begin with a measurable cardinal $\kappa$, indestructible under adding $\lambda$-many Cohen subsets to it.
A Laver-indestructible supercompact cardinal will certainly suffice, but actually much less is needed (see \cite{MR3928383}).

Let $\mathbb{P}:=\mathrm{Add}(\kappa,\lambda)$ and $G\subseteq\mathbb{P}$ be a $V$-generic filter.
Let $\langle \eta_\alpha\mid \alpha<\lambda\rangle $ be the (characteristic functions of the) Cohen subsets added to $\kappa$. So, $\eta_\alpha\in{}^\kappa 2$ for each $\alpha<\lambda$.
By our assumptions on $\kappa$ this remains measurable in $V[G]$.

\smallskip

Define $c\colon \lambda\times\kappa\rightarrow 2$ by letting $c(\alpha,\beta):=\eta_\alpha(\beta)$ for every $\alpha<\lambda$ and $\beta<\kappa$.
Let $\lusim{c}$ be a $\mathbb{P}$-name for $c$ and assume, toward contradiction, that $\lusim{A}\in[\lambda]^{\kappa^+}$, $\lusim{B}\in[\kappa]^\kappa$ and $p\Vdash_{\mathbb{P}}\lusim{c}``(\lusim{A}\times\lusim{B})=\{\check{i}\}$.
Our goal is to find $p\leq r$ and a pair of ordinals $\alpha<\lambda$, $\beta<\kappa$ such that: 
$$\text{$r\Vdash_{\mathbb{P}}\text{$``\check{\alpha}\in\lusim{A}$''}$, $r\Vdash_{\mathbb{P}}\text{$``\check{\beta}\in\lusim{B}$''}$ and $r\Vdash_{\mathbb{P}}\text{$``\lusim{c}(\check{\alpha},\check{\beta})=1-i$''}$.}$$
Since $p$ is an arbitrary condition this will show that ${\rm Gal}(\mathscr{U},\kappa^+,\lambda)$ fails in $V[G]$.
For every $\beta<\kappa$ let $\varphi_\beta$ be the statement in the forcing language saying $``\check{\beta}\in\lusim{B}$''.
Let $\mathcal{A}_\beta$ be a maximal antichain which decides $\varphi_\beta$, and set:
\[S:=\bigcup\{{\rm dom}(q)\mid q\in\mathcal{A}_\beta,\beta<\kappa\}.\]
Notice that $|S|\leq\kappa$ and hence $\one_{\mathbb{P}}\Vdash_{\mathbb{P}}(\lusim{A}\nsubseteq \dom(S))$.
Pick $\alpha\in\lambda-\dom(S)$ and a condition $p_0\geq p$ such that $p_0\Vdash_{\mathbb{P}}\check{\alpha}\in\lusim{A}$.
We may assume that $\alpha\in{\rm dom}_1(p_0)$; that is, there is $\beta<\lambda$ such that $\langle \alpha,\beta\rangle \in \dom(p_0)$.
Let $p_1:=p_0\upharpoonright S$.
Since $\one_{\mathbb{P}}\Vdash_{\mathbb{P}}|\lusim{B}|=\kappa$ and $|\dom(p(\alpha))|<\kappa$, there is $\beta\notin\dom(p(\alpha))$ and $p_0'\in\mathcal{A}_\beta$ such that $p'_0\Vdash_{\mathbb{P}} \beta\in \lusim{B}$. 
Let $q_1:=p_0\upharpoonright({\rm dom}(p_0)-{\rm dom}(p'_0))\cup p'_0$.
Finally, let $r:=q_1\cup\{(\alpha,\beta,i-1)\}$.
Now $r\Vdash_{\mathbb{P}}\check{\alpha}\in\lusim{A}$ since $p_0\leq r$, and $r\Vdash_{\mathbb{P}}\check{\beta}\in\lusim{B}$ since $p'_0\leq r$.
By definition, $r\Vdash_{\mathbb{P}}\lusim{c}(\check{\alpha},\check{\beta})=1-i$.
\end{proof}

To round out the picture let us show that ${\rm Gal}(\mathscr{U},\kappa^+,\lambda)$ is consistent and,  in fact, that even ${\rm Gal}(\mathscr{U},{<}\lambda,\lambda)$ can be forced.
To this effect we draw a co\-nnection between \emph{strong generating sequences} for normal filters and Galvin's property.
The first to notice this connection was Gitik \cite{GitDensity}.
More information and relevant open problems  are spelled out in \cite{TomMotiII}. 

\begin{definition}\label{generatingseq}
A family of sets $\mathcal{B}=\langle B_\alpha \mid  \alpha<\lambda\rangle$ is called a \emph{strong generating sequence for a filter $\mathscr{F}$ over $\kappa$} if the following are true:
\begin{enumerate}
    \item $\mathcal{B}$ is $\subseteq^*$-decreasing: namely, if $\alpha\leq\beta$ then 
    $B_{\alpha}\s^* B_\beta$;
    \footnote{As usual, $A\s^*B$ is a shorthand for  $B\setminus A$ is bounded (in $\kappa$).}
    \item For every $X\in\mathscr{F}$ there is $\alpha<\lambda$ such that $B_\alpha\subseteq^*X$.
\end{enumerate}
We will say $\mathcal{B}$ is a generating sequence for $\mathscr{F}$ if just Clause~(2) above holds.
\end{definition}
\begin{proposition}
\label{clmposmeasurable} Let $\mathscr{F}$ be a normal filter over $\kappa$ and suppose that $\mathcal{B}=\langle B_\alpha\mid  \alpha< \lambda\rangle$ is a generating sequence for $\mathscr{F}$. Then, 
\begin{enumerate}
    \item [($\aleph$)] If $\lambda<\cf(\partial)\leq\partial\leq 2^\kappa$, then ${\rm Gal}(\mathscr{F},\partial,\partial)$.
    \item [($\beth$)] In addition, if $\mathcal{B}$ is a strong generating sequence then ${\rm Gal}(\mathscr{F},\partial,\partial)$ holds, for every  $\kappa<\cf(\partial)\leq\partial<\cf(\lambda)\leq\lambda$. 
\end{enumerate}
Moreover, in Clause~$(\beth)$, if $\mathscr{F}$ is not generated by a set then $\neg{\rm Gal}(\mathscr{F},\lambda,\lambda)$.
\end{proposition}
\begin{proof}
For the scope of the proof fix $\mathcal{C}=\langle C_\alpha\mid \alpha<\partial\rangle \subseteq\mathscr{F}$. 

\smallskip

$(\aleph)$:  For every $\alpha<\partial$ choose an ordinal $\beta(\alpha)<\lambda$ for which $B_{\beta(\alpha)}\subseteq^*C_\alpha$.
Since we are assuming that $\cf(\partial)>\lambda$, there are $A\in[\partial]^{\partial}$ and $\beta<\lambda$ such that $\beta(\alpha)=\beta$, for all $\alpha\in A$. 
In particular,  for every $\alpha\in{A}$ there is an ordinal $\gamma(\alpha)<\kappa$ such that $B_\beta-\gamma(\alpha)\subseteq{C_\alpha}$.
Once again, shrink $A$ to some $\tilde{A}\in[A]^{\partial}$ in such a way that for some  $\gamma<\kappa$,  $B_\beta-\gamma\s C_\alpha$, for all $\alpha\in \tilde{A}$. 
Finally, put $B:=B_\beta-\gamma$ and note that $B\in\mathscr{F}$. It is immediate that $B\s \bigcap_{\alpha\in \tilde{A}}C_\alpha$ and so this latter set belongs to $\mathscr{F}$.

\smallskip

$(\beth)$:  
For every ordinal $\alpha<\partial$ choose $\beta(\alpha)<\lambda$ such that $B_{\beta(\alpha)}\subseteq^*C_\alpha$.
Let $\beta:=\sup_{\alpha<\partial}\beta(\alpha)$. Clearly $\beta<\lambda$, for $\cf(\lambda)> \partial$. Also, by virtue of Definition~\ref{generatingseq}(2),  $B_\gamma\s^* B_{\beta(\alpha)}$.
Since $\cf(\partial)>\kappa$ there is a subfamily $\langle C_{\alpha_\delta}\mid \delta\in\partial\rangle$ of $\mathcal{C}$ and an ordinal $\xi<\kappa$ so that $B_\beta-\xi\subseteq C_{\alpha_\delta}$ for every $\delta\in\partial$. This yields the desired result.

\smallskip

For the moreover part, assume towards contradiction that $\mathrm{Gal}(\mathscr{F},\lambda,\lambda)$ holds. Then, in particular, there is $\tilde{\mathcal{B}}\s \mathcal{B}$ with $|\tilde{\mathcal{B}}|=\lambda$ such that $\bigcap \tilde{\mathcal{B}}\in\mathscr{F}$. Denote this latter set by $B$. Combining Clauses~(1) and (2) of Definition~\ref{generatingseq} it is easy to prove that every $A\in\mathscr{F}$ includes $B$, up to a negligible set. This yields the desired contradiction and accomplishes the proof.
\end{proof}
A (strong) generating sequence of arbitrary length $\lambda=\cf(\lambda)>\kappa$ can be forced provided $\kappa$ is a huge cardinal \cite{MR1632081}. More recently, the same has been proved under the weaker assumption of supercompactness \cite{MR3201820, MR3564375}.  Combining this latter result with Proposition~\ref{clmposmeasurable}  one arrives to following immediate corollary:
\begin{corollary}
Suppose that $\kappa$ is a supercompact cardinal and $\lambda>\kappa$ is regular. Then it is consistent with \textsf{ZFC} that 
$\kappa$ is measurable and there is a normal ultrafilter $\mathscr{U}$ over $\kappa$ such that ${\rm Gal}(\mathscr{U},\partial,\partial)$ holds, for every $\lambda<\cf(\partial)\leq\partial\leq 2^\kappa$ or $\kappa<\cf(\partial)\leq\partial<\lambda$.   
\end{corollary} 

If one focuses on the club filter $\mathscr{D}_\kappa$ the concept of \emph{dominating families of ${}^\kappa\kappa$} can be used  to control the length of a generating sequence for $\mathscr{D}_\kappa$. Recall that a family of functions $\mathcal{D}\subseteq{}^\kappa\kappa$ is a \emph{dominating family} if for every $f\in{}^\kappa\kappa$ there exists $g\in\mathcal{D}$ such that $f\leq^* g$; i.e., $|\{\alpha<\kappa\mid f(\alpha)>g(\alpha)\}|<\kappa.$

The \emph{dominating number at $\kappa$,} $\mathfrak{d}_\kappa$, is defined as follows:
$$\mathfrak{d}_\kappa:=\min\{|\mathcal{D}|\mid\text{$\mathcal{D}\s {}^\kappa\kappa$ is a dominating family}\}.$$

The minimal size for a generating sequence for the club filter $\mathscr{D}_\kappa$ is known to be the dominating number $\mathfrak{d}_\kappa$ (see Claim~\ref{Minimalsize}). 
Hence,  in case $\mathfrak{d}_\kappa<\cf(2^\kappa)$,  Proposition~\ref{clmposmeasurable}$(\aleph)$ has the following  corollary:
\begin{corollary}
\label{propfullgalvin} Let $\kappa$ be an uncountable cardinal with $\kappa^{<\kappa}=\kappa$. 

If $\mathfrak{d}_\kappa<\cf(2^\kappa)$ then ${\rm Gal}(\mathscr{D}_\kappa,2^\kappa,2^\kappa)$ holds.  
\end{corollary}
\begin{proof}
It suffices to prove the following:
\begin{claim}\label{Minimalsize}
Suppose that $\kappa=\cf(\kappa)>\aleph_0$. Then, $\mathfrak{d}_\kappa$ is the minimal size for a generating sequence of $\mathscr{D}_\kappa$. More precisely, $(\aleph)$ and $(\beth)$ below hold:
\begin{itemize}
    \item[$(\aleph)$] For $\theta<\mathfrak{d}_\kappa$ and $\langle C_\alpha\mid \alpha<\theta\rangle\s \mathscr{D}_\kappa$ there is $C\in\mathscr{D}_\kappa$ with $\neg(C_\alpha\s^*C).$
     \item[$(\beth)$] There exists a family $\langle C_\alpha\mid \alpha<\mathfrak{d}_\kappa\rangle$ such that for every $C\in\mathscr{D}_\kappa$ there is $\alpha<\mathfrak{d}_\kappa$ with $C_\alpha\s^*C$.
\end{itemize}
\end{claim}
\begin{proof}[Proof of claim]
$(\aleph)$: For each $\alpha<\theta$ define maps $f_\alpha\colon \kappa\rightarrow\kappa$ as follows:
$$f_\alpha(\delta):=\min(C_\alpha-(\delta+1)).$$
Put $\mathcal{F}:=\{f_\alpha\mid \alpha<\theta\}$. Since $|\mathcal{F}|\leq \theta<\mathfrak{d}_\kappa$ we see that $\mathcal{F}$ is not dominating and hence we can pick $h\in{}^\kappa\kappa$ such that  $\neg (h\leq^* f_\alpha)$ for all $\alpha<\theta$. Next, define $C:=\{\delta<\kappa\mid h``\delta\s \delta\}$. Note that $C$ is a club at $\kappa$. Moreover, $\neg(C_\alpha\s^*C)$ for all $\alpha<\theta$, for otherwise $h\leq^*f_{\alpha^*}$ for some $\alpha^*$.

$(\beth)$:  Fix a dominating family $\mathcal{D}=\{f_\alpha\mid \alpha<\mathfrak{d}_\kappa\}$. For each $\alpha<\mathfrak{d}_\kappa$ let $C_\alpha:=\{\delta\in \mathrm{acc}(\kappa)\mid f_\alpha``\delta\s \delta\}$. Note that $C_\alpha$ is a club at $\kappa$. Now let $C\in\mathscr{D}_\kappa$ and $h\in{}^\kappa\kappa$ be its increasing enumeration. Choose $\alpha<\mathfrak{d}_\kappa$ such that $h\leq^*f_\alpha$. We now check that $C_\alpha\s^* C$. Let $\delta_0<\kappa$ be such that $h(\delta)\leq f_\alpha(\delta)$ for all $\delta\geq \delta_0$. Also, let $\delta_1:=\min(C_\alpha\setminus \delta_0+1)$. We claim that $C_\alpha\setminus \delta_1\s C$. Let $\eta\in C_\alpha\setminus \delta_1$ and for all $\theta\in (\delta_0,\eta)$ note that $\theta\leq h(\theta)<f_\alpha(\theta)<\eta$. Hence, $\eta=\sup_{\theta\in (\delta_0,\eta)}f_\alpha(\theta)=\sup_{\theta\in (\delta_0,\eta)} h(\theta)\in C$, as desired.
\end{proof}
The proof of the corollary has been accomplished.
\end{proof}

In Proposition~\ref{generatingseq} we showed that strong generating sequences of normal filters $\mathscr{F}$ 
yield several instances of $\mathrm{Gal}(\mathscr{F},\mu,\lambda)$. As the  next proposition illustrates, in some sense, this implication can be reversed:

\begin{proposition}
\label{clmgalimpliesgen} Let $\mathscr{U}$ be a normal ultrafilter over $\kappa$. Assuming that $2^\kappa=\lambda$ and ${\rm Gal}(\mathscr{U},\partial,\partial)$ holds for every $\kappa^+\leq\partial<\lambda$ it follows that there is  $\mathcal{B}=\langle B_\alpha\mid \alpha<\lambda\rangle$ a strong generating sequence for $\mathscr{U}$. 
\end{proposition}
\begin{proof}
Let $\langle A_\alpha\mid \alpha<\lambda\rangle$  be an injective enumeration of the elements of $\mathscr{U}$.
We shall construct  $\mathcal{B}=\langle B_\alpha\mid \alpha<\lambda\rangle$ by induction on $\lambda$ as follows. Suppose that $\beta<\lambda$ and $\langle B_\alpha\mid \alpha<\beta\rangle$ is $\subseteq^*$-descending and, for each $\alpha<\beta$, $B_\alpha\subseteq{A_\alpha}$. 
Let $\partial=\cf(\beta)$ and choose a cofinal sequence $\langle \alpha_\gamma\mid \gamma<\partial\rangle$ in $\beta$.
If one can find $B_\beta\in\mathscr{U}$ such that $B_\beta\subseteq{A_\beta}$ and $B_\beta\subseteq^*{B_{i_\alpha}}$ for every $\alpha<\partial$ then 
we will be done, for $\langle B_\alpha\mid \alpha<\beta\rangle$ is $\subseteq^*$-descending. Three cases are distinguished:

$\br$ \underline{Assume $\partial<\kappa$:} Put $B'_\beta:=\bigcap_{\gamma<\partial}B_{\alpha_\gamma}$. Since  $\mathscr{U}$ is $\kappa$-complete,  $B'_\delta\in\mathscr{U}$, hence $B_\beta:=B'_\beta\cap{A_\beta}$ is as required.

$\br$ \underline{Assume $\partial=\kappa$:} Put $B'_\beta:=\diagonal_{\gamma<\partial}B_{\alpha_\gamma}$. By normality of $\mathscr{U}$, $B'_\beta\in\mathscr{U}$. Also, $B'_\beta\subseteq^*B_{\alpha_\gamma}$ for every $\gamma<\partial$. Thus,  $B_\beta:=B'_\delta\cap{A_\delta}$ is as required.

$\br$ \underline{Assume $\partial\geq \kappa^+$:} In this case apply ${\rm Gal}(\mathscr{U},\partial,\partial)$ to the collection of sets $\langle B_{\alpha_\gamma}\mid \gamma<\partial\rangle$ and obtain in return an index set $I\in[\partial]^{\partial}$ such that $B'_\beta=\bigcap_{\gamma\in I}B_{\alpha_\gamma}\in\mathscr{U}$.
Since $\langle \alpha_\gamma\mid \gamma\in{I}\rangle$ was cofinal in $\partial$, $B'_\beta\subseteq^*{B_\alpha}$ for every $\alpha<\beta$.
Altogether,  $B_\beta:=B'_\beta\cap{A_\beta}$ is as desired.

\smallskip

The above argument yields a $\subseteq^*$-decreasing sequence of members of $\mathscr{U}$, and it remains to show that it satisfies Clause~(2) of Definition~\ref{generatingseq}. This is quite easy: 
If $A\in\mathscr{U}$ then $A=A_\alpha$ for some $\alpha<\lambda$ and then $B_\alpha\subseteq{A_\alpha}$.
\end{proof}


We close the section by spelling out a strinking connection between Galvin's property and density of old sets in generic extensions by Prikry forcing (see Clause~$(\beth)$ below).
An initial statement  on this vein appeared in \cite{GitDensity}. Here we provide a complete characterization for ${\rm Gal}(\mathscr{U},\partial,\lambda)$ to hold in terms of  density of old sets.

\begin{theorem}
\label{propdensity} Let $\mathscr{U}$ be a normal ultrafilter over $\kappa$ and  $G\subseteq\mathbb{P}(\mathscr{U})$ be a $V$-generic filter.  Additionally, suppose that $\kappa\leq\partial\leq\lambda$ and $\cf(\lambda)\geq\kappa^+$.
Then, the following are equivalent:
\begin{enumerate}
\item [$(\aleph)$] ${\rm Gal}(\mathscr{U},\partial,\lambda)$.
\item [$(\beth)$] Every $x\in V[G]$ with $|x|^{V[G]}=\lambda$ contains a set $y\in{V}$ with $|y|^V=\partial$.
\end{enumerate}
\end{theorem}
\begin{proof}

$(\aleph)\Rightarrow(\beth)$: Assume ${\rm Gal}(\mathscr{U},\partial,\lambda)$ holds and let $x$ be a set of size $\lambda$ in $V[G]$.
Let $\{\lusim{x}_\alpha\mid \alpha<\lambda\}$ be a sequence of $\mathbb{P}(\mathscr{U})$-names enumerating the elements of $x$.
For each $\alpha<\lambda$ choose an ordinal $\xi_\alpha$ and a condition $p_\alpha=(t_\alpha,A_\alpha)$ so that $p_\alpha\Vdash_{\mathbb{P}(\mathscr{U})}\lusim{x}_\alpha=\check{\xi}_\alpha$. Choose $I\in[\lambda]^\lambda$ and a fixed finite sequence $t$ such that $t_\alpha=t$, for every $\alpha\in I$.
This is possible because $\cf(\lambda)\geq\kappa^+$ and $\kappa^{<\kappa}=\kappa$.
Since $\langle A_\alpha\mid \alpha<\lambda\rangle\subseteq\mathscr{U}$ and  ${\rm Gal}(\mathscr{U},\partial,\lambda)$ holds, there is  a set $J\in[I]^\partial$ for which $B=\bigcap_{\alpha\in J} A_\alpha\in\mathscr{U}$.
Define $q=(t,B)$ and notice that $p_\alpha\leq{q}$ for every $\alpha\in{J}$.
Let $y:=\{\xi_\alpha\mid \alpha\in{J}\}$. Clearly,  $y\in{V}$.
Also, note that $q\Vdash_{\mathbb{P}(\mathscr{U})}\check{y}\subseteq\lusim{x}$.
Since this argument can be rendered above any condition in $\mathbb{P}(\mathscr{U})$ we are done with the proof of $(\aleph)\Rightarrow(\beth)$.

\smallskip

$(\beth)\Rightarrow(\aleph)$: Suppose that ${\rm Gal}(\mathscr{U},\partial,\lambda)$ fails and
fix  $\langle A_\alpha\mid \alpha<\lambda\rangle\subseteq\mathscr{U}$ witnessing this fact.
Let $\langle \rho_n\mid n<\omega\rangle$ be the Prikry sequence derived from the generic $G$.
For each $\alpha<\lambda$ let $n_\alpha<\omega$ be such that $\langle \rho_n\mid n\geq{n_\alpha}\rangle\subseteq{A_\alpha}$. Since $\mathbb{P}(\mathscr{U})$ is $\kappa^+$-cc and $\cf^V(\lambda)\geq \kappa^+$ then $\cf^{V[G]}(\lambda)\geq \kappa^+$. In particular, one can find a set $x\in[\lambda]^{\lambda}$ and  $n<\omega$ such that $n_\alpha=n$, for all $\alpha\in x$.

By way of contradiction assume that $y\subseteq{x}$, $y\in{V}$ and $|y|^V=\partial$.
Put $B:=\bigcap_{\alpha\in y}A_\alpha$.
If $m\geq{n}$ and $\alpha\in{y}$ then $\rho_m\in{A_\alpha}$, hence $\langle \rho_m\mid m\geq{n}\rangle\subseteq{B}$.
By Mathias criterion for genericity (see \cite{Mathias}) one concludes that $B\in\mathscr{U}$. This contradicts our initial assumption that ${\rm Gal}(\mathscr{U},\partial,\lambda)$ fails, as  witnessed by the family $\langle A_\alpha\mid \alpha<\lambda\rangle$. Therefore, the proof is accomplished.
\end{proof}

\section{Open problems}\label{SectionOpenProblems}

In this last section we collect some relevant open questions. 

\subsection{The failure at successors of singulars} 
The first two problems we present --and, perhaps, the most interesting ones-- concern the global failure of Galvin's property and the \textsf{ZFC} status of its strong negation. Specifically,


\begin{question}
\label{qglobalopen} Is it consistent with $\textsf{ZFC}$ that $\Gal{\kappa}{\kappa^{+}}$ simultaneously fails for every regular cardinal $\kappa\geq \aleph_1$?
\end{question}

\begin{question}
\label{qstrongneg} Suppose that $\kappa$ is a strong limit singular cardinal.
Is the strong failure of Galvin's property (i.e., $\neg_{\mathrm{st}}\mathrm{Gal}(\mathscr{D}_{\kappa^+},\kappa^+,\kappa^{++})$) consistent with $\textsf{ZFC}$? Alternatively, is $\neg_{\mathrm{st}}\mathrm{Gal}(\mathscr{D}_{\kappa^+},\kappa^+,\kappa^{++})$ impossible in $\textsf{ZFC}$?
\end{question}

By virtue of Galvin's theorem, a positive answer to Question~\ref{qglobalopen} would require the global failure of the GCH.\footnote{The first model of $\textsf{ZFC}$ where GCH fails everywhere was constructed by Foreman and Woodin in \cite{ForWoo}. As in our case, the authors rely on the assumption that there is  a supercompact cardinal.} Actually, even the more modest configuration  involving successor cardinals seems hard to obtain. Recall that we fail to produce this scenario due to the rather \emph{large gaps} mentioned in our (interleaved) Lévy collapses (see page~\pageref{calculation}). 
Note, however, that these gaps were essential for the further construction of guiding generics in Lemma~\ref{LemmaConstructingPair}. Generally speaking, there is tension between the behavior of the power-set function and the construction of guiding generics and, as a result, between the former and the gaps left by the (interleaved) Lévy collapses. Apart from this, in \S\ref{SectionImpossibility} we demonstrated that a \emph{Ultimate failure of Galvin's property} is impossible. In particular, a variation of the methods by Abraham-Shelah is required if one aims to answer Question~\ref{qglobalopen} in the affirmative. Anyhow, there is still some chance to obtain the consistency of $\neg \mathrm{Gal}(\mathscr{D}_\kappa,\kappa,\kappa^+)$ for all regular cardinal $\kappa$, provided $2^\kappa>\kappa^{+}$. This is a stimulating challenge. 

\smallskip

A natural attempt to answer Question~\ref{qstrongneg} in the nega\-tive would be to begin with a model where $\neg_{\mathrm{st}}\mathrm{Gal}(\mathscr{D}_{\kappa^+},\kappa^+,\kappa^{++})$ holds at a regular cardinal $\kappa$ and later show that Prirky forcing preserves the corres\-ponding witness: this attempt has been pursued  unsatisfactorily in \S\ref{SectionTheStrongFailure} (see Theorem~\ref{thmprikrydestroys}). Likewise, note that involving other classical Prikry type forcing --such as, Prikry forcing with non normal ultrafilters or Extender-based Prikry forcing-- seems to lead to similar complications. The reason for this being that all of these posets project onto the Prikry forcing, and hence none of them can preserve general witnesses. A yet further alternative to address this problem is based on the next observation concerning \emph{density of old sets}: 
\begin{proposition}
\label{clmpreserve} Let $\kappa$ be an infinite cardinal and $\mathbb{P}$ a notion of forcing that preserves $\kappa^+$. Let $G\subseteq\mathbb{P}$ be $V$-generic.
Assume that every set $x\in{V[G]}$ with $|x|^{V[G]}=\kappa^+$ contains a set $y\in{V}$ of with $|y|^V=\kappa^+$. Then,
$$V\models \neg_{\mathrm{st}}\mathrm{Gal}(\mathscr{D}_{\kappa^+},\kappa^+,\lambda)\;\Longrightarrow\;V[G]\models \neg_{\mathrm{st}}\mathrm{Gal}(\mathscr{D}_{\kappa^+},\kappa^+,\lambda).$$
\end{proposition}
\begin{proof}
Fix a collection $\mathcal{C}=\langle C_\alpha\mid \alpha<\lambda\rangle$  exemplifying $\neg_{\mathrm{st}}\mathrm{Gal}(\mathscr{D}_{\kappa^+},\kappa^+,\lambda)$ in $V$. 
Let $x\in \mathcal{P}(\lambda)$ with $|x|^{V[G]}= \kappa^+$.
By our assumptions, there is a set $y\subseteq{x}$, $y\in{V}$ such that $|y|^V=\kappa^+$.
Now $|\bigcap_{\alpha\in x}C_\alpha|\leq|\bigcap_{\alpha\in y}C_\alpha|<\kappa$. So, the collection $\mathcal{C}$
witnesses $\neg_{\mathrm{st}}\mathrm{Gal}(\mathscr{D}_{\kappa^+},\kappa^+,\lambda)$ in $V[G]$, as required.
\end{proof}
If the poset $\mathbb{P}$ is  Prikry forcing $\mathbb{P}(\mathscr{U})$ with respect to some normal ultrafilter $\mathscr{U}$ over $\kappa$ then the former density requirement is equivalent to $\mathrm{Gal}(\mathscr{U},\kappa^+,\kappa^+)$ (Theorem~\ref{propdensity}). Hence, a natural attempt to answer Question~\ref{qstrongneg} will be to produce a model where $\mathscr{U}$ is a normal ultrafilter over $\kappa$ and both $\mathrm{Gal}(\mathscr{U},\kappa^+,\kappa^+)$ and $\neg_{\mathrm{st}}\mathrm{Gal}(\mathscr{D}_{\kappa^+},\kappa^+,\kappa^{++})$ hold. Unfortunately, as the following proposition shows, this strategy is doomed to failure:



\begin{proposition} \label{clmgenstfl} Let $\mathscr{F}$ be a $\kappa$-complete filter over $\kappa$ which extends $\mathscr{D}_\kappa$. \linebreak If ${\rm Gal}(\mathscr{F},\kappa^+,\kappa^{++})$ holds 
then $\neg_{\mathrm{st}}\mathrm{Gal}(\mathscr{D}_{\kappa^+},\kappa^+,\kappa^{++})$ fails. In particular, $\mathrm{Gal}(\mathscr{F},\kappa^+,\kappa^+)$ entails the failure of $\neg_{\mathrm{st}}\mathrm{Gal}(\mathscr{D}_{\kappa^+},\kappa^+,\kappa^{++})$.
\end{proposition}
\begin{proof}
Let $\langle C_\alpha\mid \alpha<\kappa^{++}\rangle\s \mathscr{D}_{\kappa^+}$. As usual, we may assume without loss of generality  that $C_\alpha\cap \delta$ is a club on $\delta$, for some $\delta\in S^{\kappa^+}_\kappa$ and every $\alpha<\kappa^{++}$.

Let $E:=\langle \gamma_\beta\mid\beta<\kappa\rangle$ be the increasing enumeration of some club at $\delta$ in order-type $\kappa$.
By our choice on $\delta$, the set $E\cap{C_\alpha}$ is a club at $\delta$ for every $\alpha<\kappa^{++}$. Let $D_\alpha$ be a club at $\kappa$ for which $E\cap{C_\alpha}=\{\gamma_\beta\mid \beta\in{D_\alpha}\}$.
Notice that $D_\alpha\in\mathscr{F}$ for every $\alpha<\kappa^{++}$, since $\mathscr{F}$  extends the club filter $\mathscr{D}_\kappa$.

Applying ${\rm Gal}(\mathscr{F},\kappa^+,\kappa^{++})$ to the collection $\langle D_\alpha\mid \alpha<\kappa^{++}\rangle$ there exists $I\in [\kappa^{++}]^{\kappa^+}$ such that $B=\bigcap_{\alpha\in I}D_\alpha$.
Lifting back to $\delta$, the set defined by $C:=\{\gamma_\beta\mid \beta\in{B}\}$ is a subset of $C_\alpha$ for every $\alpha\in{I}$, and since $|C|=\kappa$ we conclude that $\neg_{\mathrm{st}}\mathrm{Gal}(\mathscr{D}_{\kappa^+},\kappa^+,\kappa^{++})$ cannot hold. 
\end{proof}

  A more modest attempt towards answering Question~\ref{qstrongneg} would be to prove the consistency of $\neg_{\mathrm{st}}\mathrm{Gal}(\mathscr{D}_{\kappa^+},\kappa^{++},\kappa^{++})$ for a singular cardinal $\kappa$:
  \begin{question}\label{qstrongneg2}
  Is the statement $\neg_{\mathrm{st}}\mathrm{Gal}(\mathscr{D}_{\kappa^+},\kappa^{++},\kappa^{++})$ consistent with $\textsf{ZFC}$ for a strong limit singular $\kappa$?
  \end{question}
Concerning the previous question, Corollary~\ref{CorStrongNegationSuper} already showed that in suitable Magidor/Radin generic extensions  $\neg_{\mathrm{st}}\mathrm{Gal}(\mathscr{D}_{\kappa^+},\kappa^{++},\kappa^{++})$ fails for a strong limit singular  cardinal $\kappa$. 
Besides of that, there is a worth mentioning connection between Question~\ref{qstrongneg2} and forcing axioms. Indeed, it was shown in \cite[Theorem~2.8]{MR3604115} that the \emph{Proper Forcing Axiom} (PFA) yields  $\neg_{\mathrm{st}}\mathrm{Gal}(\mathscr{D}_{\aleph_1},\aleph_2,\aleph_2)$. In light of this an appealing avenue of research would be to search for \emph{higher analogues} of PFA yielding $\neg_{\mathrm{st}}\mathrm{Gal}(\mathscr{D}_{{\kappa^+}},\kappa^{++},\kappa^{++})$ for a regular cardinal $\kappa\geq \aleph_1$. Similarly, the same wish extends to singular cardinals, although in this context one will encounter a shortage of parallels of PFA. 
A potential strategy to overcome this problem might bear on the abstract iteration scheme for singular cardinals introduced in the \emph{$\Sigma$-Prikry project} by Rinot, Sinapova and the third author \cite{PartI,PartII,PartIII}.   

\smallskip

In parallel to the above discussion,  it might be illuminating to separate strong limit cardinals from non strong limit ones. 
In the former case, the negation of  Galvin's property requires the use of Prikry-type forcings, and these latter seem to eliminate witnesses for the strong failure (Theorem~\ref{thmprikrydestroys}).
However, if one drops the strong limitude assumption on $\kappa$ then, perhaps, the principle $\neg_{\mathrm{st}}\mathrm{Gal}(\mathscr{D}_{\kappa^+},\kappa^+,\kappa^{++})$ is  obtainable by adding many Cohen subsets to some cardinal $\theta<\kappa$. This leads to the following problem:

\begin{question} 
\label{qnonslimit} Is it consistent with $\textsf{ZFC}$ that $\kappa$ is a non-strong limit singular cardinal and $\neg_{\mathrm{st}}\mathrm{Gal}(\mathscr{D}_{\kappa^+},\kappa^+,\kappa^{++})$ holds?
\end{question}

\subsection{Galvin's property and large cardinals}

Aside of successors of singular cardinals there is a further case that it is not covered by  \cite{MR830084}: namely, the case where $\kappa$ is a weakly (not strongly) inaccessible cardinal. 
\begin{question}
\label{qweaklyinac} Assume $\kappa$ is a weakly (but not strongly) inaccessible cardinal. Is the principle $\neg\Gal{\kappa }{\kappa^+}$ consistent with $\textsf{ZFC}$?
\end{question}
The above is \cite[Question 2.11]{MR3604115}, and we believe that a positive answer is plausible.
In order to prove the consistency of $\Gal{\kappa}{\kappa^+}$ at such cardinals it is necessary that the sequence $\langle 2^\theta\mid \theta<\kappa\rangle$ stabilizes, as proved in \cite[Corollary~2.10]{MR3604115}. One has then to force $2^\theta>\kappa$ (for some $\theta<\kappa$) in order to prepare the ground for a potential failure of $\Gal{\kappa}{\kappa^+}$.

\smallskip

For most large cardinals $\kappa$, Galvin's theorem  entails  ${\rm Gal}(\mathscr{F},\kappa,\kappa^+)$ for all normal filter $\mathscr{F}$ over $\kappa$.  A paradigmatic example of this are measurable cardinals. It would be interesting to regard this question in a context where $\kappa$ is (in some sense) close to be measurable but not necessarily strongly inaccessible: this is the case of real-valued measurable cardinals. 
In this regard, it is known that if $\kappa=2^{\aleph_0}$ is real-valued measurable then $\kappa=\kappa^{<\kappa}$ and so ${\rm Gal}(\mathscr{F},\kappa,\kappa^+)$ holds for all normal filter $\mathscr{F}$ over $\kappa$. Nevertheless, one can force real-valued measurable cardinals to be strictly smaller than the continuum: for instance, by starting with two measurable cardinals $\kappa<\lambda$ and adding $\lambda$-many random reals to $\lambda$ \cite{Sol}. All in all, the above suggests the following interesting question:
\begin{question}  
\label{qrvm} Is it consistent with $\textsf{ZFC}$ that $\kappa$ is a real-valued measurable cardinal and $\mathrm{Gal}(\mathscr{D}_{\kappa},\kappa,\kappa^{+})$ fails (in which case $\kappa<2^{\aleph_0}$)?
\end{question}
A positive answer to the above would yield another in the affirmative to Question~\ref{qweaklyinac}, for every real-valued measurable is weakly inaccessible.


\smallskip

We proved in Theorem~\ref{thmnegmeasurable} that $\neg{\rm Gal}(\mathscr{U},\kappa^+,\lambda)$ is consistent with $\textsf{ZFC}$, provided $\mathscr{U}$ is a normal ultrafilter over $\kappa$ and $\lambda>\kappa$.
In this respect, note that the ultrafilterhood of $\mathscr{U}$ played an important role within the proof of the auxiliary Claim~\ref{Claimpolarized}.
This fact suggests the following interesting problem:
\begin{question}
\label{qinaccm} For $\kappa<\lambda$ cardinals with $\kappa$ measurable, 
does $\textsf{ZFC}$ prove the principle ${\rm Gal}(\mathscr{D}_\kappa,\kappa^+,\lambda)$? Similarly, what is the $\textsf{ZFC}$ status of  ${\rm Gal}(\mathscr{F},\kappa^+,\lambda)$ for  normal filters $\mathscr{F}$ over $\kappa$?
\end{question} 
A similar question arises with respect to small large cardinals.
On the one hand, if $\neg{\rm Gal}(\mathscr{U},\kappa^+,\lambda)$ is consistent at very large cardinals then one expects the same consistency result at small large cardinals.
On the other hand, small large cardinals are not populated by normal ultrafilters. So,

\begin{question}
\label{qinacc} Can one prove in \textsf{ZFC} that ${\rm Gal}(\mathscr{F},\kappa^+,\lambda)$ fails for some $\kappa$-complete filters $\mathscr{F}$ over $\kappa$?
\end{question}

 Back to measurable cardinals, in Theorem~\ref{thmnegmeasurable} we forced $\neg{\rm Gal}(\mathscr{U},\kappa^+,\lambda)$, hence $\neg{\rm Gal}(\mathscr{U},\partial,\lambda)$ is consistent for all $\partial,\lambda>\kappa$. If one wishes to force the opposite relation then the most inviting case is the one where both parameters $\partial$ and $\lambda$ are equal and their cardinalities are  maximal possible: 
\begin{question}
\label{qfull} Let $\kappa$ be measurable.
Is it consistent that ${\rm Gal}(\mathscr{U},2^\kappa,2^\kappa)$ holds for every  $\kappa$-complete ultrafilter $\mathscr{U}$ over $\kappa$?
\end{question} 
Let us consider the concrete example of $\mathscr{D}_\kappa$. By \cite[Theorem 2.8]{MR3604115} if $2^\kappa=\kappa^+$ then $\neg{\rm Gal}(\mathscr{D}_\kappa,\kappa^+,\kappa^+)$.
However, if $2^\kappa>\kappa^+$ then Corollary~\ref{propfullgalvin} can be used to infer the consistency of $2^\kappa=2^{\kappa^+}$ with ${\rm Gal}(\mathscr{D}_\kappa,2^\kappa,2^\kappa)$.
\smallskip

\subsection{Consistency strength of the failure of Galvin's property}
A di\-fferent source of questions is
 the consistency strength of our results.
Consider the simplest case of one strong limit singular cardinal $\kappa$ and the failure of the Galvin property at $\kappa^+$.
We know that one has to violate \textsf{SCH} in order to obtain this configuration, hence large cardinals are essential. Recall that 
we started from a supercompact cardinal in order to have a convenient forcing-indestructibility upon $\kappa$.
It seems to the authors that much less might be needed. In effect, note that  we just need indestructibility under $\mathrm{Add}(\kappa,\lambda)$ (e.g., Theorem~\ref{thmnegmeasurable}) and $\mathbb{S}(\kappa,\lambda)$ (e.g., Theorem~\ref{thmglobal}), being this la\-tter for\-cing --essentially-- a clever combination of adding Cohen subsets to $\kappa$ and clubs at $\kappa^+$ (cf. Definition~\ref{AbSh forcing}). Therefore, it is reasonable to expect
that forcing-indestructibility under $\mathbb{S}(\kappa,\lambda)$ would be  available under weaker large cardinal assumptions. All in all, this leads to the following  problem:

\begin{question}
\label{qconstrength} What is the consistency strength of $\neg\Gal{\kappa^+}{\kappa^{++}}$ for  a strong limit singular $\kappa$? What about $\neg_{\mathrm{st}}\Gal{\kappa^+}{\kappa^{++}}$?
\end{question}

We believe the methods of \cite{MR3928383} will be helpful in this context.

\section*{Acknowledgements}
The authors wish to thank Moti Gitik, Yair Hayut and Menachem Magidor for many clever comments and helpful suggestions.

\bibliographystyle{alpha}
\bibliography{arlist}

\end{document}